\documentclass{article}
\usepackage{xcolor}
\usepackage[a4paper,top=2cm,bottom=2cm,left=2cm,right=2cm]{geometry}

\usepackage{amsmath}
\usepackage{amssymb}
\usepackage{amsthm}
\usepackage{mathrsfs}
\usepackage{tikz-cd}
\usepackage{stmaryrd}
\usepackage{hyperref}
\usepackage{authblk}
\usepackage{t-angles}
\usepackage[all]{xy}
\def\pulb{\ar@{}[dr]|(0.2){\mbox{\Large{$\lrcorner$}}}}

\usepackage{comment}

\usepackage{mathtools}

\usepackage[T1]{fontenc}
\usepackage[utf8]{inputenc}
\usepackage{lmodern}

\usepackage[toc,page]{appendix}

\usepackage[bbgreekl]{mathbbol}
\usepackage{bbm}

\newtheorem{thm}{Theorem}[section]
\newtheorem{theorem}[thm]{Theorem}

\newtheorem{corollary}[thm]{Corollary}
\newtheorem{lemma}[thm]{Lemma}
\newtheorem{proposition}[thm]{Proposition}

\theoremstyle{definition}
\newtheorem{definition}[thm]{Definition}
\newtheorem{example}[thm]{Example}

\newtheorem{observation}[thm]{Observation}

\newtheorem{remark}[thm]{Remark}

\newcommand{\cO}{{\mathcal{O}}}
\newcommand{\rSL}{{\mathrm{SL}}}
\newcommand{\C}{{\mathbb{C}}}
\newcommand{\bP}{{\mathbb{P}}}
\newcommand{\lra}{{\longrightarrow}}
\newcommand{\ra}{{\rightarrow}}
\newcommand{\de}{{\delta}}
\newcommand{\Chi}{{\mathfrak{X}}}
\newcommand{\fg}{{\mathfrak{g}}}
\newcommand{\fgp}{{\mathrm{fgp}}}
\newcommand{\cF}{{\mathcal{F}}}
\newcommand{\beq}{\begin{equation}}
	\newcommand{\eeq}{\end{equation}}

\title{{\bf Noncommutative vector field calculi}}
\author{{\large Rita Fioresi${}^{1,}$}\footnote{rita.fioresi@unibo.it}~}
\author{{\large Emanuele Latini${}^{1,}$}\footnote{emanuele.latini@unibo.it}~}
\author{{\large Thomas Weber${}^{2,}$}\footnote{thomas.weber@matfyz.cuni.cz}}
\affil{
	\centerline{\sl
		${}^1${ Alma Mater Studiorum - Università di Bologna}}
	
	\centerline{\sl  Via Zamboni 33, 40126 Bologna, Italy}
	
	\centerline{and}
	
	\centerline{\sl
		{Istituto Nazionale di Fisica Nucleare}}
	
	\centerline{\sl Sezione di Bologna, I-40126 Bologna, Italy}
	\medskip
	\centerline{\sl
		${}^2${ Mathematical Institute of Charles University}}
	
	\centerline{\sl  Sokolovsk\'a 49/83, 186 75 Prague 8, Czech Republic}
}
\date{\today}

\begin{document}
	
	\maketitle
	
	\begin{abstract}
		We discuss noncommutative differential geometry from a vector field centric point of view. This is based on the notion of first order vector field calculus (FOVC), which has been previously introduced by Borowiec under the name Cartan pair. We define the universal FOVC and construct an adjunction between the categories of FOVC and that of first order differential calculi, showing that the vector field approach is dual, though not equivalent, to the differential form one. This correspondence is then extended to covariant vector field and differential calculi. On Hopf algebras, (bi)covariant FOVC are in bijection with (bicovariant) quantum tangent spaces. For Hopf--Galois extensions, quantum tangent spaces give rise to vertical vector fields and in this setup we further describe base vector fields and horizontal vector fields and show that they are related via a noncommutative Atiyah sequence. Multiple examples, based on braided derivations, finite groups and the quantum Hopf fibration, are given. The vector field approach is further enriched by a sheaf-theoretic treatment, which recovers the former as a local, or affine, picture.
	\end{abstract}
	
	\tableofcontents
	
	\section{Introduction}
	
	Quantum differential geometry is an active area of research, initiated a few years
	after the quantum groups discovery \cite{frt}, with the pioneering works \cite{BrzMaj, SchDC, Woronowicz1989}
	(see also \cite{bm} and references therein for an extensive account on this subject).
	The naturality of K\"ahler differentials in algebraic and differential geometry \cite{ha} makes
	the approach via differential forms more feasible in the noncommutative setting, however more distant to certain
	geometrical constructions, which come intrinsically via {\sl vector fields} and
	differential operators.
	In fact, the majority of the literature on noncommutative differential geometry utilizes {\sl first order differential calculi} (FODC).
	This language proved very fruitful, in particular, the striking feature of non-uniqueness of FODC on quantum groups led to the classification of such calculi \cite{Woronowicz1989}. 
	The case of homogeneous spaces, quotients of classical groups, was first examined in the quantum setting in \cite{herm} and then in the Heckenberger--Kolb 
	construction \cite{HeckKolb, HeckKolb2} of coinvariant differential calculi on irreducible quantum flag manifolds.
	Indeed, the coinvariance condition, most natural when dealing with groups and their homogeneous spaces, allows to reduce the plethora of FODC, focusing on the ones that are most interesting for their invariant properties. The next development concerned the successful generalization of differential structures on principal bundles to the quantum setting, leading to vertical, horizontal and base forms, related via a noncommutative version of the Atiyah sequence \cite{BrzMaj},
	whose importance also become paramount in the quantum setting, when discussing associated bundles (see \cite{baum}).
	In fact, there was considerable work in the past decades, examining quantum symmetric spaces and their FODC in relation to the Atiyah sequence, see e.g. \cite{bm,ROB2} and references therein.
	
	\medskip
	
	On the other hand, noncommutative vector fields and their Lie derivatives have been introduced in \cite{Bor96, Bor97} under the name Cartan Pairs, in duality to the FODC approach. This notion was then utilized and further developed e.g. in \cite{BMvector,bm,Bor23,fmw} but is still not universally known in the community. Interestingly, noncommutative vector fields appear rather frequently in form of elements of quantum tangent spaces \cite{paolo, AschieriSchupp,HeckKolb,HeckKolb2,Woronowicz1989}. This usually happens when dual Hopf algebras, like quantum group coordinate algebras $\mathcal{O}_q(G)$ and Drinfel'd--Jimbo quantum groups $U_q(\mathfrak{g})$, are present. In these works, quantum tangent spaces are usually seen as a tool to construct differential calculi, while in \cite{paolo, AschieriSchupp} the quantum tangent space approach is taken as a serious alternative, though equivalent, point of view in duality to (bi)covariant first order differential calculi.
	
	\medskip
	In the first part of the present paper we follow the vector field centric point of view advocated in \cite{paolo, AschieriSchupp} and propose an approach to noncommutative differential geometry through the notion of noncommutative vector fields and Cartan pairs \cite{Bor96, Bor97}. Our main goal is to reveal the utility of vector fields as a dual, though not equivalent, path to noncommutative differential geometry. We clarify the relation of noncommutative vector fields and differential forms in terms of a categorical adjunction. Considerable focus is placed on examples, particularly those arising from quantum tangent spaces. We further introduce vertical, horizontal and base vector fields and relate them in terms of a noncommutative Atiyah sequence. 
	
	\medskip
	Most of the research in noncommutative differential geometry concerns the affine case, where both the quantum principal bundle and its base are affine geometrical objects, hence fully captured by their algebra of function,
	which is the object that becomes quantized. However, a true geometric approach should also
	comprehend the projective, scheme-theoretic setting, to fully account the situation where
	we do not have an affine variety, but perhaps an open inside it or a projective scheme \cite{artin, vov, ro}.
	The sheaf-theoretic approach to quantum bundles expressed in \cite{pflaum}, together with
	the thorough scheme-theoretic treatment of noncommutative geometry, was fully implemented
	first in \cite{afl,FLP} for the construction of quantum principal bundles, and their reduction, over non affine bases and then in \cite{aflw} to produce a natural FODC for the example of quotients of classical quantum groups.
	
	\medskip
	In the final part of the present paper we proceed implementing a sheaf-theoretic approach 
	to noncommutative vector fields calculi,
	where, in the spirit of \cite{afl}, modules and algebras only represent the local models. This is done in duality to the sheaf-theoretic FODC approach developed in \cite{aflw}. Several notions of the first part of the present paper, like vertical vector fields, base vector fields and the Atiyah sequence, are then generalized to the realm of sheaves.
	
	\medskip
	The paper is organized as follows.
	
	\medskip
	In Section \ref{sec:FOVC} we give the basic definition of first order vector field calculus (FOVC),
	proving the existence of a universal FOVC and examining 
	some interesting examples, including braided derivations on Hopf--Galois extension with respect to the \DJ ur\dj evi\'c braiding.
	We then discuss a natural duality between FOVC and first order differential calculus (FODC), leading to an adjunction between the categories of FOVC and FODC.
	We further introduce covariant FOVC on comodule algebras, where the Lie derivative is assumed to have an additional compatibility with the underlying Hopf algebra coactions. The resulting first order vector field calculi admit adjoint functors from and to the category of covariant FODC.

	\medskip
	In Section \ref{sec-ex} we recall the notion of (bicovariant) quantum tangent space and prove that they correspond to (bi)covariant FOVC on Hopf algebras. This correspondence is then generalized to Hopf--Galois extensions, leading to the notion of vertical vector fields, which turns out to be in duality with a construction given by \DJ ur\dj evi\'c in \cite{DurI, DurII}.
	We then introduce base vector fields and the Atiyah sequence, as defined originally in \cite{atiyah} for vector fields, though
	adapted to the noncommutative geometric setting, also leading to the notion of horizontal	vector fields. We further discuss key examples, such as quantum principal bundles on the noncommutative Hopf fibration.
	
	\medskip
	In Section \ref{sec:sheaf} we take the sheaf-theoretic point of view on the theory defined in the previous sections and introduce FOVC for quantum ringed spaces. A particular emphasis is set on quantum principal bundles, vertical vector fields and the Atiyah sequence, giving a generalization of the previous theory, which was limited to the affine setting. We provide explicit examples based on the ringed space $\mathbb{P}^1(\mathbb{C})$ and the quantum ringed spaces obtained by localizations of $\mathcal{O}_q(\mathrm{SL}_2(\mathbb{C}))$ and $\mathcal{O}_q(\mathrm{GL}_2(\mathbb{C}))$, respectively.
	
	\bigskip
	{\bf Acknowledgements.} We thank Paolo Aschieri and Alessandro Ardizzoni for helpful discussions.
	This research was supported by Gnsaga-Indam, by
	COST Action CaLISTA CA21109, HORIZON-MSCA-2022-SE-01-01 CaLIGOLA, MSCA-DN CaLiForNIA - 101119552,
	PNRR MNESYS, PNRR National Center for HPC, Big Data and Quantum Computing, PNRR SimQuSec, INFN Sezione Bologna,
	Gast Initiative.  
	T.W. is supported by the GA\v{C}R PIF 24-11324I.

	\subsubsection*{Notation}
	
	Throughout this paper we fix a field $\Bbbk$. Vector spaces, algebras, etc. are understood over $\Bbbk$ and we denote the tensor product of $\Bbbk$-vector spaces by $\otimes$. Given a Hopf algebra $H$ we denote its comultiplication and counit by $\Delta\colon H\to H\otimes H$ and $\varepsilon\colon H\to\Bbbk$, respectively, while the antipode is denoted by $S\colon H\to H$. We frequently employ Sweedler's notation for the coproduct $\Delta(h)=:h_1\otimes h_2$ of an element $h\in H$, where we omit summation indices.
	
	\section{First order vector field calculi}\label{sec:FOVC}
	
	In the first section we recall the notion of {\sl first order vector field calculus} on arbitrary associative unital algebras, which are possibly noncommutative. We discuss some simple, but important, examples, also introducing the universal first order vector field calculus in Section \ref{sec:univ}. We compare the notion of first order vector field calculus with {\sl first order differential calculi}, in terms of an adjunction of categories, see Section \ref{sec:FODC}. Then, in Section \ref{sec:covFOVC}, we continue by introducing \emph{covariant} first order vector field calculi as vector field calculi on comodule algebras which are compatible with the Hopf algebra coactions. We discuss examples of the latter and the corresponding duality with covariant first order differential calculi. In a last section, we construct an example of first order vector field calculus on arbitrary Hopf--Galois extensions, based on braided derivations with respect to the \DJ ur\dj evi\'c braiding.
	
	\subsection{Basic definition}
	
	The notion of {\sl first order vector field calculus} was first introduced in \cite{Bor96,Bor97} under the terminology of {\sl Cartan pair}. We refer to the latter as first order vector field calculi in order to emphasize their duality with first order differential calculi. This duality will be made precise in terms of a categorical adjunction later on in the article.
	
	Let $A$ be an associative unital algebra. We denote the $\Bbbk$-module of $\Bbbk$-linear endomorphisms $A\to A$ by $\mathrm{End}_\Bbbk(A)$. Note that $\mathrm{End}_\Bbbk(A)$ becomes an $A$-bimodule via
	\begin{equation*}
		(a\cdot\phi\cdot b)(c):=a\phi(bc)
	\end{equation*}
	for all $\phi\in\mathrm{End}_\Bbbk(A)$ and $a,b,c\in A$.
	
	\begin{definition}\label{def:FOVC} 
		A {pair} $(\mathfrak{X},\mathscr{L})$ is called a \textit{first order vector field calculus} (FOVC) on $A$ if
		\begin{enumerate}
			\item[i.)] $\mathfrak{X}$ is an $A$-bimodule.
			
			\item[ii.)] $\mathscr{L}\colon\mathfrak{X}\to\mathrm{End}_\Bbbk(A)$ is a $\Bbbk$-linear map such that for all $X\in\mathfrak{X}$ the \textit{Leibniz rule}
			\begin{equation}
				\mathscr{L}_X(ab)=\mathscr{L}_X(a)b+\mathscr{L}_{X\cdot a}(b)
			\end{equation}
			holds for all $a,b\in A$.
			
			\item[iii.)] $\mathscr{L}$ is injective and left $A$-linear, i.e. $\mathscr{L}_{a\cdot X}(b)=a\mathscr{L}_X(b)$ for all $X\in\mathfrak{X}$ and $a,b\in A$.
		\end{enumerate}
		For a FOVC $(\mathfrak{X},\mathscr{L})$ we call $\mathscr{L}\colon\mathfrak{X}\to\mathrm{End}_\Bbbk(A)$ the \textit{Lie derivative}.
	\end{definition}
	
	A \textit{morphism} of first order vector calculi (FOVCi) $(\mathfrak{X},\mathscr{L})$, $(\mathfrak{X'},\mathscr{L'})$ 
	on the same algebra $A$ is
	an $A$-bimodule map $\Phi\colon\mathfrak{X} \longrightarrow \mathfrak{X'}$, such
	that
	\begin{equation}\label{diagr-vc}
		\xymatrix{ \mathfrak{X}\ar[rr]^{\Phi}
			\ar[drr]_{\mathscr{L}} & & \mathfrak{X'}\ar[d]^{\mathscr{L'}} \\
			& & \mathrm{End}_\Bbbk(A)
		}
	\end{equation}
	commutes. We sometimes write $(\mathfrak{X},\mathscr{L})\to(\mathfrak{X'},\mathscr{L'})$ for a morphism of FOVCi. The category of FOVCi over a fixed algebra $A$ is denoted by $\underline{\mathrm{FOVC}}_A$. By the very definition, a morphism of FOVCi is necessarily injective.
	
	\begin{example}\label{ex:FOVC}
		We discuss basic examples of FOVCi.
		\begin{enumerate}
			\item[i.)] If $A$ is a 
			{commutative} algebra, the derivations of $A$
			$$
			\mathfrak{X}:=\mathrm{Der}(A):=\{X\in\mathrm{End}_\Bbbk(A)~|~X(ab)=X(a)b+aX(b)\text{ for all }a,b\in A\}
			$$
			form a FOVC on $A$ with Lie derivative given by the canonical injection $\mathscr{L}\colon\mathrm{Der}(A)\hookrightarrow\mathrm{End}_\Bbbk(A)$. The $A$-bimodule structure on $\mathfrak{X}$ is determined by 
			$$
			(a\cdot X)(b):=aX(b)=:(X\cdot a)(b)
			$$
			for all $X\in\mathfrak{X}$ and $a,b\in A$. The Leibniz rule is satisfied by definition.
			
			This is the canonical vector field calculus arising from (classical) differential geometry: given a smooth manifold $M$, the algebra of smooth functions $A=\mathscr{C}^\infty(M)$ endowed with the pointwise product is commutative and $\mathfrak{X}=\mathrm{Der}(A)$ is isomorphic to smooth sections $\Gamma^\infty(TM)$ of the tangent bundle $TM$.
			
			\item[ii.)] 
			Let $(H,\mathcal{R})$ be a \textit{quasitriangular Hopf algebra} \cite{Drin}. We recall that the latter is a Hopf algebra $H$ endowed with a \textit{universal $\mathcal{R}$-matrix}, i.e., an invertible element $\mathcal{R}\in H\otimes H$, which makes $H$ quasi-cocommutative, namely $\Delta^\mathrm{op}(\cdot)=\mathcal{R}\Delta(\cdot)\mathcal{R}^{-1}$, and which satisfies the hexagon equations
			\begin{equation}
				(\mathrm{id}\otimes\Delta)(\mathcal{R})=\mathcal{R}_{13}\mathcal{R}_{12},\qquad
				(\Delta\otimes\mathrm{id})(\mathcal{R})=\mathcal{R}_{13}\mathcal{R}_{23},
			\end{equation}
			where we used the common \textit{leg notation} $\mathcal{R}_{12}:=\mathcal{R}\otimes 1$, $\mathcal{R}_{23}:=1\otimes\mathcal{R}$, $\mathcal{R}_{13}:=(\mathrm{flip}_{H,H}\otimes\mathrm{id})(\mathcal{R}_{23})$ for the above elements in $H\otimes H\otimes H$. Then, a \textit{braided-commutative left $H$-module algebra} is an associative unital algebra $(A,m_A,1_A)$, which is endowed with a left $H$-module action $\rhd\colon H\otimes A\to A$, such that for all $h\in H$ and $a,b\in A$ we have 
			\begin{equation}
				h\rhd(ab)=(h_1\rhd a)(h_2\rhd b),\qquad
				h\rhd 1_A=\varepsilon(h)1_A,\qquad
				ba=(\mathcal{R}_i\rhd a)(\mathcal{R}^i\rhd b),
			\end{equation}
			where $\mathcal{R}=\mathcal{R}^i\otimes\mathcal{R}_i\in H\otimes H$ (finite sum over repeated indices understood). The above conditions are equivalent to say that $A$ is a commutative algebra in the braided monoidal category of left $H$-modules, see \cite[Theorem 9.2.4]{MajFound} for more details.
			Given a quasitriangular Hopf algebra $(H,\mathcal{R})$ and a braided-commutative left $H$-module algebra $A$ the \textit{braided derivations}
			$$
			\mathfrak{X}:=\mathrm{Der}_\mathcal{R}(A):=\{X\in\mathrm{End}_\Bbbk(A)~|~X(ab)=X(a)b+(\mathcal{R}_i\rhd a)(\mathcal{R}^i\rhd X)(b)\text{ for all }a,b\in A\}
			$$
			become an $A$-bimodule with respect to $(a\cdot X)(b):=aX(b)$ and
			$$
			(X\cdot a)(b):=(\mathcal{R}_i\rhd a)(\mathcal{R}^i\rhd X)(b)
			$$
			for all $X\in\mathrm{Der}_\mathcal{R}(A)$ and $a,b\in A$. Then $(\mathfrak{X},\mathscr{L})$ is a FOVC on $A$, again with $\mathscr{L}\colon\mathrm{Der}_\mathcal{R}(A)\hookrightarrow\mathrm{End}_\Bbbk(A)$ the canonical injection.
			This includes the examples of braided commutative algebras for \textit{triangular} Hopf algebras, in particular twisted commutative algebras and Drinfel'd twist star product algebras, as discussed in \cite{WeBC,WeThesis}.
			Braided derivations in the context of triangular Hopf algebras are further appearing in \cite{AschieriLandiPagani}.
			For a cocommutative Hopf algebra with $\mathcal{R}=1\otimes 1$ and a commutative left $H$-module algebra $A$ this recovers the previous example i.).
			Further note that any \textit{co}quasitriangular Hopf algebra $(H,\mathcal{R})$ can be seen as a braided-commutative left module algebra with respect to the quasitriangular Hopf algebra $H^\circ_\mathrm{op}\otimes H^\circ$, where $H^\circ$ denotes the Hopf dual and $H^\circ_\mathrm{op}$ is the finite dual with opposite multiplication. This is spelled out in \cite[Example 2.13]{paolo} for the cotriangular case and easily generalizes to the coquasitriangular setting.

			\item[iii.)] Let $(\mathfrak{X},\mathscr{L})$ be a FOVC on an algebra $A$ and $(\mathfrak{X}',\mathscr{L}')$ be a FOVC on another algebra $A'$. If $A$ and $A'$ are domains (i.e. there are no zero-divisors) then $(\mathfrak{X}_\otimes,\mathscr{L}^\otimes)$ with 
			$$
			\mathfrak{X}_\otimes:=\mathfrak{X}\otimes A'\oplus A\otimes\mathfrak{X}'
			$$
			the $A\otimes A'$-bimodule structure given by 
			$$
			(a\otimes a')\cdot(X\otimes b'+b\otimes X')\cdot(c\otimes c')
			:=a\cdot X\cdot c\otimes a'b'c'+abc\otimes a'\cdot X'\cdot c'
			$$
			and
			$$
			\mathscr{L}^\otimes_{X\otimes a'+a\otimes X'}(b\otimes b')
			:=\mathscr{L}_X(b)\otimes a'b'+ab\otimes\mathscr{L}'_{X'}(b')
			$$
			is a FOVC on the tensor product algebra $A\otimes A'$. To avoid the assumption of the algebras being domains one can quotient $\mathfrak{X}_\otimes$ with the $A\otimes A'$-sub-bimodule generated by the kernel of $\mathscr{L}^\otimes$.
		\end{enumerate}
	\end{example}

	\subsection{Universal 
		and inner FOVC}\label{sec:univ}
	On every associative unital algebra $A$ there is the \textit{universal FOVC} $(\mathfrak{X}_u,\mathscr{L}^u)$ defined as the endomorphisms which vanish at the unit
	\begin{equation}
		\mathfrak{X}_u:=\{X\in\mathrm{End}_\Bbbk(A)~|~X(1)=0\},
	\end{equation}
	together with the canonical injection $\mathscr{L}^u\colon\mathfrak{X}_u\hookrightarrow\mathrm{End}_\Bbbk(A)$. The $A$-bimodule structure on $\mathfrak{X}_u$ is defined by $(a\cdot X)(b):=aX(b)$ and
	\begin{equation}\label{UnivRightAction}
		(X\cdot a)(b):=X(ab)-X(a)b
	\end{equation}
	for all $X\in\mathfrak{X}_u$ and $a,b\in A$. Note that the operation \eqref{UnivRightAction} closes in $\mathfrak{X}_u$ since $(X\cdot a)(1)=X(a1)-X(a)1=0$. Let us verify that \eqref{UnivRightAction} is in fact a right $A$-action:
	\begin{align*}
		((X\cdot a)\cdot b)(c)
		&=(X\cdot a)(bc)-(X\cdot a)(b)c\\
		&=X(abc)-X(a)bc-X(ab)c+X(a)bc\\
		&=X(abc)-X(ab)c\\
		&=(X\cdot(ab))(c)
	\end{align*}
	holds for all $X\in\mathfrak{X}_u$ and $a,b,c\in A$. Moreover, the Leibniz rule
	$$
	\mathscr{L}^u_X(a)b+\mathscr{L}^u_{X\cdot a}(b)
	=X(a)b+(X\cdot a)(b)
	=X(ab)
	=\mathscr{L}^u_X(ab)
	$$
	holds by definition of the right $A$-action.
	
	We show that for every FOCV $(\mathfrak{X},\mathscr{L})$ on an associative unital algebra $A$ there is a morphism $(\mathfrak{X},\mathscr{L})\to(\mathfrak{X}_u,\mathscr{L}^u)$ of FOVCi to the universal FOVC.
	Since such a morphism is necessarily injective this shows that every FOVC can be understood as a sub-FOVC of the universal one, thus justifying the name a posteriori. This observation is implicit in \cite[Section 4]{Bor23}.
	
	\begin{proposition}\label{prop:univ}
		Let $A$ be an associative unital algebra.
		Then, the universal FOVC $(\mathfrak{X}_u,\mathscr{L}^u)$ is a FOVC on $A$ and it possesses the following universal property:
		for every FOVC $(\mathfrak{X},\mathscr{L})$ on $A$ there is a morphism 
		$$
		(\mathfrak{X},\mathscr{L})\to(\mathfrak{X}_u,\mathscr{L}^u)
		$$
		of FOVCi. Explicitly, there exists an injective morphism $\Phi\colon\mathfrak{X}\to\mathfrak{X}_u$ of $A$-bimodules such that $\mathscr{L}=\mathscr{L}^u\circ\Phi$.
	\end{proposition}
	\begin{proof}
		We already argued that $(\mathfrak{X}_u,\mathscr{L}^u)$ is a FOVC on $A$. To prove its universal property,
		let $(\mathfrak{X},\mathscr{L})$ be an arbitrary FOVC on $A$. By the Leibniz rule 
		$$
		\mathscr{L}_X(1)
		=\mathscr{L}_X(1\cdot 1)
		=\mathscr{L}_X(1)1+\mathscr{L}_{X\cdot 1}(1)
		=2\mathscr{L}_X(1)
		$$
		for all $X\in\mathfrak{X}$, which implies $\mathscr{L}_X(1)=0$ and thus $\mathscr{L}(\mathfrak{X})\subseteq\mathfrak{X}_u$. This means we can define $\Phi\colon\mathfrak{X}\to\mathfrak{X}_u$ as the corestriction of $\mathscr{L}\colon\mathfrak{X}\to\mathrm{End}_\Bbbk(A)$. By assumption, $\Phi$ is injective and left $A$-linear. It is also right $A$-linear, since
		$$
		(\Phi(X\cdot a))(b)
		=\mathscr{L}_{X\cdot a}(b)
		=\mathscr{L}_X(ab)-\mathscr{L}_X(a)b
		=\Phi(X)(ab)-(\Phi(X)(a))b
		=(\Phi(X)\cdot a)(b)
		$$
		by the Leibniz rule, where in the last equation we employed the right $A$-module action on $\mathfrak{X}_u$. Since the above holds for all $b\in A$ this implies $\Phi(X\cdot a)=\Phi(X)\cdot a$ for all $X\in\mathfrak{X}$ and $a\in A$. Thus, $\Phi$ is an injective morphism of $A$-bimodules. Since $\mathscr{L}^u\colon\mathfrak{X}^u\hookrightarrow\mathrm{End}_\Bbbk(A)$ is the canonical injection $\mathscr{L}^u(\Phi(X))=\mathscr{L}^u(\mathscr{L}_X)
		=\mathscr{L}_X$ follows. This completes the proof of the proposition.
	\end{proof}
	In categorical terms the above proposition states that $(\mathfrak{X}_u,\mathscr{L}^u)$ is the terminal object in the category of FOVCi over a fixed algebra.
	
	\medskip
	
	We now introduce the class of inner FOVCi. As an important instance of these, we elaborate on FOVCi on finite sets.
	
	\begin{definition}
		A FOVC $(\mathfrak{X},\mathscr{L})$ on an associative unital algebra $A$ is called \textit{inner} if there is a left $A$-linear map $\theta\colon\mathfrak{X}\to A$ such that
		\begin{equation}
			\mathscr{L}_X(a)=\theta(X)a-\theta(X\cdot a)
		\end{equation}
		for all $X\in\mathfrak{X}$ and $a\in A$.
	\end{definition}
	In other words, inner FOVCi are FOVCi $(\mathfrak{X},\mathscr{L})$ endowed with a special left $A$-linear map from $\mathfrak{X}$ to the algebra, such that the failure of this map to be right $A$-linear is measured by the Lie derivative. Next, we discuss FOVCi on finite sets. They turn out to be automatically inner and correspond to directed graphs on the given finite set.
	
	\begin{example}
		\label{graph-calculus}
		Let $X$ be a finite set, $A=\C[X]$ the algebra of functions on $X$. It follows that $A=\mathrm{span}_{\mathbb{C}}\{\de_x:X \lra \C\,|\, \de_x(y):=\de_{x,y}\}$ is a finite-dimensional vector space, where $\de_{x,y}$ denotes the Kronecker delta. 
		We show that FOVCi on $A$ correspond to directed graphs on the finite set $X$, where no self-loops and no multiple arrows in the same direction are allowed. Given such a graph on $X$ we
		define $\Chi:=\mathrm{span}_\C \{\chi_{x\ra y}\}$ as the
		vector space generated by the set of symbols $\{\chi_{x\ra y}\}$, 
		one corresponding to each edge $x\to y$ of the graph, with the following $A$-bimodule structure
		\begin{equation}\label{graph:bimod}
			f \cdot \chi_{x\ra y} \cdot g:=f(x)g(y)\chi_{x \ra y},
		\end{equation}
		for all $f,g\in A$, whenever $x\to y$ is an edge of the graph.
		We further define the Lie derivative
		\begin{equation}\label{graph:LieDer}
			\mathscr{L}_{\chi_{x\ra y}}(f):=(f(x)-f(y))\de_x,
		\end{equation}
		for all $f\in A$.
		It is straightforward to show that $(\Chi, \mathscr{L})$ is a FOVC. Moreover, this FOVC is {\sl inner}, with 
		\begin{equation}
			\theta\colon \Chi \to A, \qquad \theta(\chi_{x\ra y})=\de_x.
		\end{equation}
		Notice that by the very definition of the $A$-bimodule structure of $\Chi$, we have that $\theta(f \cdot \chi_{x\ra y} \cdot g)=f(x)g(y)\de_x$ for all $f,g\in A$ and all edges $x\to y$ of the graph.
		
		On the other hand, for an arbitrary FOVC $(\mathfrak{X},\mathscr{L})$ on $A=\mathbb{C}[X]$ by Proposition \ref{prop:univ} there is a morphism $(\mathfrak{X},\mathscr{L})\to(\mathfrak{X}_u,\mathscr{L}^u)$ to the universal FOVC on $A$. Since $A=\mathrm{span}_\mathbb{C}\{\delta_x~|~x\in X\}$ is a finite-dimensional vector space, $\mathfrak{X}_u=\{X\in\mathrm{End}_\mathbb{C}(A)~|~X(1)=0\}$ can be identified as a vector subspace of all $n\times n$-matrices with values in $\mathbb{C}$ ($n=|X|$ being the size of $X$) with matrix-vector multiplication corresponding to the evaluation of the corresponding endomorphism. Since these endomorphisms have to vanish at $1$ the diagonal of the corresponding matrix has to be zero. A vector space basis of these matrices is given by the matrices $\{E_{i,j}\}_{1\leq i\neq j\leq n}$, where $E_{i,j}$ is the matrix with $1$ at the position $(i,j)$ and zero elsewhere. To $E_{i,j}$ we assign an edge $x_i\to x_j$ (for a certain choice of numbering the elements of $X$) and thus the graph of the universal FOVC has an edge for each pair $(x,y)$ with $x,y\in X$ and $x\neq y$. By Proposition \ref{prop:univ} $\mathfrak{X}$ sits injectively inside $\mathfrak{X}_u$ and consequently we can assign a subgraph of the former one to the FOVC $(\mathfrak{X},\mathscr{L})$. By construction this correspondence is bijective.
		
		We shall come back to this example in a later section.
	\end{example}
	Noncommutative geometry on finite sets is relevant in the study of gauge theories \cite{DimMul,MajSim}, quantum gravity \cite{Cas} and quantum Riemannian geometry \cite[Chapter 8.2.2]{bm}.
	
	\subsection{First order differential calculi and FOVC}\label{sec:FODC}

	In this section we recall the well-established notion of first order differential calculus (FODC). We prove that the categories of first order vector field calculi (FOVCi) and of first order differential calculi (FODCi) over a fixed algebra are adjoint. Thus, first order vector field and differential calculi can be understood as ``dual to each other''.
	In the finitely generated projective setting the categories are equivalent.
	
	Recall that a  \textit{first order differential calculus} (FODC) \cite{Woronowicz1989} on an associative unital algebra $A$ is a tuple $(\Gamma,\mathrm{d})$, where
	\begin{enumerate}
		\item[i.)] $\Gamma$ is an $A$-bimodule.
		
		\item[ii.)] $\mathrm{d}\colon A\to\Gamma$ is a $\Bbbk$-linear map satisfying the Leibniz rule
		$$
		\mathrm{d}(ab)=\mathrm{d}(a)b+a\mathrm{d}(b)
		$$
		for all $a,b\in A$.
		
		\item[iii.)] $A\otimes A\to\Gamma$, $a^i\otimes b^i\mapsto a^i\mathrm{d}(b^i)$ is a (left $A$-linear and) surjective map.
	\end{enumerate}
	A \textit{morphism} of FODCi $(\Gamma,\mathrm{d})$ and $(\Gamma',\mathrm{d}')$ over the same algebra $A$ is an $A$-bimodule morphism $\Psi\colon\Gamma\to\Gamma'$ such that $\mathrm{d}'=\Psi\circ\mathrm{d}$. Note that $\Psi$ is necessarily surjective. On every algebra there exists the \textit{universal FODC} $(\Gamma_u,\mathrm{d}_u)$, where $\Gamma_u\subseteq A\otimes A$ is the kernel of the multiplication $A\otimes A\to A$ endowed with the obvious $A$-bimodule structure and $\mathrm{d}_u(a):=1\otimes a-a\otimes 1$ for all $a\in A$. For every FODC $(\Gamma,\mathrm{d})$ on $A$ there is a (necessarily surjective) morphism $(\Gamma_u,\mathrm{d}_u)\to(\Gamma,\mathrm{d})$ of FODCi, i.e., we can understand $(\Gamma,\mathrm{d})$ as a quotient of $(\Gamma_u,\mathrm{d}_u)$ (see e.g. \cite[Appendix A]{DurI}). There is a huge variety of examples and applications of FODCi, which is beyond the scope of this article. We refer the interested reader to the monograph \cite{bm} and \cite{DurII,SchDC}.
	
	Before entering the correspondence of vector field and differential calculi we recall that, given an $A$-bimodule $\Gamma$, there is a natural $A$-bimodule structure on the right $A$-linear maps $\mathfrak{X}:=\mathrm{Hom}_A(\Gamma,A)$ defined by
	\begin{equation}\label{A-bimod}
		\langle a\cdot X\cdot b,\omega\rangle:=a\langle X,b\cdot\omega\rangle
	\end{equation}
	for all $X\in\mathfrak{X}$, $a,b\in A$ and $\omega\in\Gamma$, where we denoted the evaluation $\mathrm{Hom}_A(\Gamma,A)\otimes\Gamma\to A$ by $\langle\cdot,\cdot\rangle$. With this $A$-bimodule structure the evaluation becomes an $A$-bilinear map and descends to an evaluation
	$$
	\langle\cdot,\cdot\rangle\colon\mathfrak{X}\otimes_A\Gamma\to A.
	$$
	Similarly, given an $A$-bimodule $\mathfrak{X}$, there is a natural $A$-bimodule structure on the left $A$-linear maps $\Gamma:={}_A\mathrm{Hom}(\mathfrak{X},A)$ induced by the evaluation $\langle X,a\cdot\omega\cdot b\rangle:=\langle X\cdot a,\omega\rangle b$ and the induced $A$-bilinear pairing reads as before (note however, that they correspond to different evaluations).
	
	\begin{proposition}[{\cite[Theorem 2.5]{Bor96}}]\label{prop:duality}
		Let $A$ be an associative unital algebra.
		\begin{enumerate}
			\item[i.)] For any FODC $(\Gamma,\mathrm{d})$ on $A$ there is a FOVC $(\mathfrak{X},\mathscr{L})$ on $A$ defined by $\mathfrak{X}:=\mathrm{Hom}_A(\Gamma,A)$ with Lie derivative $\mathscr{L}\colon\mathfrak{X}\to\mathrm{End}_\Bbbk(A)$ determined by
			\beq\label{fodc-fovc}
			\mathscr{L}_X(a):=\langle X,\mathrm{d}(a)\rangle
			\eeq
			for all $X\in\mathfrak{X}$ and $a\in A$.
			
			\item[ii.)] For any FOVC $(\mathfrak{X},\mathscr{L})$ there is a FODC $(\Gamma,\mathrm{d})$ on $A$ defined by $\Gamma:={}_A\mathrm{Hom}(\mathfrak{X},A)$ with differential $\mathrm{d}\colon A\to\Gamma$ determined by
			$$
			\langle X,\mathrm{d}(a)\rangle:=\mathscr{L}_X(a)
			$$
			for all $a\in A$ and $X\in\mathfrak{X}$.
		\end{enumerate}
	\end{proposition}
	\begin{proof}
		\begin{enumerate}
			\item[i.)] Let $(\Gamma,\mathrm{d})$ be a FODC on $A$. Then $\mathfrak{X}:=\mathrm{Hom}_A(\Gamma,A)$ is an $A$-bimodule w.r.t. the actions \eqref{A-bimod} as we already observed. Moreover, for all $a,b\in A$ and $X\in\mathfrak{X}$ the Leibniz rule
			\begin{align*}
				\mathscr{L}_X(ab)
				=\langle X,\mathrm{d}(ab)\rangle
				=\langle X,\mathrm{d}(a)b+a\mathrm{d}(b)\rangle
				=\langle X,\mathrm{d}(a)\rangle b+\langle X\cdot a,\mathrm{d}(b)\rangle
				=\mathscr{L}_X(a)b+\mathscr{L}_{X\cdot a}(b)
			\end{align*}
			holds by the Leibniz rule of $\mathrm{d}$ and the linearity of the pairing. It remains to prove that $\mathscr{L}\colon\mathfrak{X}\to\mathrm{End}_\Bbbk(A)$ is injective. Assume that $\mathscr{L}_X$ equals the zero endomorphism for an $X\in\mathfrak{X}$, i.e., $\mathscr{L}_X(a)=0$ for all $a\in A$. By the surjectivity of $(\Gamma,\mathrm{d})$ we have that for all $\omega\in\Gamma$ there exist finitely many $a^i,b^i\in A$ such that $\omega=a^i\mathrm{d}(b^i)$ with summation over $i$ understood. Thus
			\begin{align*}
				\langle X,\omega\rangle
				=\langle X,a^i\mathrm{d}(b^i)\rangle
				=\langle X,\mathrm{d}(a^ib^i)-\mathrm{d}(a^i)b^i\rangle
				=\mathscr{L}_X(a^ib^i)-\mathscr{L}_X(a^i)b^i
				=0
			\end{align*}
			by the Leibniz rule of $\mathrm{d}$, the linearity of the pairing and by the previous assumption $\mathscr{L}_X(a)=0$ for all $a\in A$. The pairing is non-degenerate, which implies that $X=0$ and proves injectivity.
			\item[ii.)] Assume that $(\mathfrak{X},\mathscr{L})$ is a FOVC on $A$. We already observed that $\Gamma:={}_A\mathrm{Hom}(\mathfrak{X},A)$ is an $A$-bimodules with actions defined symmetrically to \eqref{A-bimod}. Then we obtain for all $a,b\in A$ and $X\in\mathfrak{X}$
			\begin{align*}
				\langle X,\mathrm{d}(ab)\rangle
				=\mathscr{L}_X(ab)
				=\mathscr{L}_X(a)b+\mathscr{L}_{X\cdot a}(b)
				=\langle X,\mathrm{d}(a)\rangle b+\langle X\cdot a,\mathrm{d}(b)\rangle
				=\langle X,\mathrm{d}(a)b+a\mathrm{d}(b)\rangle
			\end{align*}
			by the Leibniz rule of $\mathscr{L}$ and the linearity of $\langle\cdot,\cdot\rangle$. Since the pairing is non-degenerate and the above holds for all $X\in\mathfrak{X}$ this implies the Leibniz rule for $\mathrm{d}$. It remains to prove surjectivity of $(\Gamma,\mathrm{d})$. Note that $\mathrm{d}(A)A:=\{\eta\in\Gamma~|~\exists a^i,b^i\in A \text{ s.t. }\omega=\mathrm{d}(a^i)b^i\}\subseteq\Gamma$ is a vector subspace. Assume there is an $\omega\in\Gamma$ such that $\omega\notin\mathrm{d}(A)A$. In particular $\omega\neq 0$. Since $\Gamma={}_A\mathrm{Hom}(\mathfrak{X},A)$ and $\omega\neq 0$ there exists an $X\in\mathfrak{X}$ such that $\langle X,\omega\rangle\neq 0$ and $\langle X,\mathrm{d}(A)A\rangle=0$. The latter implies that $\mathscr{L}_X(a)=\langle X,\mathrm{d}(a)\rangle=0$ for all $a\in A$, which gives $X=0$ by the injectivity of $\mathscr{L}$. This is a contradiction to $\langle X,\omega\rangle\neq 0$. Consequently $\omega$ does not exist and $\Gamma=\mathrm{d}(A)A$. An easy exercise (using the Leibniz rule) shows that $\mathrm{d}(A)A=A\mathrm{d}A$ as $A$-bimodules, completing the proof of the proposition.
		\end{enumerate}
	\end{proof}
	We would like to stress that the above correspondence is not $1$:$1$ in general. However, this is the case if we restrict ourself to the study of finitely generated projective modules, as we comment in the following.
	\begin{observation}\label{obs:fd}
		Let $\Gamma$ and $\mathfrak{X}$ be finitely generated projective 
		left, resp. right, $A$-modules.
		If $\Gamma={}_A\mathrm{Hom}(\mathfrak{X},A)$, then
		$\mathfrak{X}\cong\mathrm{Hom}_A(\Gamma,A)$, see e.g. \cite[Section 2.2]{paolo}. When $(\mathfrak{X},\mathscr{L})$ is a FOVC
		one can immediately check that the above isomorphism is a FOVC isomorphism (see also Proposition \ref{prop:duality}).
		Similarly, if $\mathfrak{X}=\mathrm{Hom}_A(\Gamma,A)$, then
		$\Gamma \cong {}_A\mathrm{Hom}(\mathfrak{X},A)$ as FODCi. Hence, in the case of 
		finitely generated projective $A$-modules, 
		Proposition \ref{prop:duality} gives a 1:1 correspondence
		between FOVCi and FODCi on $A$. 
	\end{observation}
	Moreover, the dual of an inner FOVC is again inner, while the converse holds in the finitely generated projective case.
	\begin{lemma}\label{lem:inner}
		The dual of an inner FOVC is an inner FODC. The dual of a finitely generated projective inner FODC is an inner FOVC.
	\end{lemma}
	\begin{proof}
		For the first statement let $(\mathfrak{X},\mathscr{L})$ be an inner FOVC with corresponding left $A$-linear $\theta\colon\mathfrak{X}\to A$. Since the dual FODC is defined by $\Gamma={}_A\mathrm{Hom}(\mathfrak{X},A)$ {we have} $\theta\in\Gamma$ and
		$$
		\langle X,\mathrm{d}a\rangle
		=\mathscr{L}_X(a)
		=\theta(X)a-\theta(X\cdot a)
		=\langle X,\theta\rangle a-\langle X\cdot a,\theta\rangle
		=\langle X,\theta\cdot a\rangle-\langle X,a\cdot\theta\rangle,
		$$
		which implies $\mathrm{d}a=\theta\cdot a-a\cdot\theta$.
		
		The other implication is proven by reading the above formula the other way round and using $\mathfrak{X}=\mathrm{Hom}_A(\Gamma,A)$. To identify $\theta\in{}_A\mathrm{Hom}(\mathfrak{X},A)={}_A\mathrm{Hom}(\mathrm{Hom}_A(\Gamma,A),A)\cong\Gamma$ we need that $\Gamma$ is finitely generated projective as an $A$-module, see Observation \ref{obs:fd}.
	\end{proof}
	
	As an instance of the above discussion we examine the FOVC of Example \ref{graph-calculus} and its dual FODC.
	
	\begin{example}
		Let $X$ be a finite set, $A=\C[X]$ the algebra of functions on $X$. Let $(\Chi, \mathscr{L})$ be the {FOVC corresponding to a directed graph on $X$}, as discussed in Example \ref{graph-calculus}. We define a FODC $(\Gamma,\mathrm{d})$ 
		on $A$ via {$\Gamma:={}_A\mathrm{Hom}(\mathfrak{X},A)$ and (\ref{fodc-fovc}). Since $\mathfrak{X}=\mathrm{span}_\mathbb{C}\{\chi_{x\to y}\}$ is the vector space generated by $\chi_{x\to y}$ for all edges $x\to y$ of the graph, it follows that $\Gamma$ is vector space spanned by the dual $\omega_{x\to y}$, where 
			$$
			\langle\chi_{x\to y},\omega_{y'\to x'}\rangle=\delta_{x,x'}\delta_{y,y'}\delta_x.
			$$
			The differential is determined by}
		\begin{equation}
			\langle \chi_{x\ra y}, df \rangle := \mathscr{L}_{\chi_{x\ra y}}(f)=(f(x)-f(y))\de_x
		\end{equation}
		for all $f\in A$ and edges $x\to y$ of the graph.
		Since $(\Chi, \mathscr{L})$ is inner, also $(\Gamma, \mathrm{d})$ is inner by Lemma \ref{lem:inner} with
		$$
		df=[\Theta,f], \qquad \hbox{where} \quad \Theta=\sum_{y\ra x} (f(x)-f(y))\omega_{x\ra y}.
		$$
		This is the FODC appearing in \cite[Proposition 1.24]{bm}, {see also \cite{Cas,DimMul}}.
	\end{example}

	We can rephrase Proposition \ref{prop:duality} in categorical terms. 
	For this, we introduce the dualizing functors
	\begin{equation}\label{dualfunctors}
		(\cdot)^*\colon{}_A\mathcal{M}_A\to({}_A\mathcal{M}_A)^\mathrm{op},\qquad
		{}^*(\cdot)\colon({}_A\mathcal{M}_A)^\mathrm{op}\to{}_A\mathcal{M}_A
	\end{equation}
	on objects $M\in{}_A\mathcal{M}_A$ by $M^*:=\mathrm{Hom}_A(M,A)$, ${}^*M:={}_A\mathrm{Hom}(M,A)$ and on morphisms $\phi\colon M\to M'$ in ${}_A\mathcal{M}_A$ by $\phi^*\colon\mathrm{Hom}_A(M',A)\to\mathrm{Hom}_A(M,A)$ and ${}^*\phi\colon{}_A\mathrm{Hom}(M',A)\to{}_A\mathrm{Hom}(M,A)$, where
	\begin{equation}\label{eq:DualFuncMorph}
		\phi^*(f):=f\circ\phi\qquad\text{and}\qquad{}^*\phi(g):=g\circ\phi
	\end{equation}
	for all $f\in\mathrm{Hom}_A(M',A)$ and $g\in{}_A\mathrm{Hom}(M',A)$. Above, $({}_A\mathcal{M}_A)^\mathrm{op}$ denotes the opposite category of ${}_A\mathcal{M}_A$, which has the same objects, but is endowed with morphisms of reversed arrows. Thus, in other words, \eqref{dualfunctors} correspond to \textit{contra}variant endofunctors. It is straightforward to show that \eqref{dualfunctors} are in fact functorial.

	As shown in \cite[Proposition 9]{PoinsotPorst}, $(\cdot)^*$ is the left adjoint functor of ${}^*(\cdot)$. In more detail, for every object $M$ in ${}_A\mathcal{M}_A$ the unit of the adjunction is the morphism $\eta_M\colon M\to{}^*(M^*)$ in ${}_A\mathcal{M}_A$ defined for all $M$ by
	\begin{equation}\label{eq:unit}
		\eta_M(m)\colon M^*\to A,\qquad\phi\mapsto\phi(m)
	\end{equation}
	and one easily checks that \eqref{eq:unit} is well-defined, left $A$-linear and that $\eta_M$ is $A$-bilinear. Moreover, the counit of the adjunction is the morphism $\epsilon_M\colon({}^*M)^*\to M$ in $({}_A\mathcal{M}_A)^\mathrm{op}$ (recall that the arrows in the opposite categories are reversed) corresponding to the morphism $\epsilon'_M\colon M\to({}^*M)^*$ in ${}_A\mathcal{M}_A$, given for all $m\in M$ by
	\begin{equation}\label{eq:counit}
		\begin{split}
			\epsilon'_M(m)\colon{}^*M\to A,\qquad\psi\mapsto\psi(m)
		\end{split}
	\end{equation}
	and one verifies that \eqref{eq:counit} is well-defined, right $A$-linear and that $\epsilon'_M$ is $A$-bilinear.
	
	We show that the previous adjunction induces an adjunction between the category of first order differential calculi and the opposite category of first order vector field calculi.
	\begin{proposition}\label{prop:Adjoint}
		For any associative unital algebra $A$ the dualizing functors $(\cdot)^*$, ${}^*(\cdot)$
		induce an adjunction between the category $\underline{\mathrm{FODC}}_A$ of FODCi on $A$ and $(\underline{\mathrm{FOVC}}_A)^\mathrm{op}$ the category of {FOVCi} on $A$ with opposite morphisms:
		\begin{equation*}
			\begin{tikzcd}
				\underline{\mathrm{FODC}}_A \arrow[rr, bend left, "(\cdot)^*"] \arrow{ddd}[swap]{\mathrm{forget}_\mathrm{d}} & \perp & (\underline{\mathrm{FOVC}}_A)^\mathrm{op} \arrow[ll, bend left, "{}^*(\cdot)"] \arrow{ddd}{(\mathrm{forget}_\mathscr{L})^\mathrm{op}}\\
				& & \\
				& & \\
				{}_A\mathcal{M}_A \arrow[rr, bend left, "(\cdot)^*"] & \perp & ({}_A\mathcal{M}_A)^\mathrm{op} \arrow[ll, bend left, "{}^*(\cdot)"]
			\end{tikzcd}
		\end{equation*}
		where
		$\mathrm{forget}_\mathrm{d}(\Gamma,\mathrm{d})=\Gamma$ and $\mathrm{forget}_\mathscr{L}(\mathfrak{X},\mathscr{L})=\mathfrak{X}$
		are the forgetful functors.
	\end{proposition}
	
	\begin{proof}
		The lower part of the diagram states the previously mentioned adjunction between ${}_A\mathcal{M}_A$ and $({}_A\mathcal{M}_A)^\mathrm{op}$, as proven in \cite[Proposition 9]{PoinsotPorst}. We further showed in Proposition \ref{prop:duality} that the functors \eqref{dualfunctors} restrict and corestrict to the categories in the top of the diagram, at least on objects. We prove that they also restrict and corestrict on morphisms and that the unit and counit of the adjunction are morphisms in the correct categories.
		
		Let $(\Gamma,\mathrm{d})\to(\Gamma',\mathrm{d}')$ be a morphism of FODCi on $A$ with $\Psi\colon\Gamma\to\Gamma'$ denoting the corresponding $A$-bilinear morphism. Then $\Psi^*\colon(\Gamma')^*\to\Gamma^*$, $f\mapsto f\circ\Psi$ is the induced $A$-bimodule morphism between $\mathfrak{X}':=(\Gamma')^*=\mathrm{Hom}_A(\Gamma',A)$ and $\mathfrak{X}:=\Gamma^*=\mathrm{Hom}_A(\Gamma,A)$. By Proposition \ref{prop:duality} the $A$-bimodules $\mathfrak{X}'$ and $\mathfrak{X}$ are first order vector field calculi with Lie derivatives given respectively by $\mathscr{L}'_{X'}(a)=\langle X',\mathrm{d}'(a)\rangle$ and $\mathscr{L}_X(a)=\langle X,\mathrm{d}(a)\rangle$ for all $X'\in\mathfrak{X}'$, $X\in\mathfrak{X}$ and $a\in A$. Then, the $A$-bilinear map $\Psi^*\colon\mathfrak{X}'\to\mathfrak{X}$ is a morphism of FOVCi, i.e., $\mathscr{L}\circ\Psi^*=\mathscr{L}'$ for all $X'\in\mathfrak{X}'$ since
		\begin{align*}
			\mathscr{L}(\Psi^*(X'))(a)
			=\langle\Psi^*(X'),\mathrm{d}a\rangle
			=\langle X',\Psi(\mathrm{d}a)\rangle
			=\langle X',\mathrm{d}'a\rangle
			=\mathscr{L}_{X'}(a)
		\end{align*}
		for all $a\in A$, using that $\Psi$ is a morphism of FODCi. Equivalently, we obtained a morphism $(\mathfrak{X},\mathscr{L})\to(\mathfrak{X}',\mathscr{L}')$ in $(\underline{\mathrm{FOVC}}_A)^\mathrm{op}$.

		On the other hand, and entirely similarly, given a morphism $(\mathfrak{X},\mathscr{L})\to(\mathfrak{X}',\mathscr{L}')$ of first order vector field calculi with corresponding $A$-bimodule morphism $\Phi\colon\mathfrak{X}\to\mathfrak{X}'$, the induced $A$-bimodule morphism ${}^*\Phi\colon{}^*(\mathfrak{X}')\to{}^*\mathfrak{X}$ gives a morphism $({}^*(\mathfrak{X}'),\mathrm{d}')\to({}^*\mathfrak{X},\mathrm{d})$ of FODCi, where $\mathrm{d}'\colon A\to{}^*(\mathfrak{X}')$ and $\mathrm{d}\colon A\to{}^*\mathfrak{X}$ are determined by
		$$
		\langle X',\mathrm{d}'a\rangle:=\mathscr{L}'_{X'}(a)\qquad\text{and}\qquad
		\langle X,\mathrm{d}a\rangle:=\mathscr{L}_{X}(a)
		$$
		for all $X'\in\mathfrak{X}'$, $X\in\mathfrak{X}$ and $a\in A$ via Proposition \ref{prop:duality}. In fact,
		$$
		\langle X,{}^*\Phi(\mathrm{d}'a)\rangle
		=\langle\Phi(X),\mathrm{d}'a\rangle
		=\mathscr{L}'_{\Phi(X)}(a)
		=\mathscr{L}_X(a)
		=\langle X,\mathrm{d}a\rangle
		$$
		for all $X\in\mathfrak{X}$ and $a\in A$ implies ${}^*\Phi\circ\mathrm{d}'=\mathrm{d}$. Thus, reversing the arrow of the morphism $\Phi$ in order to understand it as a morphism in $(\underline{\mathrm{FOVC}}_A)^\mathrm{op}$, we obtain a morphism in the correct category. 
		
		It remains to prove that, given a FODC $(\Gamma,\mathrm{d})$ and a FOVC $(\mathfrak{X},\mathscr{L})$, the unit of the adjunction $\eta_\Gamma\colon\Gamma\to{}^*(\Gamma^*)$ is a morphism in $\underline{\mathrm{FODC}}_A$ and the counit of the adjunction $\epsilon_\mathfrak{X}\colon({}^*\mathfrak{X})^*\to\mathfrak{X}$ is a morphism in $(\underline{\mathrm{FOVC}}_A)^\mathrm{op}$. We already argued before the proposition that this is the case on the level of $A$-bimodules and thus it remains to show compatibility with the differential and Lie derivative, respectively. Recalling equation \eqref{eq:unit}, denoting the Lie derivative on $\Gamma^*$ by $\mathscr{L}$ and denoting the differential on ${}^*(\Gamma^*)$ by $\tilde{\mathrm{d}}$, we obtain the equality $\eta_\Gamma(\mathrm{d}a)=\tilde{\mathrm{d}a}$ for all $a\in A$ since
		\begin{align*}
			\langle X,\eta_\Gamma(\mathrm{d}a)\rangle
			=X(\mathrm{d}a)
			=\mathscr{L}_X(a)
			=\langle X,\tilde{d}a\rangle
		\end{align*}
		for all $X\in\Gamma^*$. Thus, the unit $\eta_M$ is a morphism in $\underline{\mathrm{FODC}}_A$. In complete analogy one shows that $\epsilon_\mathfrak{X}\colon\mathfrak{X}\to({}^*\mathfrak{X})^*$ is a morphism in $\underline{\mathrm{FOVC}}_A$ and thus it follows that the counit of the adjunction is a morphism in $(\underline{\mathrm{FOVC}}_A)^\mathrm{op}$.
	\end{proof}

	Denote the full subcategory of finitely generated projective 
	left $A$-modules in $\underline{\mathrm{FODC}}_A$ by $\underline{\mathrm{FODC}}^{\fgp}_A$. Similarly, $\underline{\mathrm{FOVC}}^{\fgp}_A$ denotes the full subcategory of finitely generated projective
	right $A$-modules in $\underline{\mathrm{FOVC}}_A$.
	
	\begin{theorem}\label{dual-cp} 
		The dualizing functors $(\cdot)^*$, ${}^*(\cdot)$ induce an equivalence of categories
		$$
		\underline{\mathrm{FODC}}^{\fgp}_A\cong(\underline{\mathrm{FOVC}}^{\fgp}_A)^\mathrm{op}.
		$$
	\end{theorem}
	
	\begin{proof}
		This follows from Proposition \ref{prop:Adjoint} and Observation \ref{obs:fd}, 
		also noting that, in the notation of the proof of Proposition \ref{prop:Adjoint}, the differential $\mathrm{d}$ can be identified with the differential $\tilde{\mathrm{d}}$ on ${}^*(\Gamma^*)$ under the $A$-bimodule isomorphism $\Gamma\xrightarrow{\cong}{}^*(\Gamma^*)$ and similarly for the Lie derivatives on $\mathfrak{X}$ and $({}^*\mathfrak{X})^*$.
	\end{proof}
	
	\begin{remark}
		{We have an equivalence of categories between FODCi on $\Bbbk[V]$, $V$ a finite set,
			and the category of directed graphs with vertices $V$, no self-loops and no multiple arrows of the same direction \cite[Chapter 1.4]{bm}. Because of Theorem \ref{dual-cp} 
			we immediately obtain an equivalence categories between the opposite category of FOVCi on $\Bbbk[V]$
			and the category of directed graphs with vertices $V$, as above.}
	\end{remark}

	\subsection{Covariant first order vector calculi}\label{sec:covFOVC}
	
	Let $H$ be a Hopf algebra and $A$ a right $H$-comodule algebra with coaction $\Delta_A\colon A\to A\otimes H$. Recall that this means that $(A,\Delta_A)$ is a right $H$-comodule and that $A$ is an algebra such that $\Delta_A$ is an algebra morphism. We employ the following common short notation for the coaction and write
	$$
	\Delta_A(a)=:a_0\otimes a_1,\qquad(\Delta_A\otimes\mathrm{id})(\Delta_A(a))=(\mathrm{id}\otimes\Delta)\Delta_A(a)=:a_0\otimes a_1\otimes a_2
	$$
	etc., omitting summation symbols. For more details on comodules and comodule algebras we refer to the textbook \cite{mont}.

	In the following we introduce covariant FOVCi as FOVCi $(\mathfrak{X},\mathscr{L})$ on comodule algebras $A$, where $\mathfrak{X}$ is endowed with a compatible coaction such that the Lie derivative $\mathscr{L}\colon\mathfrak{X}\to\mathrm{End}_\Bbbk(A)$, understood as a linear map
	\begin{equation}\label{eq:LTilde}
		\widetilde{\mathscr{L}}\colon\mathfrak{X}\otimes A\to A,\qquad\widetilde{\mathscr{L}}(X\otimes a):=\mathscr{L}_X(a),
	\end{equation}
	is colinear.
	\begin{definition}\label{def:covFOVC}
		A FOVC $(\mathfrak{X},\mathscr{L})$ on a right $H$-comodule algebra $A$ is called \textit{right $H$-covariant} if
		\begin{enumerate}
			\item[i.)] {$\mathfrak{X}$ is a right $H$-covariant $A$-bimodule, i.e., if there is a right $H$-coaction $\Delta_{\mathfrak{X}}\colon\mathfrak{X}\to\mathfrak{X}\otimes H$ such that
				$$
				\Delta_{\mathfrak{X}}(a\cdot X\cdot a')=\Delta_A(a)\Delta_{\mathfrak{X}}(X)\Delta_A(a')
				$$
				for all $a,a'\in A$ and $X\in\mathfrak{X}$.}
			
			\item[ii.)] the map $\widetilde{\mathscr{L}}\colon\mathfrak{X}\otimes A\to A$ defined in \eqref{eq:LTilde} is right $H$-colinear, i.e., the diagram
			\begin{equation}\label{diagLcov}
				\begin{tikzcd}
					\mathfrak{X}\otimes A \arrow{rr}{\widetilde{\mathscr{L}}} \arrow{d}[swap]{\Delta_\otimes}
					& & A \arrow{d}{\Delta_A}\\
					\mathfrak{X}\otimes A\otimes H \arrow{rr}{\widetilde{\mathscr{L}}\otimes\mathrm{id}_H}
					& & A\otimes H
				\end{tikzcd}
			\end{equation}
			commutes, where we endow $\mathfrak{X}\otimes A$ with the diagonal right $H$-coaction $\Delta_\otimes\colon\mathfrak{X}\otimes A\to\mathfrak{X}\otimes A\otimes H$, $X\otimes a\mapsto X_0\otimes a_0\otimes X_1a_1$.
		\end{enumerate}
		A morphism of right $H$-covariant FOVCi is a morphism of the underlying FOVCi such that the corresponding $A$-bimodule morphism is right $H$-colinear, in addition.
		The category of right $H$-covariant FOVCi on a fixed right $H$-comodule algebra $A$ is denoted by $\underline{\mathrm{FOVC}}_A^H$.

	\end{definition}
	Similarly, a \textit{left $H$-covariant FOVC} on a left $H$-comodule algebra $A$ is a FOVC $(\mathfrak{X},\mathscr{L})$ on $A$ such that $\mathfrak{X}$ is a left $H$-covariant $A$-bimodule and \eqref{eq:LTilde} is left $H$-colinear, while an \textit{$H$-bicovariant FOVC} on an $H$-bicomodule algebra $A$ is a FOVC $(\mathfrak{X},\mathscr{L})$ on $A$ such that $\mathfrak{X}$ is an $H$-bicovariant $A$-bimodule such that $(\mathfrak{X},\mathscr{L})$ is left and right $H$-covariant. The concept of covariant FOVCi is implicit in \cite{BMvector}.
	
	As a first example of a covariant first order vector field calculus, we discuss the well-known noncommutative $2$-torus \cite{ConnesRieffel}.
	\begin{example}
		Fix a real parameter $\theta\in\mathbb{R}$ and consider the algebra $A:=\mathcal{O}_\theta(\mathbb{T}^2)$ generated by invertible elements $u,v$, modulo the relation $vu=e^{\mathrm{i}\theta}uv$, known as the algebraic noncommutative $2$-torus. We further consider the commutative Hopf algebra $H=\mathcal{O}(U(1))$ generated by one invertible group-like element $t$, i.e. $\Delta(t^{\pm 1})=t^{\pm 1}\otimes t^{\pm 1}$, $\varepsilon(t^{\pm 1})=1$ and $S(t^{\pm 1})=t^{\mp 1}$. Then $A$ becomes a right $H$-comodule algebra if we endow it with the right $H$-coaction $\Delta_A\colon A\to A\otimes H$, determined on generators by
		$$
		\Delta_A\begin{pmatrix}
			u \\
			v
		\end{pmatrix}:=\begin{pmatrix}
			u \\
			v
		\end{pmatrix}\otimes\begin{pmatrix}
			t \\
			t^{-1}
		\end{pmatrix}
		$$
		and extended as an algebra morphism.
		
		There is a right $H$-covariant FOVC $(\mathfrak{X},\mathscr{L})$ on $A$, defined as the left $A$-span $\mathfrak{X}:=\mathrm{span}_A\{\partial_u,\partial_v\}$ of two elements $\partial_u$ and $\partial_v$. We structure $\mathfrak{X}$ as an $A$-bimodule with the obvious left action and right $A$-action determined on generators by
		\begin{equation}
			\begin{split}
				a\partial_u\cdot u
				&=au\partial_u,\\
				a\partial_u\cdot v
				&=e^{\mathrm{i}\theta}av\partial_u,
			\end{split}\qquad\qquad
			\begin{split}
				a\partial_v\cdot u
				&=e^{-\mathrm{i}\theta}au\partial_v,\\
				a\partial_v\cdot v
				&=av\partial_v,
			\end{split}
		\end{equation}
		where $a\in A$. The $A$-bimodule $\mathfrak{X}$ becomes a right $H$-covariant $A$-bimodule if endowed with the right $H$-coaction $\Delta_\mathfrak{X}\colon\mathfrak{X}\to\mathfrak{X}\otimes H$
		\begin{equation}\label{coactT2}
			\Delta_\mathfrak{X}(a\partial_u+b\partial_v)
			:=a_0\partial_u\otimes a_1t^{-1}+b_0\partial_v\otimes b_1t,
		\end{equation}
		where $a,b\in A$. In fact, one verifies on generators that the right $A$-action on $\mathfrak{X}$ is compatible with the coaction \eqref{coactT2}, namely
		\begin{align*}
			\Delta_\mathfrak{X}(\partial_u)\cdot\Delta_A\begin{pmatrix}
				u \\
				v
			\end{pmatrix}
			&=(\partial_u\otimes t^{-1})\cdot\bigg(\begin{pmatrix}
				u \\
				v
			\end{pmatrix}\otimes\begin{pmatrix}
				t \\
				t^{-1}
			\end{pmatrix}\bigg)
			=\partial_u\cdot\begin{pmatrix}
				u \\
				v
			\end{pmatrix}\otimes\begin{pmatrix}
				1 \\
				t^{-2}
			\end{pmatrix}
			=\begin{pmatrix}
				u \\
				e^{\mathrm{i}\theta}v
			\end{pmatrix}\partial_u\otimes\begin{pmatrix}
				1 \\
				t^{-2}
			\end{pmatrix}\\
			&=\Delta_\mathfrak{X}\bigg(\begin{pmatrix}
				u \\
				e^{\mathrm{i}\theta}v
			\end{pmatrix}\partial_u\bigg)
			=\Delta_\mathfrak{X}\bigg(\partial_u\cdot\begin{pmatrix}
				u \\
				v
			\end{pmatrix}\bigg)
		\end{align*}
		and similarly for $\partial_v$.
		We further define the Lie derivative for all $a,b\in A$ on generators by
		\begin{equation}
			\mathscr{L}_{a\partial_u+b\partial_v}\begin{pmatrix}
				u \\
				v
			\end{pmatrix}=\begin{pmatrix}
				a \\
				b
			\end{pmatrix}
		\end{equation}
		and extend the above via the Leibniz rule. Explicitly, 
		$$
		\mathscr{L}_{a\partial_u+b\partial_v}(u^mv^n)
		=mau^{m-1}v^n+ne^{-m\mathrm{i}\theta}bu^mv^{n-1}
		$$ 
		for all $a,b\in A$ and $m,n\in\mathbb{Z}$. From the above formula it immediately follows that $\mathscr{L}_{a\partial_u+b\partial_v}=0$ is the zero endomorphism on $A$ if and only if $a=b=0$. Thus, $\mathscr{L}$ is injective and $(\mathfrak{X},\mathscr{L})$ a FOVC. For right $H$-covariance of the FOVC we note that for all $a\in A$ and $m,n\in\mathbb{Z}$
		\begin{equation*}
			\begin{tikzcd}
				a\partial_u\otimes u^mv^n \arrow{rrr}{\widetilde{\mathscr{L}}} \arrow{d}[swap]{\Delta_\otimes}
				& & & mau^{m-1}v^n \arrow{d}{\Delta_A}\\
				a_0\partial_u\otimes u^mv^n\otimes a_1t^{-1}t^{m-n} \arrow{rrr}{\widetilde{\mathscr{L}}\otimes\mathrm{id}}
				& & & ma_0u^{m-1}v^n\otimes a_1t^{m-n-1}
			\end{tikzcd}
		\end{equation*}
		commutes and similarly for $\partial_v$. This shows that $(\mathfrak{X},\mathscr{L})$ is a right $H$-covariant FOVC. Since $H$ is commutative and cocommutative it is easy to infer that $(\mathfrak{X},\mathscr{L})$ is also left $H$-covariant and even $H$-bicovariant.
	\end{example}
	
	We now specialize the adjunction of the categories of first order vector field calculi and first order differential calculi (see Proposition \ref{prop:Adjoint}) to their covariant analogues. We have seen that the functors \eqref{dualfunctors} defined on objects by $(\cdot)^*\colon M\mapsto M^*:=\mathrm{Hom}_A(M,A)$ and ${}^*(\cdot)\colon M\mapsto{}^*M:={}_A\mathrm{Hom}(M,A)$, respectively, map FODCi to FOVCi and vice versa. In the case of covariant FODCi and covariant FOVCi we consider instead the functors $\underline{(\cdot)^*}\colon{}_A\mathcal{M}_A^H\to({}_A\mathcal{M}_A^H)^\mathrm{op}$ and $\underline{{}^*(\cdot)}\colon({}_A\mathcal{M}_A^H)^\mathrm{op}\to{}_A\mathcal{M}_A^H$, defined on objects $M\in{}_A\mathcal{M}_A^H$ by
	\begin{equation}\label{rationalmorh}
		\begin{split}
			\underline{(\cdot)^*}\colon M\mapsto\underline{M^*}&:=\underline{\mathrm{Hom}}_A(M,A)
			:=\{\phi\in\mathrm{Hom}_A(M,A)~|~\Delta_\mathrm{ad}(\phi)\in\mathrm{Hom}_A(M,A)\otimes H\},\\
			\underline{{}^*(\cdot)}\colon M\mapsto\underline{{}^*M}&:={}_A\underline{\mathrm{Hom}}(M,A):=\{\psi\in{}_A\mathrm{Hom}(M,A)~|~\Delta'_\mathrm{ad}(\psi)\in{}_A\mathrm{Hom}(M,A)\otimes H\},  
		\end{split}
	\end{equation}
	where $\Delta_\mathrm{ad}\colon\mathrm{Hom}_A(M,A)\to\mathrm{Hom}_\Bbbk(M,A\otimes H)$ and $\Delta'_\mathrm{ad}\colon{}_A\mathrm{Hom}(M,A)\to\mathrm{Hom}_\Bbbk(M,A\otimes H)$ are defined by
	\begin{equation}
		\begin{split}
			\Delta_\mathrm{ad}(\phi)(m)&:=\phi(m_0)_0\otimes\phi(m_0)_1S(m_1),\\
			\Delta'_\mathrm{ad}(\psi)(m)&:=\phi(m_0)_0\otimes S^{-1}(m_1)\phi(m_0)_1
		\end{split}
	\end{equation}
	for all $m\in M$. Note that in general $\underline{\mathrm{Hom}}_A(M,A)\subsetneq\mathrm{Hom}_A(M,A)$ and ${}_A\underline{\mathrm{Hom}}(M,A)\subsetneq{}_A\mathrm{Hom}(M,A)$, since in general $\mathrm{Hom}_A(M,A)\otimes H\subsetneq\mathrm{Hom}_A(M,A\otimes H)$ and ${}_A\mathrm{Hom}(M,A)\otimes H\subsetneq{}_A\mathrm{Hom}(M,A\otimes H)$. 
	Elements of $\underline{\mathrm{Hom}}_A(M,A)$ and ${}_A\underline{\mathrm{Hom}}(M,A)$ are called \textit{rational morphisms}. They haven been introduced in \cite{Ulbrich,Caenepeel}. In complete analogy to \cite[Section 2.2]{PaoloThomas} one shows that $\underline{\mathrm{Hom}}_A(M,A),{}_A\underline{\mathrm{Hom}}(M,A)$ are right $H$-covariant $A$-bimodules with coactions induced by $\Delta_\mathrm{ad}$ and $\Delta'_\mathrm{ad}$, respectively. In other words, $\underline{\mathrm{Hom}}_A(M,A)$ and ${}_A\underline{\mathrm{Hom}}(M,A)$ are the subspaces of $\mathrm{Hom}_A(M,A)$ and ${}_A\mathrm{Hom}(M,A)$ which admit right adjoint coactions. 
	For example, one verifies that $\Delta_\mathrm{ad}(\phi\cdot a)=\Delta_\mathrm{ad}(\phi)\Delta_A(a)\in\underline{\mathrm{Hom}}_A(M,A)\otimes H$ holds for $\phi\in\underline{\mathrm{Hom}}_A(M,A)$ and $a\in A$ by computing
	\begin{align*}
		(\phi_0\cdot a_0)(m)\otimes \phi_1a_1
		&=\phi_0(a_0\cdot m)\otimes\phi_1a_1\\
		&=\phi((a_0\cdot m)_0)_0\otimes\phi((a_0\cdot m)_0)_1S((a_0\cdot m)_1)a_1\\
		&=\phi(a_0\cdot m_0)_0\otimes\phi(a_0\cdot m_0)_1S(a_1m_1)a_2\\
		&=\phi(a\cdot m_0)_0\otimes\phi(a\cdot m_0)_1S(m_1)\\
		&=(\phi\cdot a)(m_0)_0\otimes(\phi\cdot a)(m_0)_1S(m_1)\\
		&=(\phi\cdot a)_0(m)\otimes(\phi\cdot a)_1
	\end{align*}
	for all $m\in M$, where we also abbreviated $\Delta_\mathrm{ad}(\phi)=:\phi_0\otimes\phi_1$.
	
	We show that \eqref{rationalmorh} are adjoint functors, which respect covariant first order differential and vector field calculi.
	Namely, we define the functors 
	\begin{equation}\label{RationalFunctors}
		\underline{(\cdot)^*}\colon{}_A\mathcal{M}_A^H\to({}_A\mathcal{M}_A^H)^\mathrm{op},\qquad
		\underline{{}^*(\cdot)}\colon({}_A\mathcal{M}_A^H)^\mathrm{op}\to{}_A\mathcal{M}_A^H
	\end{equation}
	on objects as in \eqref{rationalmorh} and on morphisms as in \eqref{eq:DualFuncMorph}.
	Recall that a FODC $(\Gamma,\mathrm{d})$ on a right $H$-comodule algebra $A$ is called \textit{right $H$-covariant}, if $\Gamma\in{}_A\mathcal{M}_A^H$ is a right $H$-covariant $A$-bimodule and $\mathrm{d}\colon A\to\Gamma$ is right $H$-colinear. 
	
	\begin{proposition}\label{prop:duality'}
		For any right $H$-comodule algebra $A$ the dualizing functors $\underline{{}^*(\cdot)}$, $\underline{(\cdot)^*}$
		define an adjunction between the categories $\underline{\mathrm{FODC}}_A^H$ of right $H$-covariant FODCi on $A$ and the category $(\underline{\mathrm{FOVC}}_A^H)^\mathrm{op}$ of right $H$-covariant {FOVCi} on $A$ with opposite morphisms:
		\begin{equation*}
			\begin{tikzcd}
				\underline{\mathrm{FODC}}_A^H \arrow[rr, bend left, "\underline{(\cdot)^*}"] \arrow{ddd}[swap]{\underline{\mathrm{forget}}_\mathrm{d}} & \perp & (\underline{\mathrm{FOVC}}_A^H)^\mathrm{op} \arrow[ll, bend left, "\underline{{}^*(\cdot)}"] \arrow{ddd}{\big(\underline{\mathrm{forget}}_\mathscr{L}\big)^\mathrm{op}}\\
				& & \\
				& & \\
				{}_A\mathcal{M}_A^H \arrow[rr, bend left, "\underline{(\cdot)^*}"] & \perp & ({}_A\mathcal{M}_A^H)^\mathrm{op} \arrow[ll, bend left, "\underline{{}^*(\cdot)}"]
			\end{tikzcd}
		\end{equation*}
		where
		$\underline{\mathrm{forget}}_\mathrm{d}(\Gamma,\mathrm{d})=\Gamma$ and $\underline{\mathrm{forget}}_\mathscr{L}(\mathfrak{X},\mathscr{L})=\mathfrak{X}$
		are the forgetful functors.
	\end{proposition}
	\begin{proof}
		This follows from Proposition \ref{prop:duality} combined with the following observations.
		\begin{enumerate}
			\item[i.)] Given an object $M\in{}_A\mathcal{M}_A^H$ the $A$-bimodules $\underline{\mathrm{Hom}}_A(M,A)$ and ${}_A\underline{\mathrm{Hom}}(M,A)$ are objects in ${}_A\mathcal{M}_A^H$. This follows from the above discussion, further referring to \cite[Section 2.2]{PaoloThomas}. In particular, the functors \eqref{RationalFunctors} are well-defined and, together with \cite[Proposition 9]{PoinsotPorst}, we infer that $\underline{(\cdot)^*}$ in \eqref{eq:DualFuncMorph} is the left adjoint of $\underline{{}^*(\cdot)}$ in \eqref{eq:DualFuncMorph}.
			\item[ii.)] Given a right $H$-covariant FODC $(\Gamma,\mathrm{d})$ on $A$ it follows that the induced right $H$-covariant $A$-bimodule $\mathfrak{X}:=\underline{\mathrm{Hom}}_A(\Gamma,A)$ is a right $H$-covariant FOVC on $A$ with Lie derivative $\mathscr{L}\colon\mathfrak{X}\to\mathrm{End}_\Bbbk(A)$ defined by $\mathscr{L}_X(a):=\langle X,\mathrm{d}a\rangle$ for all $X\in\mathfrak{X}$ and $a\in A$. By following the same lines of Proposition \ref{prop:duality} $i.)$ one shows that $(\mathfrak{X},\mathscr{L})$ is a FOVC on $A$. The latter is right $H$-covariant, since
			\begin{align*}
				\Delta_A(\widetilde{\mathscr{L}}(X\otimes a))
				&=\Delta_A(\mathscr{L}_X(a))\\
				&=\Delta_A(\langle X,\mathrm{d}a\rangle)\\
				&\overset{(*)}{=}\langle X_0,\mathrm{d}(a_0)\rangle\otimes X_1a_1\\
				&=(\widetilde{\mathscr{L}}\otimes\mathrm{id})(\Delta_\otimes(X\otimes a))
			\end{align*}
			for all $X\in\mathfrak{X}$ and $a\in A$. Above we used the right $H$-colinearity of $\mathrm{d}$ and, in $(*)$, we used that the evaluation $\langle\cdot,\cdot\rangle\colon\mathfrak{X}\otimes\Gamma\to A$ is right $H$-colinear, which is the case since
			\begin{align*}
				\langle X_0,\omega_0\rangle\otimes X_1\omega_1
				=X(\omega_0)_0\otimes X(\omega_0)_1S(\omega_1)\omega_2
				=X(\omega)_0\otimes X(\omega)_1
				=\Delta_A(\langle X,\omega\rangle)
			\end{align*}
			for all $X\in\mathfrak{X}$ and $\omega\in\Gamma$, using the definition \eqref{rationalmorh}.
			\item[iii.)] Given a right $H$-covariant FOVC $(\mathfrak{X},\mathscr{L})$ on $A$ it follows that the induced right $H$-covariant $A$-bimodule $\Gamma:={}_A\underline{\mathrm{Hom}}(\mathfrak{X},A)$ is a right $H$-covariant FODC on $A$ with differential $\mathrm{d}\colon A\to\Gamma$ defined for all $X\in\mathfrak{X}$ by $\langle X,\mathrm{d}a\rangle:=\mathscr{L}_X(a)$ for all $a\in A$. By following the same lines of Proposition \ref{prop:duality} $ii.)$ one shows that $(\Gamma,\mathrm{d})$ is a FODC on $A$. The latter is right $H$-covariant since, for all $X\in\mathfrak{X}$ and $a\in A$,
			\begin{align*}
				\langle X,\mathrm{d}(a)_0\rangle\otimes\mathrm{d}(a)_1
				&=\langle X_0,\mathrm{d}a\rangle_0\otimes S^{-1}(X_1)\langle X_0,\mathrm{d}a\rangle_1\\
				&=\mathscr{L}_{X_0}(a)_0\otimes S^{-1}(X_1)\mathscr{L}_{X_0}(a)_1\\
				&=\mathscr{L}_{X_0}(a_0)\otimes S^{-1}(X_2)X_1a_1\\
				&=\mathscr{L}_X(a_0)\otimes a_1\\
				&=\langle X,\mathrm{d}(a_0)\rangle\otimes a_1
			\end{align*}
			holds. This implies the right $H$-covariance of $(\Gamma,\mathrm{d})$.
		\end{enumerate}
		The fact that the restriction and corestriction of $\underline{(\cdot)^*}$ and $\underline{{}^*(\cdot)}$ to right $H$-covariant FODCi and FOVCi form an adjunction follows in complete analogy to Proposition \ref{prop:duality} from the above observations. 
	\end{proof}
	In the case of finitely generated projective modules, the rational morphisms coincide with space of linear maps, i.e., if $M\in{}_A\mathcal{M}_A^H$ is finitely generated projective as a right $A$-module (resp. left $A$-module), then $\underline{\mathrm{Hom}}_A(M,A)=\mathrm{Hom}_A(MA)$ (resp. ${}_A\underline{\mathrm{Hom}}(M,A)={}_A\mathrm{Hom}(M,A)$), which can be shown in complete analogy to \cite[Proposition 2.11]{PaoloThomas}. Thus, Proposition \ref{prop:duality'} and Theorem \ref{dual-cp} imply the following result.
	\begin{corollary}\label{cor:covFGPequ}
		For any right $H$-comodule algebra $A$, the category ${}_{\fgp}\underline{\mathrm{FODC}}^H_A$ of right $H$-covariant FODCi on $A$ which are finitely generated projective as 
		left $A$-modules is equivalent to the category ${}_{\fgp}\underline{\mathrm{FOVC}}^H_A$ of right $H$-covariant FOVCi on $A$ which are finitely generated projective as 
		right $A$-modules, with opposite morphisms:
		$$
		{}_{\fgp}\underline{\mathrm{FODC}}^H_A\cong({}_{\fgp}\underline{\mathrm{FOVC}}^H_A)^\mathrm{op}.
		$$
	\end{corollary}
	There are variants of Proposition \ref{prop:duality'} and Corollary \ref{cor:covFGPequ} for left $H$-covariant and $H$-bicovariant first order differential and vector field calculi. Since the adaptions are straightforward we refrain from spelling them out in detail. We discuss more explicit examples of covariant FOVCi in Section~
	\ref{sec-ex} after we introduce the tool of quantum tangent space.

	\subsection{Derivations via the \DJ ur\dj evi\'c braiding}\label{sec:durd}
	
	In this section we construct a FOVC on any Hopf--Galois extension as the braided derivations with respect to the canonical \DJ ur\dj evi\'c braiding \cite{DurdBraid}.
	
	Since Hopf--Galois extension are appearing throughout this article, we recall them in some detail.
	A \textit{Hopf--Galois extension} \cite{KrTa} is a right $H$-comodule algebra $A$ with right $H$-coaction $\Delta_A\colon A\to A\otimes H$ and coinvariant subalgebra $B:=A^{\mathrm{co}H}:=\{a\in A~|~\Delta_A(a)=a\otimes 1\}\subseteq A$ such that the canonical map
	$$
	\chi\colon A\otimes_BA\to A\otimes H,\qquad a\otimes_Ba'\mapsto aa'_0\otimes a'_1
	$$
	is invertible. For the \textit{translation map} $\tau:=\chi^{-1}\circ(1_A\otimes\cdot)\colon H\to A\otimes_BA$ we employ the common short notation 
	$$
	\tau(h)=:h^{\langle 1\rangle}\otimes_Bh^{\langle 2\rangle}
	$$ 
	for all $h\in H$. It is well-known (see e.g. \cite{Brz}) that the following equations hold for all $h,g\in H$ and $a\in A$:
	\begin{align}
		h^{\langle 1\rangle}h^{\langle 2\rangle}&=\varepsilon(h)1 \label{transl1}\\
		\tau(hg)&=g^{\langle 1\rangle}h^{\langle 1\rangle}\otimes_Bh^{\langle 2\rangle}g^{\langle 2\rangle}\label{transl2}\\
		h^{\langle 1\rangle}\otimes_B(h^{\langle 2\rangle})_0\otimes(h^{\langle 2\rangle})_1&=(h_1)^{\langle 1\rangle}\otimes_B(h_1)^{\langle 2\rangle}\otimes h_2\\
		(h^{\langle 1\rangle})_0\otimes_Bh^{\langle 2\rangle}\otimes(h^{\langle 1\rangle})_1&=(h_2)^{\langle 1\rangle}\otimes_B(h_2)^{\langle 2\rangle}\otimes S(h_1)\label{transl4}\\
		a_0(a_1)^{\langle 1\rangle}\otimes_B(a_1)^{\langle 2\rangle}&=1_A\otimes_Ba\label{transl5}\\
		h^{\langle 1\rangle}(h^{\langle 2\rangle})_0\otimes(h^{\langle 1\rangle})_1
		&=1_A\otimes h
	\end{align}
	Moreover, $b\tau(h)=\tau(h)b$ for all $h\in H$ and $b\in B$. 
	
	Special cases of Hopf--Galois extensions are \textit{cleft extensions} $B\subseteq A$, where $A$ is a right $H$-comodule algebra, $B:=A^{\mathrm{co}H}$ and there exists a right $H$-colinear convolution invertible map $j\colon H\to A$, the \textit{cleaving map}. In fact, in this case $\tau(h):=j^{-1}(h_1)\otimes_Bj(h_2)$ induces an inverse of the Hopf--Galois map, where $j^{-1}\colon H\to A$ denotes the convolution inverse of $j$, i.e. $j(h_1)j^{-1}(h_2)=\varepsilon(h)1=j^{-1}(h_1)j(h_2)$ for all $h\in H$. See \cite[Section 7.2]{mont} for more information on cleft extensions.
	
	From now on until the end of the section we let $B:=A^{\mathrm{co}H}\subseteq A$ be an arbitrary Hopf--Galois extension.
	In \cite{DurdBraid} \DJ ur\dj evi\'c defines an isomorphism $\sigma\colon A\otimes_BA\to A\otimes_BA$ of $B$-bimodules
	\begin{equation}
		\sigma(a\otimes_Ba'):=a_0a'(a_1)^{\langle 1\rangle}\otimes_B(a_1)^{\langle 2\rangle},
	\end{equation}
	which satisfies the braid relations
	$$
	(\sigma\otimes_B\mathrm{id}_A)\circ(\mathrm{id}_A\otimes_B\sigma)\circ(\sigma\otimes_B\mathrm{id}_A)=(\mathrm{id}_A\otimes_B\sigma)\circ(\sigma\otimes_B\mathrm{id}_A)\circ(\mathrm{id}_A\otimes_B\sigma),
	$$
	see also \cite[Section 5.1]{AntEmaTho} for a more recent account.
	The algebra $A$ is braided-commutative with respect to $\sigma$, i.e., $m\circ\sigma=m$, where $m\colon A\otimes_BA\to A$ is the multiplication. Using the short notation
	\begin{equation}
		\sigma(a\otimes_Ba')=:{}_\alpha a'\otimes_B~{}^\alpha a
	\end{equation}
	the braided-commutativity reads ${}_\alpha a'\cdot{}^\alpha a=a\cdot a'$ and one further shows that the equations
	\begin{equation}\label{eq:alphanotation}
		{}_\alpha e\otimes_B{}^\alpha(ac)
		={}_{\alpha\beta}e\otimes_B{}^\alpha a~{}^\beta c,\qquad
		{}_\alpha(ce)\otimes_B{}^\alpha a
		={}_\alpha c~{}_\beta e\otimes_B{}^{\beta\alpha}a
	\end{equation}
	are satisfied for all $a,c,e\in A$.
	
	Using the above ``braiding'' we define a right $A$-action on all right $B$-linear endomorphisms 
	$$
	\mathrm{End}_B(A):=\{X\in\mathrm{End}_\Bbbk(A)~|~X(ab)=X(a)b\text{ for all }a\in A,~b\in B\}
	$$
	of $A$, namely $X\cdot a\in\mathrm{End}_B(A)$, where
	\begin{equation}\label{Durd:RightA}
		(X\cdot a)(a'):=X({}_\alpha a'){}^\alpha a
		=X(a_0a'(a_1)^{\langle 1\rangle})(a_1)^{\langle 2\rangle}
	\end{equation}
	for all $X\in\mathrm{End}_B(A)$ and $a,a'\in A$. Note that this is well-defined precisely because $X\in\mathrm{End}_B(A)$ is right $B$-linear.
	\begin{lemma}\label{lem:rightact}
		The operation \eqref{Durd:RightA} is a right $A$-action on $\mathrm{End}_B(A)$.
	\end{lemma}
	\begin{proof}
		Let $X\in\mathrm{End}_B(A)$ and $a,a'\in A$.
		We have already noted that $(X\cdot a)(a')=X(a_0a'(a_1)^{\langle 1\rangle})(a_1)^{\langle 2\rangle}$ is well-defined by the right $B$-linearity of $X$. Moreover, $X\cdot a\in\mathrm{End}_B(A)$, since
		$$
		(X\cdot a)(a'b)
		=X(a_0a'b(a_1)^{\langle 1\rangle})(a_1)^{\langle 2\rangle}
		=X(a_0a'b(a_1)^{\langle 1\rangle})(a_1)^{\langle 2\rangle}b
		=(X\cdot a)(a')b
		$$
		for all $b\in B$, since $b\tau(h)=\tau(h)b$ for all $h\in H$ and $b\in B$. Clearly $X\cdot 1=X$, since $\Delta_A(1)=1\otimes 1_H$ and $\tau(1_H)=1\otimes_B1$. Finally, using that $\Delta_A$ is an algebra morphism and that $\tau(hg)=g^{\langle 1\rangle}h^{\langle 1\rangle}\otimes_Bh^{\langle 2\rangle}g^{\langle 2\rangle}$ for all $h,g\in H$, we obtain
		\begin{align*}
			((X\cdot a)\cdot a')(a'')
			&=(X\cdot a)(a'_0a''(a'_1)^{\langle 1\rangle})(a'_1)^{\langle 2\rangle}\\
			&=X(a_0a'_0a''(a'_1)^{\langle 1\rangle}(a_1)^{\langle 1\rangle})(a_1)^{\langle 2\rangle}(a'_1)^{\langle 2\rangle}\\
			&\overset{\eqref{transl2}}{=}X((aa')_0a''((aa')_1)^{\langle 1\rangle})((aa')_1)^{\langle 2\rangle}\\
			&=(X\cdot(aa'))(a'')
		\end{align*}
		for all $a''\in A$. Alternatively, one uses ${}_\alpha(a'a'')\otimes_B{}^\alpha a
		={}_\alpha a'~{}_\beta a''\otimes_B{}^{\beta\alpha}a$, as in \eqref{eq:alphanotation}.
		This concludes the proof.
	\end{proof}
	From Lemma \ref{lem:rightact} it follows that $\mathrm{End}_B(A)$ is an $A$-bimodule with respect to the right $A$-action \eqref{Durd:RightA} and the left $A$-action given by left multiplication. We continue by proving that the $A$-subbimodule of $\mathrm{End}_B(A)$ given by the $\sigma$-braided derivations is a first order vector field calculus.
	\begin{theorem}\label{ThmSigmaVC}
		Let $B=A^{\mathrm{co}H}\subseteq A$ be an arbitrary Hopf--Galois extension. There is a FOVC $(\mathfrak{X}_\sigma(A),\mathscr{L})$ on $A$, where
		$$
		\mathfrak{X}_\sigma(A):=\{X\in\mathrm{End}_B(A)~|~X(aa')=X(a)a'+\underbrace{X(a_0a'(a_1)^{\langle 1\rangle})(a_1)^{\langle 2\rangle}}_{=(X\cdot a)(a')}\text{ for all }a,a'\in A\}
		$$
		and $\mathscr{L}\colon\mathfrak{X}_\sigma(A)\to\mathrm{End}_\Bbbk(A)$ is the canonical injection $X\mapsto X$. The $A$-bimodule structure of $\mathfrak{X}_\sigma(A)$ is given by
		$(a\cdot X)(a'):=aX(a')$ and $(X\cdot a)(a')=X(a_0a'(a_1)^{\langle 1\rangle})(a_1)^{\langle 2\rangle}$ for all $a,a'\in A$ and $X\in\mathfrak{X}_\sigma(A)$.
	\end{theorem}
	\begin{proof}
		By the previous lemma $(X\cdot a)(a')=X(a_0a'(a_1)^{\langle 1\rangle})(a_1)^{\langle 2\rangle}$ is a right $A$-action on $\mathrm{End}_B(A)$. We prove that this restricts to a right $A$-action on $\mathfrak{X}_\sigma(A)$. Let $X\in\mathfrak{X}_\sigma(A)$ and $a\in A$. Then, for all $b,c\in A$ we have
		\begin{align*}
			(X\cdot a)(b)c+((X\cdot a)\cdot b)(c)
			&=X(ab)c-X(a)bc+(X\cdot(ab))(c)\\
			&=X(ab)c-X(a)bc+X(abc)-X(ab)c\\
			&=(X\cdot a)(bc),
		\end{align*}
		implying $X\cdot a\in\mathfrak{X}_\sigma(A)$. The left $A$-action obviously closes in $\mathfrak{X}_\sigma(A)$. Thus $\mathfrak{X}_\sigma(A)$ is an $A$-bimodule. The Leibniz rule of $(\mathfrak{X}_\sigma(A),\mathscr{L})$ holds by construction, which finishes the proof of the theorem.
	\end{proof}
	In other words, on every Hopf--Galois extension the right $B$-linear endomorphisms which are braided derivations with respect to the \DJ ur\dj evi\'c braiding form a first order vector field calculus on $A$. Let us make some observations about this in the following.
	\begin{remark}
		\begin{enumerate}
			\item[i.)] Consider a Hopf--Galois extension $B=A^{\mathrm{co}H}\subseteq A$. Any $X\in\mathfrak{X}_\sigma(A)$ vanishes on $B$, i.e.
			$X(B)=0$, since $X(b)=X(1\cdot b)=X(1)b=0$ for all $b\in B$. 
			
			\item[ii.)] Consider a Hopf--Galois extension $B=A^{\mathrm{co}H}\subseteq A$ with $A$ \textit{commutative}. Then, the FOVC of Theorem~\ref{ThmSigmaVC} consists of the right $B$-linear derivations of $A$, i.e. 
			$$
			\mathfrak{X}_\sigma(A)=\mathrm{Der}_B(A):=\{X\in\mathrm{End}_B(A)~|~X(aa')=X(a)a'+aX(a')\text{ for all }a,a'\in A\}.
			$$
			This is the case since for all $X\in\mathfrak{X}_\sigma(A)$ and $a,a'\in A$
			$$
			X(a_0a'(a_1)^{\langle 1\rangle})(a_1)^{\langle 2\rangle}=X(a_0(a_1)^{\langle 1\rangle}a')(a_1)^{\langle 2\rangle}=X(a')a=aX(a')
			$$
			holds by \eqref{transl5} and the commutativity of $A$.
			
			In i.) we have seen that any $X\in\mathfrak{X}_\sigma(A)$ vanishes on $B$. We show that the converse holds in the commutative setting. Let $X\in\mathrm{Der}(A)$ be a derivation on the commutative algebra $A$ and assume that $X(B)=0$. Then, for all $a\in A$ and $b\in B$ we have
			$$
			X(ab)=X(a)b+aX(b)=X(a)b,
			$$
			which proves that $X\in\mathfrak{X}_\sigma(A)$. Thus, in the commutative setting $\mathfrak{X}_\sigma(A)=\{X\in\mathrm{Der}(A)~|~X(B)=0\}$. Note that in the general case ($A$ not necessarily commutative) we cannot make this comparison, since $X\cdot a$ is only defined for $X$ right $B$-linear and thus we cannot speak of a $\Bbbk$-linear $\sigma$-derivation in this case.
			
			\item[iii.)] For any Hopf algebra $H$ we obtain a Hopf--Galois extension $\Bbbk\cong H^{\mathrm{co}H}\subseteq H$ by the existence of the antipode, where $H$ coacts on itself via the coproduct. In this case
			$$
			\mathfrak{X}_\sigma(H)=\{X\in\mathrm{End}_\Bbbk(H)~|~X(hg)=X(h)g+X(h_1gS(h_2))h_3\text{ for all }h,g\in H\}
			$$
			since the translation map reads $\tau(h)=S(h_1)\otimes h_2$ for all $h\in H$ in this case. More in general, for every cleft extension $B\subseteq A$ with cleaving map $j\colon H\to A$ we obtain
			$$
			\mathfrak{X}_\sigma(A)=\{X\in\mathrm{End}_B(A)~|~X(aa')=X(a)a'+X(a_0a'j^{-1}(a_1))j(a_2)\},
			$$
			where $j^{-1}\colon H\to A$ denotes the convolution inverse of $j$.
		\end{enumerate}
	\end{remark}
	
	\section{FOVC on quantum principal bundles}\label{sec-ex}
	
	In the first section we recall the quantum tangent space approach \cite{Woronowicz1989} of Woronowicz in the light of vector field calculi, specializing to bicovariant quantum tangent spaces in Section \ref{sec:bicov}. In particular, we prove a correspondence of (bicovariant) quantum tangent spaces with (bi)covariant first order vector field calculi on Hopf algebras. Then, in Section \ref{fovc-ca-sec}, we extend the previous framework from Hopf algebras to 
	arbitrary Hopf--Galois extensions and obtain certain ``vertical vector field calculi''. Here, verticality can be understood in the sense of (or in fact dual to) \DJ ur\dj evi\'c \cite{DurII}, as explained in Section \ref{sec:duality}. We continue to discuss ``base vector fields'' in Section \ref{base-vf-sec}, leading up to an ``Atiyah sequence'' of vector field calculi on Hopf--Galois extensions in Section \ref{sec:Atiyah}. A class of examples of vector field calculi obeying such conditions is given by cleft extensions, or, equivalently, crossed product algebras. This is discussed in Section \ref{sec:crossed}.
	
	\subsection{FOVC on a Hopf algebra}
	\label{fovc-h}
	
	Let $H$ be a Hopf algebra and  $H^\circ$ its restricted dual 
	(also called finite dual), which naturally inherits a Hopf algebra structure \cite[Theorem 9.1.3]{mont}. 
	In the following we denote the comultiplication, counit and antipode of $H^\circ$ by $\Delta$, $\varepsilon$ and $S$, i.e., using the same symbols as in the case of $H$.
	We have a non-degenerate pairing
	$\langle \cdot , \cdot \rangle_H:H^\circ \times H \to \Bbbk$, $ \langle u , h \rangle_H=u(h)$,
	satisfying
	\begin{equation*}
		\begin{split}
			\langle uv , h \rangle_H&=\langle u , h_1 \rangle_H\langle v , h_2 \rangle_H,\\
			\langle u , hg \rangle_H&=\langle u_1 , h \rangle_H\langle u_2 , g \rangle_H,
		\end{split}\qquad\qquad
		\begin{split}
			\langle u , 1 \rangle_H&=\varepsilon(u),\\
			\langle 1 , h \rangle_H&=\varepsilon(h),
		\end{split}\qquad\qquad
		\begin{split}
			\langle S(u),h\rangle_H=\langle u,S(h)\rangle_H.
		\end{split}
	\end{equation*}
	For a finite-dimensional vector space $\fg \subset (H^\circ)^+=\ker(\varepsilon\colon H^\circ\to\Bbbk)$, we define:
	\begin{equation}\label{def:XH}
		\Chi_H:=H \otimes \fg, \qquad \mathscr{L}\colon\Chi_H \to  \mathrm{End}_\Bbbk(H),
		\quad \mathscr{L}_{h \otimes X}(g):=hg_1X(g_2), \quad h,g \in H, \, X \in \fg. 
	\end{equation}
	The question arises which properties of $\mathfrak{g}$ are required in order to obtain a FOVC structure on $(\mathfrak{X}_H,\mathscr{L})$. This leads to the definition of a quantum tangent space.
	
	\begin{definition}\label{def:QTS}
		Let $\fg \subset  (H^\circ)^+$. We say that $\fg$ is a \textit{quantum tangent space} if
		\begin{enumerate}
			\item[i.)] $\fg$ is finite-dimensional.
			
			\item[ii.)] $\Delta(\fg) \subset \fg \otimes\varepsilon+H^\circ \otimes \fg$.
		\end{enumerate}
	\end{definition}
	
	\begin{remark}
		In the original work \cite[Section 5]{Woronowicz1989}, as well as in \cite[Section 2.3.1]{AschieriThesis}, a quantum tangent space is supposed to satisfy 
		\begin{equation}\label{eq:Liebracket}
			[X,Y]:=X_1YS(X_2)\in\mathfrak{g}\qquad\qquad\text{ for all }X\in H^\circ,Y\in\mathfrak{g}
		\end{equation}
		in addition. In their framework, condition \eqref{eq:Liebracket} ensures that the corresponding differential calculus is bicovariant. We are going to introduce \textit{bicovariant} quantum tangent spaces, with an adjoint \textit{co}action condition replacing \eqref{eq:Liebracket}, in the next section.
		Note that in \cite{AschieriThesis,AschieriSchupp} it is further assumed that $\mathfrak{g}$ generates $H^\circ$ as an algebra, which is in order to mimic the fact that a classical Lie algebra generates its universal enveloping algebra. Definition \ref{def:QTS} is more in the spirit of \cite{HeckKolb} and \cite[Section 14.1.2]{ks}, where covariant calculi are considered instead of bicovariant ones and where we include the possibility that $H^\circ$ exceeds the algebra generated by $\mathfrak{g}$. Further note that for most results the finite dimensionality assumption in Definition \ref{def:QTS} $i.)$ is not essential. Nevertheless, we include it in the definition because all examples we encounter meet this requirement and, as seen in a moment, it turns out to be very convenient to work with a basis of $\mathfrak{g}$.
	\end{remark}
	
	Let $\{\chi_i\}$ be a 
	$\Bbbk$-basis for $\fg$.
	The condition $\Delta(\fg) \subset \fg \otimes\varepsilon+H^\circ \otimes \fg$ implies
	that
	$$
	\Delta(\chi_i)=c^j_i \chi_j \otimes \varepsilon+ f^j_i \otimes \chi_j, \qquad c^j_i \in \Bbbk, \, f^j_i \in H^\circ,
	$$
	and a short calculation, 
	using the coassociativity and counitality of $\Delta\colon H^\circ\to H^\circ\otimes H^\circ$, shows that:
	\begin{equation}\label{eq:fij}
		c^j_i =\de_{ij}, \quad f_i^j(g)f_j^k(k)=f_i^k(gk), \quad f_i^j(1)=\varepsilon(f_i^j)=\de_{ij},
	\end{equation}
	for all $g,k\in H$. Hence
	\beq\label{qu-plane}
	\Delta(\chi_i)=\chi_i \otimes \varepsilon+ f^j_i \otimes \chi_j.
	\eeq

	\medskip
	We have the following result, partly contained 
	in Woronowicz \cite[Section 5]{Woronowicz1989}, though expressed in the language of first order differential calculi.

	\begin{theorem}\label{dual-thm}
		Let $\mathfrak{g}\subseteq(H^\circ)^+$ be finite-dimensional and define $(\mathfrak{X}_H,\mathscr{L})$ as in \eqref{def:XH}.
		Then, the following are equivalent:
		
		\begin{enumerate}
			\item[i.)] $(\mathfrak{X}_H,\mathscr{L})$ is a FOVC.
			
			\item[ii.)] $\mathfrak{g}$ is a quantum tangent space.
		\end{enumerate}
		In this case $(\mathfrak{X}_H,\mathscr{L})$ is left $H$-covariant via the left $H$-coaction $\mathfrak{X}_H\to H\otimes\mathfrak{X}_H$, $h\otimes\chi_i\mapsto h_1\otimes h_2\otimes\chi_i$.
	\end{theorem}
	\begin{proof}
		Let us prove the stated equivalence.
		\begin{enumerate}
			\item[i.) $\Rightarrow$ ii.)] By assumption $(\mathfrak{X},\mathscr{L})$ is a FOVC, which means that $\mathfrak{X}_H=H\otimes\mathfrak{g}$ is not only a left $H$-module, but also a right $H$-module such that
			$$
			\mathscr{L}_{h\otimes\chi_i}(gk)
			=\mathscr{L}_{h\otimes\chi_i}(g)k
			+\mathscr{L}_{(h\otimes\chi_i)\cdot g}(k)
			$$
			for all $h,g,k\in H$. In particular, $(h\otimes\chi_i)\cdot g\in H\otimes\mathfrak{g}$, which implies that there is a $\Bbbk$-linear map $\phi_i^j\colon H\to H$ such that $(h\otimes\chi_i)\cdot g=h\phi_i^j(g)\otimes\chi_j$. Here, $\phi_i^j(g)$ has to be on the right of $h$, since we want this right action to commute with the left action. Moreover, $\phi_i^j(g)\phi_j^k(k)=\phi_i^k(gk)$ and $\phi_i^j(1)=\delta_i^j1_H$ follow from the axioms of a right action. Let us define the linear functionals $f_i^j\in H\to\Bbbk$ by $f_i^j:=\varepsilon\circ\phi_i^j$. Since $\varepsilon$ is an algebra homomorphism we have $f_i^j(g)f_j^k(k)=f_i^k(gk)$ and $f_i^j(1)=\delta_i^j$. This implies that $f_i^j\in H^\circ$.
			
			Then, by the Leibniz rule of $\mathscr{L}$,
			\begin{align*}
				hg_1k_1\chi_i(g_2k_2)
				&=\mathscr{L}_{h\otimes\chi_i}(gk)\\
				&=\mathscr{L}_{h\otimes\chi_i}(g)k
				+\mathscr{L}_{(h\otimes\chi_i)\cdot g}(k)\\
				&=\mathscr{L}_{h\otimes\chi_i}(g)k
				+\mathscr{L}_{h\phi_i^j(g)\otimes\chi_j}(k)\\
				&=hg_1\chi_i(g_2)k+h\phi_i^j(g)k_1\chi_j(k_2)
			\end{align*}
			holds for all $h,g,k\in H$. Setting $h=1$ and applying $\varepsilon$ to the above equality gives
			$$
			\chi_i(gk)=\chi_i(g)\varepsilon(k)+\underbrace{\varepsilon(\phi_i^j(g))}_{=f_i^j(g)}\chi_j(k)
			$$
			or, equivalently,
			$
			\Delta(\chi_i)=\chi_i\otimes\varepsilon+f_i^j\otimes\chi_j,
			$
			which implies that $\mathfrak{g}$ is a quantum tangent space.
			
			\item[ii.) $\Rightarrow$ i.)] If $\Delta(\mathfrak{g})\subseteq \mathfrak{g}\otimes\varepsilon+
			H^\circ\otimes\mathfrak{g}$ there are linear functionals $f_i^j\colon H\to\Bbbk$ with $f_i^j\in H^\circ$ such that
			$$
			\Delta(\chi_i)=\chi_i\otimes\varepsilon+f_i^j\otimes\chi_j.
			$$
			where  $\{\chi_i\}$ is a basis for $\fg$.
			By the coassociativity of $\Delta$ this implies $\Delta(f_i^j)=f_i^k\otimes f_k^j$ and by its counitality we obtain $\varepsilon(f_i^j)=f_i^j(1)=\delta_i^j$. 
			Using these properties we easily check that
			$$
			(h\otimes\chi_i)\cdot g:=hg_1f_i^j(g_2)\otimes\chi_j
			$$
			defines a right $H$-action on $\mathfrak{X}_H=H\otimes\mathfrak{g}$, which commutes with the obvious left $H$-action.
			Moreover,
			\begin{align*}
				\mathscr{L}_{h\otimes\chi_i}(gk)
				&=hg_1k_1\chi_i(g_2k_2)\\
				&=hg_1k_1\big(\chi_i(g_2)\varepsilon(k_2)+f_i^j(g_2)\chi_j(k_2)\big)\\
				&=hg_1\chi_i(g_2)k+hg_1f_i^j(g_2)k_1\chi_j(k_2)\\
				&=\mathscr{L}_{h\otimes\chi_i}(g)k
				+\mathscr{L}_{(h\otimes\chi_i)\cdot g}(k),
			\end{align*}
			which is the Leibniz identity. To finish the proof that
			$(\mathfrak{X}_H,\mathscr{L})$ is a FOVC, we need to show injectivity of $\mathscr{L}$. 
			Fix  an element $u=a^i \otimes \chi_i$ with $a^i\in H$ and
			assume that $\mathscr{L}_u(h)=0$ for all $h \in H$.
			We need to show $u=0$, which is equivalent to $a^i=0$ for all indices $i$. Observe that for all $g,h\in H$ we have $\mathscr{L}_{u\cdot g}(h)=\mathscr{L}_u(gh)-\mathscr{L}_u(g)h=0$ by the Leibniz rule. In particular, for any $k\in H$ it follows that
			\begin{align*}
				0
				&=\mathscr{L}_{u\cdot S(k_1)} k_2
				=\mathscr{L}_{u}(S(k_1)k_2)
				-\mathscr{L}_{u}(S(k_1))k_2
				=0-a^i S(k_1)_1\chi_i(S(k_1)_2) k_2\\
				&=-a^iS(k_2)\chi_i(S(k_1))k_3=-a^i\varepsilon(k_2)\chi_i(S(k_1))
				=-a^i\chi_i(S(k))
			\end{align*}
			Let us now choose $k=S^{-1}(x^j)$, where $\{x^j\}$ is a dual basis for $\{\chi_i\}$, that
			is $\chi_i(x^j)=\delta_i^j$. Then
			$$
			0=a^i\chi_i(S(k))=a^i\chi_i(S(S^{-1}(x^j))=a^j
			$$
			for all indices $j$, thus obtaining injectivity of $\mathscr{L}$.
		\end{enumerate}
		Let us assume that $i.)$ or, equivalently, $ii.)$ holds. It remains to prove that the FOVC $(\mathfrak{X}_H,\mathscr{L})$ is left $H$-covariant with respect to the left $H$-coaction 
		$$
		{}_{\mathfrak{X}_H}\Delta\colon\mathfrak{X}_H\to H\otimes \mathfrak{X}_H,\qquad h\otimes\chi_i\mapsto h_1\otimes h_2\otimes\chi_i.
		$$
		First of all, ${}_{\mathfrak{X}_H}\Delta$ is clearly a left $H$-coaction on $\mathfrak{X}_H$ by the coassociativity and counitality of $\Delta$. Moreover, the $H$-bimodule structure of $\mathfrak{X}_H$ is compatible with ${}_{\mathfrak{X}_H}\Delta$, and thus $\mathfrak{X}_H$ a left $H$-covariant $H$-bimodule, since
		\begin{align*}
			\Delta(h){}_{\mathfrak{X}_H}\Delta(g\otimes\chi^i)\Delta(k)
			&=h_1g_1k_1\otimes h_2\cdot(g_2\otimes\chi^i)\cdot k_2\\
			&=h_1g_1k_1\otimes h_2g_2k_2f^i_j(k_3)\otimes\chi^j\\
			&={}_{\mathfrak{X}_H}\Delta(hgk_1f^i_j(k_2)\otimes\chi^j)\\
			&={}_{\mathfrak{X}_H}\Delta(h\cdot(g\otimes\chi^i)\cdot k)
		\end{align*}
		for all $h,k\in H$ and $g\otimes\chi^i\in\mathfrak{X}_H$.
		Furthermore, denoting the diagonal left $H$-coaction on $\mathfrak{X}_H\otimes H$ by ${}_\otimes\Delta$ and setting $\widetilde{\mathscr{L}}\colon\mathfrak{X}_H\otimes H\to H$, $X\otimes h\mapsto\mathscr{L}_X(h)$ as usual, we obtain
		\begin{align*}
			(\mathrm{id}\otimes\widetilde{\mathscr{L}})({}_\otimes\Delta((h\otimes\chi^i)\otimes g))
			&=h_1g_1\otimes\widetilde{\mathscr{L}}((h_2\otimes\chi^i)\otimes g_2)\\
			&=h_1g_1\otimes\mathscr{L}_{h_2\otimes\chi^i}(g_2)\\
			&=h_1g_1\otimes h_2g_2\chi^i(g_3)\\
			&=\Delta(hg_1\chi^i(g_2))\\
			&=\Delta(\widetilde{\mathscr{L}}((h\otimes\chi^i))\otimes g)
		\end{align*}
		for all $h\otimes\chi^i\in\mathfrak{X}_H$ and $g\in H$. Thus, $(\mathfrak{X}_H,\mathscr{L})$ is left $H$-covariant.
	\end{proof}
	
	To exemplify the notion of left covariant FOVC arising from a quantum tangent space we discuss the example of a three-dimensional vector field calculus on the quantum group $\mathrm{SL}_q(2)$. It turns out to be the dual of the well-known differential calculus of Woronowicz \cite{Woronowicz1989}.
	
	\begin{example}[3D FOVC on $\mathrm{SL}_q(2)$]\label{ex:SL2:FOVC}
		Let $q\in\mathbb{C}$ not be zero and not a root of unity.
		We consider the quantum group $A:=\mathcal{O}_q(\mathrm{SL}_2(\mathbb{C}))$ of $q$-deformed special linear $2\times 2$-matrices. As an algebra $A=\mathbb{C}\langle\alpha,\beta,\gamma,\delta\rangle/I_M$ is generated by generators $\alpha,\beta,\gamma,\delta$ modulo the ideal $I_M$ generated by the Manin relations
		\begin{equation}\label{manin-r}
			\beta\alpha=q\alpha\beta,\quad
			\gamma\alpha=q\alpha\gamma,\quad
			\delta\beta=q\beta\delta,\quad
			\delta\gamma=q\gamma\delta,\quad
			\delta\alpha-\alpha\delta=(q-q^{-1})\beta\gamma,\quad
			\gamma\beta=\beta\gamma
		\end{equation}
		and the quantum determinant relation
		$$
		\alpha\delta-q^{-1}\beta\gamma=1.
		$$
		It is well-known \cite{frt} that $A$ is a Hopf algebra with comultiplication and counit defined on generators by
		$$
		\Delta\begin{pmatrix}
			\alpha & \beta\\
			\gamma & \delta
		\end{pmatrix}=\begin{pmatrix}
			\alpha & \beta\\
			\gamma & \delta
		\end{pmatrix}\otimes\begin{pmatrix}
			\alpha & \beta\\
			\gamma & \delta
		\end{pmatrix}\qquad
		\varepsilon\begin{pmatrix}
			\alpha & \beta\\
			\gamma & \delta
		\end{pmatrix}=\begin{pmatrix}
			1 & 0\\
			0 & 1
		\end{pmatrix}
		$$
		and extended as algebra morphisms and with antipode defined on generators by
		$$
		S\begin{pmatrix}
			\alpha & \beta\\
			\gamma & \delta
		\end{pmatrix}=\begin{pmatrix}
			\delta & -q\beta\\
			-q^{-1}\gamma & \alpha
		\end{pmatrix}
		$$
		and extended as an anti-algebra morphism. There is a quantum tangent space $\mathfrak{g}_A=\mathrm{span}_\mathbb{C}\{\chi_0,\chi_\pm\}$ on $A$, spanned by the three functionals $\chi_0$, $\chi_\pm\in A^\circ$, which are determined on generators by
		$$
		\chi_0(\alpha)=q^{-2}\,,\,\,\,\chi_0(\delta)=-1\,,\,\,\,\chi_+(\gamma)=q^{2}\,,\,\,\,\chi_-(\beta)=q^{-1},
		$$
		where trivial cases are omitted. The values of these functionals on arbitrary elements of $A$ is then determined via the comultiplication
		\begin{equation}
			\Delta(\chi_0)
			=\chi_0\otimes\varepsilon
			+\mathrm{deg}_{q^{-2}}\otimes\chi_0,\qquad
			\Delta(\chi_\pm)
			=\chi_\pm\otimes\varepsilon
			+\mathrm{deg}_{q^{-1}}\otimes\chi_\pm,
		\end{equation}
		where $\mathrm{deg}_{q^n}\colon A\to\mathbb{C}$, $\mathrm{deg}_{q^n}(a):=q^{n|a|}\varepsilon(a)$ for all $n\in\mathbb{Z}$ and where $|a|$ denotes the degree of $a$, which is determined by $|\alpha|=|\gamma|=1$ and $|\beta|=|\delta|=-1$. From Theorem \ref{dual-thm} we thus obtain a left covariant FOVC $(\mathfrak{X}(A),\mathscr{L}^A)$ on $A$, where $\mathfrak{X}(A):=A\otimes\mathfrak{g}_A$.
		The left $A$-module structure is the multiplication on $A$, while the right $A$-module structure is determined by
		\begin{equation}
			\begin{split}
				(a\otimes\chi_0)\cdot a'&=q^{-2|a'|}aa'\otimes \chi_0,\\
				(a\otimes\chi_\pm)\cdot a'&=q^{-|a'|}aa'\otimes\chi_\pm
			\end{split}
		\end{equation}
		for all $a,a'\in A$. Moreover, the Lie derivative is defined by
		\begin{equation}\label{eq:chi0}
			\begin{split}
				\,\,\mathscr{L}^A_{\chi_0}(\alpha)&=q^{-2}\alpha,\\
				\mathscr{L}^A_{\chi_0}(\gamma)&=q^{-2}\gamma,
			\end{split}\qquad
			\begin{split}
				\,\,\,   \mathscr{L}^A_{\chi_0}(\beta)&=-\beta,\\
				\mathscr{L}^A_{\chi_0}(\delta)&=-\delta,
			\end{split}
		\end{equation}
		\begin{equation}
			\begin{split}
				\mathscr{L}^A_{\chi_+}(\alpha)&=q^{2}\beta,\\
				\mathscr{L}^A_{\chi_+}(\gamma)&=q^{2}\delta,
			\end{split}\qquad
			\begin{split}
				\mathscr{L}^A_{\chi_+}(\beta)&=0,\\
				\mathscr{L}^A_{\chi_+}(\delta)&=0,
			\end{split}\qquad\qquad
			\begin{split}
				\mathscr{L}^A_{\chi_-}(\alpha)&=0,\\
				\mathscr{L}^A_{\chi_-}(\gamma)&=0,
			\end{split}\qquad
			\begin{split}
				\mathscr{L}^A_{\chi_-}(\beta)&=q^{-1}\alpha,\\
				\mathscr{L}^A_{\chi_-}(\delta)&=q^{-1}\gamma.
			\end{split}
		\end{equation}
		where, to ease the notation, we write $\mathscr{L}^A_{\chi}$ in place of
		$\mathscr{L}^A_{1\otimes \chi}$, for $\chi \in \fg_A$. One verifies that the dual differential calculus of $(\mathfrak{X}(A),\mathscr{L}^A)$, in the sense of Proposition~\ref{prop:duality'}, is the 3D calculus of Woronowicz \cite{Woronowicz1989}, see also \cite[Example 2.32]{bm}.
	\end{example}
	
	{We have seen that, given a quantum tangent space $\mathfrak{g}$, there is a left covariant FOVC $(\mathfrak{X}_H,\mathscr{L}^H)$ on $H$. It turns out that all left covariant FOVCi $(\mathfrak{X},\mathscr{L})$ on $H$ with finite-dimensional $\Bbbk$-vector subspace of coinvariants ${}^{\mathrm{co}H}\mathfrak{X}:=\{X\in\mathfrak{X}~|~{}_\mathfrak{X}\Delta(X)=1\otimes X\}$ are associated to a quantum tangent space, as in the correspondence of Theorem \ref{dual-thm}. This is the analogue of \cite[Theorem 1.5]{Woronowicz1989} for left covariant vector field calculi on Hopf algebras and it also greatly relies on the Fundamental Theorem of Hopf modules \cite[Theorem~1.9.4]{mont}.}
	\begin{theorem}\label{thm:leftcov=qts}
		Let $H$ be a Hopf algebra. Then there is a bijective correspondence
		\begin{equation}
			\begin{Bmatrix}
				\text{left covariant FOVC}\\
				(\mathfrak{X},\mathscr{L})\text{ on $H$ with }\\
				\dim{}^{\mathrm{co}H}\mathfrak{X}<\infty
			\end{Bmatrix}\overset{1:1}{\longleftrightarrow}\begin{Bmatrix}
				\text{quantum tangent}\\
				\text{spaces in }H^\circ
			\end{Bmatrix}
		\end{equation}
	\end{theorem}
	\begin{proof}
		Consider a left covariant FOVC $(\mathfrak{X},\mathscr{L})$ on $H$ and assume $\dim{}^{\mathrm{co}H}\mathfrak{X}<\infty$. Define $\mathfrak{g}:={}^{\mathrm{co}H}\mathfrak{X}$. Since $\mathscr{L}\colon\mathfrak{X}\to\mathrm{End}_\Bbbk(H)$ is injective we can identify $\mathfrak{g}$ as a subspace of $\mathrm{End}_\Bbbk(H)$. Postcomposing with the counit gives a subspace in $H^*$, i.e., we identify $X\in\mathfrak{g}$ with the $\Bbbk$-linear map $\chi:=\varepsilon\circ\mathscr{L}_X\colon H\to\Bbbk$. Under the assumption that $\mathfrak{g}$ is finite-dimensional, we choose a basis $\{X_i\}$ of $\mathfrak{g}$ with corresponding $\{\chi_i\}$ in $H^*$. On $\mathfrak{g}={}^{\mathrm{co}H}\mathfrak{X}$ there is the adjoint right $H$-action 
		$\lhd\colon\mathfrak{g}\otimes H\to\mathfrak{g}$, $X\lhd h:=S(h_1)Xh_2$, as one easily verifies. By \cite[Lemma 1.6.4]{mont} this induces a left $H^\circ$-coaction $\mathfrak{g}\to H^\circ\otimes\mathfrak{g}$ (here the finite-dimensionality of $\mathfrak{g}$ is crucial). On the basis $X_i$ this coaction reads
		\begin{equation}
			X_i\mapsto f_i^j\otimes X_j
		\end{equation}
		for some $\{f_i^j\}\subseteq H^\circ$, which satisfy \eqref{eq:fij} by the left coaction axioms. Then, for all $h,g\in H$ we have
		\begin{align*}
			\chi_i(hg)
			=\varepsilon(\mathscr{L}_{X_i}(hg))
			=\varepsilon(\mathscr{L}_{X_i}(h)g+\underbrace{\mathscr{L}_{X_i\cdot h}(g)}_{=h_1\mathscr{L}_{S(h_2)\cdot X_i\cdot h_3}(g)})
			=\chi_i(h)\varepsilon(g)+f_i^j(h)\chi_j(g),
		\end{align*}
		which shows $\mathfrak{g}\subseteq H^\circ$ and, more specifically, $\Delta(\mathfrak{g})\subseteq\mathfrak{g}\otimes\varepsilon+H^\circ\otimes\mathfrak{g}$, i.e., that $\mathfrak{g}$ is a quantum tangent space.
		
		The converse implication follows from Theorem \ref{dual-thm}.
		
		It is left to prove that this correspondence is bijective. Let $(\mathfrak{X},\mathscr{L})$ be a left covariant FOVC on $H$ with $\dim{}^{\mathrm{co}H}\mathfrak{X}<\infty$ and consider the corresponding quantum tangent space $\mathfrak{g}={}^{\mathrm{co}H}\mathfrak{X}$. Then $H\otimes\mathfrak{g}$ is a left covariant FOVC on $H$ by Theorem \ref{dual-thm} and we show that $\mathfrak{X}\cong H\otimes\mathfrak{g}$ are isomorphic as left covariant FOVCi. The corresponding isomorphism (given via the Fundamental Theorem of Hopf modules) is 
		$$
		\theta\colon\mathfrak{X}\to H\otimes\mathfrak{g},\qquad\qquad
		X\mapsto X_{-2}\otimes S(X_{-1})\cdot X_0
		$$
		with inverse $H\otimes\mathfrak{g}\to\mathfrak{X}$, $h\otimes X_i\mapsto h\cdot X_i$, where $\{X_i\}$ is a basis of $\mathfrak{g}$. Using that $\mathfrak{X}$ is a left covariant $H$-bimodule it is straightforward to check that $\theta$ is a bijection, $H$-bilinear and left $H$-colinear. The compatibility with the Lie derivatives follows from
		\begin{align*}
			\mathscr{L}^H_{\theta(X)}(h)
			&=\mathscr{L}^H_{X_{-2}\otimes S(X_{-1})\cdot X_0}(h)\\
			&=X_{-2}h_1\varepsilon(\mathscr{L}_{S(X_{-1})\cdot X_0}(h_2))\\
			&=X_{-2}h_1\varepsilon(S(X_{-1})\mathscr{L}_{X_0}(h_2))\\
			&=X_{-1}h_1\varepsilon(\mathscr{L}_{X_0}(h_2))\\
			&\overset{(*)}{=}(\mathscr{L}_X(h))_1\varepsilon((\mathscr{L}_X(h))_2)\\
			&=\mathscr{L}_X(h)
		\end{align*}
		for all $X\in\mathfrak{X}$ and $h\in H$, where in the second equation we used the definition \eqref{def:XH}, as well as the identification of $\mathfrak{g}$ as a subspace of $H^*$, as explained before, and in $(*)$ we used the left $H$-covariance of $(\mathfrak{X},\mathscr{L})$. Conversely, if $\mathfrak{g}\subseteq H^\circ$ is a quantum tangent space and we construct the left covariant FOVC $\mathfrak{X}:=H\otimes\mathfrak{g}$ according to Theorem \ref{dual-thm}, it is immediate to see that ${}^{\mathrm{co}H}\mathfrak{X}=\mathfrak{g}$. This concludes the proof of the theorem.
	\end{proof}
	
	\subsection{Bicovariant quantum tangent spaces}\label{sec:bicov}
	
	Consider a quantum tangent space $\mathfrak{g}$ with basis $\{\chi_i\}$ and the corresponding functionals $\{f_i^j\}$ such that $\Delta(\chi_i)=\chi_i\otimes\varepsilon+f_i^j\otimes\chi_j$. Then $\mathfrak{g}$ becomes a right $H$-module via
	\begin{equation}\label{eq:leftharpoon}
		\leftharpoonup\colon\mathfrak{g}\otimes H\to\mathfrak{g},\qquad\qquad \chi_i\leftharpoonup h:=f_i^j(h)\chi_j,
	\end{equation}
	using \eqref{eq:fij}. Let us further assume that $\mathfrak{g}$ is a right $H$-comodule with coaction $\Delta_\mathfrak{g}\colon\mathfrak{g}\to\mathfrak{g}\otimes H$. The latter is determined on the basis $\{\chi_i\}$ by 
	\begin{equation}\label{eq:gcoact}
		\Delta_\mathfrak{g}(\chi_i)=\chi_j\otimes N_i^j
	\end{equation} 
	for some $N_i^j\in H$ and the coaction axioms of $\Delta_\mathfrak{g}$ are equivalent to the equations
	\begin{align}
		\Delta(N^j_i)
		&=N^j_k\otimes N^k_i,\label{eq:N1}\\
		\varepsilon(N^j_i)
		&=\delta^j_i,\label{eq:N2}
	\end{align}
	for $\{N_i^j\}$. Similarly to the account in \cite[Section 4]{SchYD} we give
	an important characterization of Yetter-Drinfel'd module structures on $\mathfrak{g}$.
	
	\begin{proposition}
		The right $H$-module and $H$-comodule $(\mathfrak{g},\leftharpoonup,\Delta_\mathfrak{g})$ is a Yetter-Drinfel'd module if and only if
		\begin{align}
			h_1f^i_j(h_2)N^k_i
			&=N^i_jf_i^k(h_1)h_2\label{eq:N3}
		\end{align}
		holds for all $h\in H$.
	\end{proposition}
	
	\begin{proof}
		We need to show that $(\mathfrak{g},\leftharpoonup,\Delta_\mathfrak{g})$ is a {Yetter-Drinfel'd module}, i.e.,
		\begin{equation}\label{YD}
			(X\leftharpoonup h_2)_0\otimes h_1(X\leftharpoonup h_2)_1
			=X_0\leftharpoonup h_1\otimes X_1h_2
		\end{equation}
		holds for all $X\in\mathfrak{g}$ and $h\in H$ if and only if $h_1f^i_j(h_2)N^k_i=N^i_jf_i^k(h_1)h_2$ holds for all $h\in H$. It is sufficient to show \eqref{YD} on the basis $\{\chi_i\}$ and elements $h\in H$. There, the left hand side of \eqref{YD} reads $f_i^j(h_2)\chi_k\otimes h_1N_k^j$, while the right hand side reads $f_j^k(h_1)\chi_k\otimes N_i^jh_2$. Thus, the equivalence with \eqref{eq:N3} is shown.
	\end{proof}
	
	We now wish to endow the quantum tangent space $\mathfrak{g}$ with a right $H$-coaction compatible with the right $H$-action \eqref{eq:leftharpoon}, to view it as a Yetter-Drinfel'd module, as in the previous proposition.
	There is a candidate for a suitable right $H$-coaction on a quantum tangent space $\mathfrak{g}$ and it is given as follows.
	
	Define first
	\begin{equation}\label{eq:Adcoact}
		\mathrm{ad}_\mathfrak{g}\colon\mathfrak{g}\to\mathrm{End}_\Bbbk(H),\qquad\qquad\mathrm{ad}_\mathfrak{g}(\chi_i):=(\mathrm{id}\otimes\chi_i)\circ\mathrm{ad},
	\end{equation}
	where $\mathrm{ad}\colon H\to H\otimes H$, $\mathrm{ad}(h):=h_1S(h_3)\otimes h_2$ is the adjoint left $H$-coaction. Note that we have the following inclusions of vector spaces
	$$
	\mathfrak{g}\otimes H\subseteq H^\circ\otimes H\subseteq H^*\otimes H\subseteq\mathrm{End}_\Bbbk(H).
	$$
	Then, the following lemma shows that \eqref{eq:Adcoact} is in fact a natural candidate for a right $H$-coaction if 
	it maps into the appropriate vector subspace.

	\begin{lemma}\label{lem:YD}
		Let $\mathfrak{g}$ be a quantum tangent space and assume that
		\begin{equation}\label{eq:adrestricts}
			\mathrm{ad}_\mathfrak{g}(\mathfrak{g})\subseteq\mathfrak{g}\otimes H.
		\end{equation}
		Then, the corestriction $\mathrm{ad}_\mathfrak{g}\colon\mathfrak{g}\to\mathfrak{g}\otimes H$ is a right $H$-coaction such that $(\mathfrak{g},\leftharpoonup,\mathrm{ad}_\mathfrak{g})$ is a Yetter-Drinfel'd module, where $\leftharpoonup\colon\mathfrak{g}\otimes H\to\mathfrak{g}$ is defined in \eqref{eq:leftharpoon}.
	\end{lemma}
	\begin{proof}
		Under the assumption \ref{eq:adrestricts}, we show that $\mathrm{ad}_\mathfrak{g}\colon\mathfrak{g}\to\mathfrak{g}\otimes H$ is a right $H$-coaction. In fact, using that $\mathrm{ad}\colon H\to H\otimes H$, $h\mapsto h_1S(h_3)\otimes h_2$ is a left $H$-coaction, i.e., $(\mathrm{id}\otimes\mathrm{ad})\circ\mathrm{ad}=(\Delta\otimes\mathrm{id})\circ\mathrm{ad}$ and $(\varepsilon\otimes\mathrm{id})\circ\mathrm{ad}=\mathrm{id}$, we obtain
		\begin{align*}
			(\mathrm{ad}_\mathfrak{g}\otimes\mathrm{id})(\mathrm{ad}_\mathfrak{g}(\chi_i))
			&=(\mathrm{ad}_\mathfrak{g}\otimes\mathrm{id})\circ(\mathrm{id}\otimes\chi_i)\circ\mathrm{ad}\\
			&=(\mathrm{id}\otimes((\mathrm{id}\otimes\chi_i)\circ\mathrm{ad}))\circ\mathrm{ad}\\
			&=(\mathrm{id}\otimes\mathrm{id}\otimes\chi_i)\circ(\mathrm{id}\otimes\mathrm{ad})\circ\mathrm{ad}\\
			&=(\mathrm{id}\otimes\mathrm{id}\otimes\chi_i)\circ(\Delta\otimes\mathrm{id})\circ\mathrm{ad}\\
			&=\Delta\circ(\mathrm{id}\otimes\chi_i)\circ\mathrm{ad}\\
			&=(\mathrm{id}\otimes\Delta)(\mathrm{ad}_\mathfrak{g}(\chi_i))
		\end{align*}
		and 
		$(\mathrm{id}\otimes\varepsilon)(\mathrm{ad}_\mathfrak{g}(\chi_i))
		=\varepsilon\circ(\mathrm{id}\otimes\chi_i)\circ\mathrm{ad}
		=\chi_i\circ(\varepsilon\otimes\mathrm{id})\circ\mathrm{ad}
		=\chi_i$ 
		for all $\chi_i\in\mathfrak{g}$.
		
		It remains to prove that the right $H$-action \eqref{eq:leftharpoon} is compatible with the right $H$-coaction $\mathrm{ad}_\mathfrak{g}$ in the sense that \eqref{YD} is satisfied. Let $h\in H$ and consider $\{N_i^j\}\subseteq H$ corresponding to $\Delta_\mathfrak{g}=\mathrm{ad}_\mathfrak{g}$ as in \eqref{eq:gcoact}. To show that
		\begin{align*}
			(\chi_i\leftharpoonup h_2)_0\otimes h_1(\chi_i\leftharpoonup h_2)_1
			=f_i^j(h_2)\chi_j\otimes h_1N_i^j
		\end{align*}
		coincides with
		\begin{align*}
			\chi_j\leftharpoonup h_1\otimes N_i^jh_2
			=f_j^k(h_1)\chi_k\otimes N_i^jh_2
		\end{align*}
		we evaluate these two elements of $\mathfrak{g}\otimes H\subseteq H^\circ\otimes H$ at $g\in H$, leading to
		\begin{align*}
			f_i^j(h_2)\chi_j(g)h_1N_i^j
			=h_1f_i^j(h_2)g_1S(g_3)\chi_i(g_2)
		\end{align*}
		and
		\begin{align*}
			f_j^k(h_1)\chi_k(g)N_i^jh_2
			&\overset{\eqref{qu-plane}}{=}(\chi_j(h_1g)-\chi_j(h_1)\varepsilon(g))N_i^jh_2\\
			&=(h_1g_1S(h_3g_3)\chi_i(h_2g_2)-h_1S(h_3)\chi_i(h_2)\varepsilon(g))h_4\\
			&=h_1g_1S(g_3)\chi_i(h_2g_2)-h_1\chi_i(h_2)\varepsilon(g)\\
			&\overset{\eqref{qu-plane}}{=}h_1g_1S(g_2)\chi_i(h_2)
			+h_1g_1S(g_3)f_i^j(h_2)\chi_j(g_2)
			-h_1\chi_i(h_2)\varepsilon(g)\\
			&=h_1f_i^j(h_2)g_1S(g_3)\chi_j(g_2).
		\end{align*}
		Since they coincide we conclude the proof of the lemma.
	\end{proof}
	
	The previous discussion motivates the following definition.
	\begin{definition}\label{def:bicovQTS}
		Let $\mathfrak{g}\subseteq H^\circ$ be a quantum tangent space with induced right $H$-action \eqref{eq:leftharpoon} and consider the $\Bbbk$-linear map $\mathrm{ad}_\mathfrak{g}\colon\mathfrak{g}\to\mathrm{End}_\Bbbk(H)$ defined in \eqref{eq:Adcoact}. We call $\mathfrak{g}$ a \textit{bicovariant} quantum tangent space if 
		$$
		\mathrm{ad}_\mathfrak{g}(\mathfrak{g})\subseteq\mathfrak{g}\otimes H
		$$
		holds. In this case we denote the corestriction by $\mathrm{ad}_\mathfrak{g}\colon\mathfrak{g}\to\mathfrak{g}\otimes H$ by abuse of notation.
	\end{definition}
	\begin{corollary}
		For bicovariant quantum tangent space $\mathfrak{g}$ we have
		\begin{equation}\label{eq:N4}
			h_1\chi_i(h_2)=N_i^j\chi_j(h_1)h_2
		\end{equation}
		for all $h\in H$, where $\{N_i^j\}$ correspond to $\Delta_\mathfrak{g}:=\mathrm{ad}_\mathfrak{g}$ via \eqref{eq:gcoact}.
	\end{corollary}
	
	\begin{proof}
		This follows directly from the computation
		\begin{align*}
			N_i^j\chi_j(h_1)h_2
			=h_1S(h_3)\chi_i(h_2)h_4
			=h_1\chi_i(h_2)
		\end{align*}
		for all $h\in H$.
	\end{proof}
	
	Given a bicovariant quantum tangent space $\mathfrak{g}$ it follows from Lemma \ref{lem:YD} that $(\mathfrak{g},\leftharpoonup,\mathrm{ad}_\mathfrak{g})$ is a Yetter-Drinfel'd module, with right $H$-action $\leftharpoonup\colon\mathfrak{g}\otimes H\to\mathfrak{g}$ defined in \eqref{eq:leftharpoon} and right $H$-coaction $\mathrm{ad}_\mathfrak{g}\colon\mathfrak{g}\to\mathfrak{g}\otimes H$ induced by \eqref{eq:Adcoact}. Let $\{N_i^j\}\subseteq H$ be the elements determining the coaction $\Delta_\mathfrak{g}=\mathrm{ad}_\mathfrak{g}$ according to \eqref{eq:gcoact}. Since $\mathfrak{g}$ is in particular a quantum tangent space, we infer from Theorem \ref{dual-thm} that $\mathfrak{X}_H:=H\otimes\mathfrak{g}$ is a left covariant FOVC with Lie derivative defined in \eqref{def:XH}. We define
	\begin{equation}\label{XHright}                 
		\Delta_{\mathfrak{X}_H}\colon\mathfrak{X}_H\to\mathfrak{X}_H\otimes H,\qquad
		h\otimes\chi_i\mapsto h_1\otimes\chi_j\otimes h_2N^j_i.
	\end{equation}
	As a diagonal coaction, \eqref{XHright} is a right $H$-comodule structure on $\mathfrak{X}_H$. We further prove that $(\mathfrak{X}_H,\mathscr{L})$ is a bicovariant FOVC if endowed with this coaction and that the bicovariance is in fact equivalent to $\mathfrak{g}$ being a bicovariant quantum tangent space.
	
	\begin{theorem}\label{thm:bicovVF}
		For a quantum tangent space $\mathfrak{g}\subseteq H^\circ$ the following are equivalent.
		\begin{enumerate}
			\item[i.)] $\mathfrak{g}$ is a bicovariant quantum tangent space.
			\item[ii.)] $(\mathfrak{X}_H,\mathscr{L})$ is a bicovariant FOVC.
			\item[iii.)] $(\mathfrak{g},\leftharpoonup,\mathrm{ad}_\mathfrak{g})$ is a Yetter-Drinfel'd module.
		\end{enumerate}
		If one, and thus all, of the above conditions hold, \eqref{XHright} is the right $H$-coaction on $\mathfrak{X}_H$.
	\end{theorem}
	\begin{proof}
		We start with the equivalence of $i.)$ and $ii.)$. Building on Theorem \ref{dual-thm} it is left to prove that, given a quantum tangent space $\mathfrak{g}$, \eqref{XHright} is a right $H$-coaction such that $\mathfrak{X}_H$ is a bicovariant $H$-bimodule and $\mathscr{L}$ is compatible with $\Delta_{\mathfrak{X}_H}$ if and only if the quantum tangent space $\mathfrak{g}$ is bicovariant, in addition.
		
		Assume that $\mathfrak{g}$ is a bicovariant quantum tangent space. Then \eqref{XHright} is a right $H$-coaction, since it is a diagonal coaction and $\mathrm{ad}_\mathfrak{g}$ is right $H$-coaction according to Lemma \ref{lem:YD}.
		Furthermore, the right $H$-coaction $\Delta_{\mathfrak{X}_H}$ is compatible with the $H$-bimodule structure of $\mathfrak{X}_H$, namely
		\begin{align*}
			\Delta_{\mathfrak{X}_H}(h\cdot(g\otimes\chi_i)\cdot h')
			&=\Delta_{\mathfrak{X}_H}(hgh'_1f_i^j(h'_2)\otimes\chi_j)\\
			&=h_1g_1h'_1\otimes\chi_k\otimes h_2g_2h'_2f_i^j(h'_3)N^k_j\\
			&\overset{\eqref{eq:N3}}{=}h_1g_1h'_1f^k_j(h'_2)\otimes\chi_k\otimes h_2g_2N^j_ih'_3\\
			&=h_1\cdot(g_1\otimes\chi_j)\cdot h'_1\otimes h_2g_2N^i_jh'_2\\
			&=\Delta(h)\Delta_{\mathfrak{X}_H}(g\otimes\chi_i)\Delta(h')
		\end{align*}
		for all $h,h'\in H$ and $g\otimes\chi^i\in\mathfrak{X}_H$, where we can employ \eqref{eq:N3} since $(\mathfrak{g},\leftharpoonup,\mathrm{ad}_\mathfrak{g})$ is a Yetter-Drinfel'd module, again by Lemma \ref{lem:YD}. The right $H$-coaction $\Delta_{\mathfrak{X}_H}$ clearly commutes with the left $H$-coaction on $\mathfrak{X}_H$ defined in Theorem \ref{dual-thm}, implying that $\mathfrak{X}_H$ is a bicovariant $H$-bimodule. We already showed in Theorem \ref{dual-thm} that $\mathscr{L}$ is injective, satisfies the Leibniz rule and is compatible with the left $H$-coaction. For the remaining compatibility with the right $H$-coaction we calculate
		\begin{align*}
			(\widetilde{\mathscr{L}}\otimes\mathrm{id})(\Delta_\otimes((h\otimes\chi_i)\otimes g))
			&=\widetilde{\mathscr{L}}((h_1\otimes\chi_j)\otimes g_1)\otimes h_2N^j_ig_2\\
			&=\mathscr{L}_{h_1\otimes\chi_j}(g_1)\otimes h_2N^j_ig_2\\
			&=h_1g_1\chi_j(g_2)\otimes h_2N^j_ig_3\\
			&\overset{\eqref{eq:N4}}{=}h_1g_1\otimes h_2g_2\chi_i(g_3)\\
			&=\Delta(\mathscr{L}_{h\otimes\chi_i}(g))\\
			&=\Delta(\widetilde{\mathscr{L}}((h\otimes\chi_i)\otimes g))
		\end{align*}
		for all $h,g\in H$.
		
		Conversely, assume that $\mathfrak{g}$ is a quantum tangent space endowed with a right $H$-coaction $\Delta_\mathfrak{g}\colon\mathfrak{g}\to\mathfrak{g}\otimes H$ (with $\{N_i^j\}\subseteq H$ the associated elements) such that $(\mathfrak{X}_H,\mathscr{L})$ is not only a left covariant, but even a bicovariant FOVC on $H$ with respect to the diagonal right coaction \eqref{XHright}. We show that this implies $\Delta_\mathfrak{g}=\mathrm{ad}_\mathfrak{g}$, which means that $\mathfrak{g}$ is a bicovariant quantum tangent space. By assumption, $\mathscr{L}$ is compatible with $\Delta_{\mathfrak{X}_H}$, i.e.
		$$
		h_1\chi_j(h_2)\otimes N_i^jh_3
		=\mathscr{L}_{\chi_j}(h_1)\otimes N_i^jh_2
		=\Delta(\mathscr{L}_{\chi_i}(h))
		=h_1\otimes h_2\chi_i(h_3)
		$$
		for all $h\in H$. Applying $\varepsilon\otimes\mathrm{id}$ to the above equation gives
		$
		N_i^j\chi_j(h_2)h_3=h_1\chi_i(h_2)
		$
		for all $h\in H$. This equation is equivalent to
		$$
		N_i^j\chi_j(h)=h_1S(h_3)\chi_i(h_2)
		$$
		for all $h\in H$ via a standard convolution inverse argument (consider $h_1\otimes S(h_2)$, apply the left and right hand side of the former to the first tensor, respectively, and multiply afterwards to obtain the latter formula). On the other hand, evaluating $\mathrm{ad}_\mathfrak{g}\colon\mathfrak{g}\to\mathrm{End}_\Bbbk(H)$ at $\chi_i\in\mathfrak{g}$ the outcome at $h\in H$ gives
		$$
		\mathrm{ad}_\mathfrak{g}(\chi_i)(h)
		=h_1S(h_3)\chi_i(h_2)
		$$
		by the very definition \eqref{eq:Adcoact}. Thus $\mathrm{ad}_\mathfrak{g}(\chi_i)(h)=\Delta_\mathfrak{g}(\chi_i)(h)$ for all $h\in H$, which means that $\mathrm{ad}_\mathfrak{g}(\chi_i)$ and $\Delta_\mathfrak{g}(\chi_i)$ are equal as elements in $\mathrm{End}_\Bbbk(H)$. Since $\Delta_\mathfrak{g}(\chi)\in\mathfrak{g}\otimes H\subseteq\mathrm{End}_\Bbbk(H)$ by assumption it follows that $\mathrm{ad}_\mathfrak{g}(\chi_i)\in\mathfrak{g}\otimes H\subseteq\mathrm{End}_\Bbbk(H)$, i.e. $\mathrm{ad}_\mathfrak{g}(\mathfrak{g})\subseteq\mathfrak{g}\otimes H$ and $\mathfrak{g}$ is a bicovariant quantum tangent space.
		
		For the remaining equivalence of $i.)$ and $iii.)$, we have already shown in Lemma \ref{lem:YD} that $i.)$ implies $iii.)$. If, conversely, $(\mathfrak{g},\leftharpoonup,\mathrm{ad}_\mathfrak{g})$ is a Yetter-Drinfel'd module, the adjoint coaction $\mathrm{ad}_\mathfrak{g}\colon\mathfrak{g}\to\mathfrak{g}\otimes H$ is well-defined by assumption and thus $\mathfrak{g}$ is a bicovariant quantum tangent space.
	\end{proof}
	
	\begin{remark}\label{rem:bracket}
		Note that the bicovariance condition $\mathrm{ad}_\mathfrak{g}\subseteq\mathfrak{g}\otimes H$ of a quantum tangent space $\mathfrak{g}\subseteq H^\circ$ implies 
		\begin{equation}\label{eq:Liebracket'}
			[X,Y]:=X_1YS(X_2)\in\mathfrak{g}\qquad\text{ for all }X\in H^\circ, Y\in\mathfrak{g}.
		\end{equation}
		In fact, every right $H$-coaction $M\to M\otimes H$, $m\mapsto m_0\otimes m_1$ induces a left $H^*$-action $H^*\otimes M\to M$, $\alpha\otimes m\mapsto m_0\alpha(m_1)$, see for example \cite[Lemma 1.6.4]{mont} and, in particular, a left $H^\circ$-action. In our case the induced left $H^\circ$-action $[\cdot,\cdot]\colon H^\circ\otimes\mathfrak{g}\to\mathfrak{g}$ reads $[X,\chi_i]=\chi_jX(N_i^j)$, which evaluates to
		\begin{equation}
			[X,\chi_i](h)
			=\chi_j(h)X(N_i^j)
			=X(h_1S(h_3))\chi_i(h_2)
			=(X_1\chi_iS(X_2))(h)
		\end{equation}
		for all $X\in H^\circ$ and $h\in H$. Thus proving that \eqref{eq:adrestricts} implies \eqref{eq:Liebracket'}. Even though $\mathfrak{g}$ is finite-dimensional, the converse is not true in general: given a quantum tangent space $\mathfrak{g}$, condition \eqref{eq:Liebracket'} does not imply \eqref{eq:adrestricts} in general. Namely, again by \cite[Lemma 1.6.4]{mont}, the left $H^\circ$-action $[\cdot,\cdot]\colon H^\circ\otimes\mathfrak{g}\to\mathfrak{g}$ induces a right $(H^\circ)^\circ$-coaction $\mathfrak{g}\to\mathfrak{g}\otimes(H^\circ)^\circ$ since $\mathfrak{g}$ is finite-dimensional. However, in general $(H^\circ)^\circ\nsubseteq H$ and thus it is not clear if the latter corresponds to a right $H$-coaction. If the inclusion (or equality) holds, the conditions \eqref{eq:adrestricts} and \eqref{eq:Liebracket'} are equivalent.
		
		In the literature, condition \eqref{eq:Liebracket'} usually appears to describe bicovariant quantum tangent spaces, see e.g. \cite[Section 5]{Woronowicz1989}, \cite[Section 14.2.3]{ks} and \cite[Section 2.3.1]{AschieriThesis}. There, vector fields are defined by the functionals vanishing on the ideal characterizing the corresponding bicovariant first order differential calculus. And in fact, condition \eqref{eq:Liebracket'} characterizes the bicovariance of the differential calculus, see \cite[Corollary 14.9]{ks}. In this work we facilitate the vector field centric view, where condition \eqref{eq:adrestricts} is equivalent to bicovariance of the vector field calculus, as shown in Theorem \ref{thm:bicovVF}. Although, as we argued before, both concepts are closely related. This gives an alternative, differential form independent approach to \cite[Section 4]{AschieriSchupp}, where we further provide compatibility of the Lie derivative of arbitrary vector fields, rather than for the one restricted to invariant vector fields.
	\end{remark}
	
	As in the previous section, we show that bicovariant quantum tangent spaces are the structures characterizing bicovariant vector field calculi (with finite-dimensional coinvariant part).
	\begin{theorem}
		There is a bijective correspondence
		\begin{equation}
			\begin{Bmatrix}
				\text{bicovariant FOVC}\\
				(\mathfrak{X},\mathscr{L})\text{ on $H$ with }\\
				\dim{}^{\mathrm{co}H}\mathfrak{X}<\infty
			\end{Bmatrix}\overset{1:1}{\longleftrightarrow}\begin{Bmatrix}
				\text{bicovariant quantum}\\
				\text{tangent spaces in }H^\circ
			\end{Bmatrix}
		\end{equation}
	\end{theorem}
	\begin{proof}
		From Theorem \ref{thm:leftcov=qts} we know that left covariant FOVCi on $H$ correspond bijectively to quantum tangent spaces. It remains to prove that the quantum tangent space of a bicovariant FOVC is bicovariant and vice versa.
		Let $(\mathfrak{X},\mathscr{L})$ be a bicovariant FOVC on $H$ with $\dim{}^{\mathrm{co}H}\mathfrak{X}<\infty$ and consider the corresponding quantum tangent space $\mathfrak{g}:={}^{\mathrm{co}H}\mathfrak{X}$ according to Theorem \ref{thm:leftcov=qts}. Consider $h\in H$ and an element $X\in\mathfrak{g}$ with corresponding $\chi:=\varepsilon\circ\mathscr{L}_X\in H^\circ$. By the bicovariance of $(\mathfrak{X},\mathscr{L})$ we have
		$$
		X_{-1}h_1\otimes\mathscr{L}_{X_0}(h_2)\otimes S(X_1h_3)
		=(\Delta\otimes S)\Delta(\mathscr{L}_X(h)),
		$$
		which, recalling that $X_{-1}\otimes X_0=1\otimes X$, is equivalent to
		$$
		\underbrace{h_1\otimes\mathscr{L}_{X_0}(h_2)\otimes S(X_1h_3)X_2}_{=h_1\otimes\mathscr{L}_X(h_2)\otimes S(h_3)}
		=(\mathscr{L}_{X_0}(h))_1\otimes(\mathscr{L}_{X_0}(h))_2\otimes S((\mathscr{L}_{X_0}(h))_3)X_1
		$$
		via a standard convolution inverse argument. Multiplying the first and third tensor factor implies
		$$
		h_1S(h_3)\otimes\mathscr{L}_X(h_2)
		=(\mathscr{L}_{X_0}(h))_1S((\mathscr{L}_{X_0}(h))_3)X_1\otimes(\mathscr{L}_{X_0}(h))_2.
		$$
		Employing this equality in the third equation of
		$$
		\mathrm{ad}_\mathfrak{g}(\chi)(h)
		=h_1S(h_3)\chi(h_2)
		=h_1S(h_3)\varepsilon(\mathscr{L}_X(h_2))
		=\varepsilon(\mathscr{L}_{X_0}(h))X_1
		=\chi_0(h)\chi_1
		$$
		implies that $\mathrm{ad}_\mathfrak{g}(\chi)=\chi_0\otimes \chi_1\in\mathfrak{g}\otimes H$ and thus that $\mathfrak{g}$ is a bicovariant quantum tangent space.
		The converse has been shown in Theorem \ref{thm:bicovVF}.
	\end{proof}
	
	The following example specializes Example \ref{graph-calculus} to the group case. It is the vector field calculus analogue of \cite[Proposition 1.52]{bm}.
	
	\begin{example}[Bicovariant calculi on finite groups]
		Let $G$ be a finite group with neutral element $e\in G$ and consider a \textit{Cayley graph on $G$}. The latter is a directed graph with vertices given by the elements of $G$, without self-loops, with at most one edge $x\to y$ from $x$ to $y$ for all vertices $x,y\in G$ and with a subset $\mathcal{C}\subseteq G$ such that $e\notin\mathcal{C}$ and such that $x\to y$ is an edge of the graph if and only if there is an element $g\in\mathcal{C}$ such that $y=xg$. Following Example \ref{graph-calculus} there is a FOVC $(\mathfrak{X},\mathscr{L})$ on the algebra $H=\Bbbk(G)$ of $\Bbbk$-valued functions on $G$ given by $\mathfrak{X}:=\mathrm{span}_\Bbbk\{\chi_{x\to y}\}$, where we assign a basis vector field to each edge $x\to y$ of the graph. The $H$-bimodule structure and Lie derivative are defined in \eqref{graph:bimod} and \eqref{graph:LieDer}, respectively. Note that $H$ is a Hopf algebra with comultiplication, counit and antipode defined on the basis $\{\delta_g\}_{g\in G}$ by
		$$
		\Delta(\delta_g)=\sum_{h\in G}\delta_h\otimes\delta_{h^{-1}g},\qquad
		\varepsilon(\delta_g)=\delta_{g,e},\qquad
		S(\delta_g)=\delta_{g^{-1}}.
		$$
		Define $\chi_g:=\sum_{x\in G}\chi_{x\to xg}$ for all $g\in\mathcal{C}$ and let $\mathfrak{g}:=\mathrm{span}_\Bbbk\{\chi_g~|~g\in\mathcal{C}\}$. Then
		$$
		H\otimes\mathfrak{g}\xrightarrow{\cong}\mathfrak{X},\qquad \delta_x\otimes\chi_g\mapsto\delta_x\chi_g=\chi_{x\to xg}
		$$
		is an isomorphism with inverse $\chi_{x\to xg}\mapsto\delta_x\otimes\chi_g$. The above becomes an isomorphism of $H$-bimodules if we endow $H\otimes\mathfrak{g}$ with the obvious left $H$-action and right $H$-action
		$$
		(\delta_x\otimes\chi_g)\cdot \delta_h:=\delta_{h,xg}\delta_x\otimes\chi_{xg}.
		$$
		We verify that $\mathfrak{g}$ is a quantum tangent space and thus $(\mathfrak{X},\mathscr{L})$ is a left covariant FOVC on $H$ according to Theorem~\ref{dual-thm}. First of all, we identify $\mathfrak{g}$ as a subspace of $H^*$ via 
		$$
		\chi_g(f):=f(e)-f(g)
		$$
		for all $g\in\mathcal{C}$ and $f\in H$. This is consistent with the Lie derivative \eqref{graph:LieDer} as shown by 
		\begin{align*}
			(\mathrm{id}\otimes\chi_g)(\Delta(\delta_y))
			&=\sum_{z\in G}\delta_z\chi_g(\delta_{z^{-1}y})
			=\sum_{z\in G}(\delta_{z^{-1}y}(e)-\delta_{z^{-1}y}(g))\delta_z
			=\delta_y-\delta_{g^{-1}y}
			=\mathscr{L}_{\sum_{x\in G}\chi_{x\to xg}}(\delta_y)
		\end{align*}
		for all $g,y\in G$.
		Since $H$ is finite-dimensional its finite dual coincides with the linear dual. In particular, $H^*=H^\circ$ is a Hopf algebra with counit given by the evaluation at the unit function $1_H\colon H\to\Bbbk$, $1_H(f):=1_\Bbbk$ and coproduct $\Delta\colon H^*\to H^*\otimes H^*$ determined for $\alpha\in H^*$ by $\Delta(\alpha)(f\otimes f')=\alpha(ff')$ for all $f,f'\in H$. Elements of $\mathfrak{g}$ clearly vanish when evaluated at the unit function and their coproduct reads
		\begin{equation}
			\Delta(\chi_g)=\chi_g\otimes\varepsilon+\mathrm{ev}_g\otimes\chi_g\in\mathfrak{g}\otimes\varepsilon+ H^\circ\otimes\mathfrak{g}
		\end{equation}
		for all $g\in\mathcal{C}$ and where $\mathrm{ev}_g\colon H\to\Bbbk$ is defined by $\mathrm{ev}_g(f):=f(g)$. In fact,
		\begin{align*}
			\chi_g(f)\varepsilon(f')+\mathrm{ev}_g(f)\chi_g(f')
			&=(f(e)-f(g))f'(e)+f(g)(f'(e)-f'(g))\\
			&=f(e)f'(e)-f(g)f'(g)\\
			&=\chi_g(ff')
		\end{align*}
		for all $f,f'\in H$. Thus, $\mathfrak{g}$ is a quantum tangent space.
		
		Assume further that $\mathcal{C}$ is closed under adjunction, i.e., that for all $x\in G$ one has $xgx^{-1}\in\mathcal{C}$ if $g\in\mathcal{C}$. We show that in this case $\mathrm{ad}_\mathfrak{g}(\mathfrak{g})\subseteq\mathfrak{g}\otimes H$ and thus $(\mathfrak{X},\mathscr{L})$ is a bicovariant FOVC on $H$ according to Theorem \ref{thm:bicovVF}.
		The claim is that $\sum_{x\in G}\chi_{xgx^{-1}}\otimes\delta_x\in\mathfrak{g}\otimes H$ can be identified with $\mathrm{ad}_\mathfrak{g}(\chi_g)=(\mathrm{id}\otimes\chi_g)\circ\mathrm{ad}\in\mathrm{End}_\Bbbk(H)$.
		To this end, let $h\in G$ and calculate
		\begin{align*}
			\sum_{x\in G}\chi_{xgx^{-1}}(\delta_h)\delta_x
			=\sum_{x\in G}(\delta_h(e)-\delta_h(xgx^{-1}))\delta_x,
		\end{align*}
		which equals
		\begin{align*}
			\mathrm{ad}_\mathfrak{g}(\chi_g)(\delta_h)
			=\sum_{x,y\in G}\delta_x\delta_{h^{-1}y}\chi_g(\delta_{x^{-1}y})
			=\sum_{x\in G}\delta_x\chi_g(\delta_{x^{-1}hx})
			=\sum_{x\in G}(\delta_{x^{-1}hx}(e)-\delta_{x^{-1}hx}(g))\delta_x
		\end{align*}
		since $\delta_{x^{-1}hx}(e)$ is non-zero precisely if $h=e$ and $\delta_{x^{-1}hx}(g)$ is non-zero precisely if $h=xgx^{-1}$.
		Thus, the claim follows and $(\mathfrak{X},\mathscr{L})$ is a bicovariant FOVC on $H$.
		The previous consideration makes it clear that the condition of $\mathcal{C}$ being closed under adjunction is also necessary for bicovariance.
	\end{example}
	Recall that in Example \ref{ex:SL2:FOVC} we built a left covariant FOVC on the quantum group $\cO_q(\rSL_2)$. In the following example, we continue to build a bicovariant FOVC on $\cO_q(\rSL_2)$ instead.
	\begin{example}\label{ex-pair1}
		For a complex number $q$ with $q\neq 0$ and not a root of unity, let $A=\cO_q(\rSL_2)$ with its finite dual $A^\circ=U_q\mathfrak{sl}(2)$. The latter is the Hopf algebra generated by $\{
		E,F,K,K^{-1}\}$ modulo the relations
		$$
		\begin{array}{ccc}
			KE=q^{-1}EK,      & KK^{-1}=1=K^{-1}K,  & KF=q FK,\\[5mm]
			& EF-FE=\frac{K^2-K^{-2}}{\lambda}, &
		\end{array}
		$$
		where $\lambda=q^{-1}-q$. The comultiplication and antipode of $A^\circ$ are determined on generators by
		\begin{equation*}
			\begin{split}
				\Delta(K^{\pm 1})
				&=K^{\pm 1}\otimes K^{\pm 1},\\
				S(K^{\pm 1})
				&=K^{\mp 1},
			\end{split}\qquad
			\begin{split}
				\Delta(E)
				&=E\otimes K+K^{-1}E,\\
				S(E)
				&=-q^{-1}E,
			\end{split}\qquad
			\begin{split}
				\Delta(F)
				&=F\otimes K^{-1}+K\otimes F,\\
				S(F)
				&=-qF,
			\end{split}
		\end{equation*}
		see \cite[Section 3.1.2]{ks} with convention $q\mapsto q^{-1}$. The dual pairing of Hopf algebras is determined on generators by
		$
		K(\alpha)=q^{\frac 1 2}\,,\,\, K(\delta)=q^{-\frac 1 2}\,,\, E(\gamma)=1=F(\beta),
		$
		where the trivial relations are omitted. We now define
		\begin{equation}\label{tg-space-sl}
			\fg=\mathrm{span}\{\chi_1,\chi_2,\chi_3,\chi_4\}\subseteq A^\circ,
		\end{equation}
		where
		\begin{equation}\label{chis}
			\begin{array}{rcl}
				\chi_1&:=&\frac{1}{\lambda}(K^{-2}+q\lambda^2 EF-\epsilon) \\[3mm]
				\chi_2&:=&-q^{\frac 1 2}EK \\[3mm]
				\chi_3&:=&-q^{\frac 1 2}KF\\[3mm]
				\chi_4&:=&\frac{1}{\lambda}(K^2-\epsilon). 
			\end{array}
		\end{equation}
		The coproducts of these elements read
		\begin{equation}\label{coaction-sl}
			\begin{array}{rcl}
				\Delta\chi_1&=&\chi_1\otimes \epsilon+K^{-2}\otimes \chi_1- q^{\frac 1 2}\lambda EK^{-1}\otimes \chi_3- q^{\frac 1 2}\lambda K^{-1}F\otimes \chi_2+q\lambda^2 EF   \otimes\chi_4 \\[3mm]
				\Delta \chi_2&=& \chi_2\otimes \epsilon+\epsilon\otimes \chi_2+\lambda\chi_2\otimes \chi_4 \\[3mm]
				\Delta \chi_3&=& \chi_3\otimes \epsilon+\epsilon\otimes \chi_3+\lambda\chi_3\otimes \lambda\chi_4 \\[3mm]
				\Delta \chi_4&=&\chi_4 \otimes  \epsilon +K^2\otimes \chi_4
			\end{array}
		\end{equation}
		and hence $\fg$, as in (\ref{tg-space-sl}), is a quantum tangent space (compare with \cite[Section 14.2.4]{ks}, where the right coideal counterpart of this example is discussed). Thus, according to Theorem \ref{dual-thm},
		$(\cO_q(\rSL_2)\otimes \mathfrak{g}, \mathscr{L}^{\mathfrak{g}})$
		with $\mathscr{L}^{\mathfrak{g}}_{h\otimes X}(f)= hf_{1}X(f_{2})$ is a left covariant FOVC on $\cO_q(\rSL_2)$.
		
		One verifies that the quantum Lie bracket, given by the restricted adjoint action $[\cdot,\cdot]\colon\mathfrak{g}\otimes\mathfrak{g}\to U_q\mathfrak{sl}(2)$, $[\xi,\zeta]:=\xi_1\zeta S(\xi_2)$, closes in $\mathfrak{g}$. In fact,
		\begin{equation}\label{eq:qbracket}
			\begin{split}
				[\chi_1,\chi_1]
				&=q^2\lambda(\chi_1-\chi_4),\\
				[\chi_1,\chi_2]
				&=(\lambda-q^3)\chi_2,\\
				[\chi_1,\chi_3]
				&=q^{-1}\chi_3,\\
				[\chi_1,\chi_4]
				&=\lambda(\chi_1-\chi_4),    
			\end{split}\qquad
			\begin{split}
				[\chi_2,\chi_1]
				&=q\chi_2,\\
				[\chi_2,\chi_3]
				&=-[\chi_3,\chi_2]=q(\chi_1-\chi_4),\\
				[\chi_2,\chi_4]
				&=-[\chi_4,\chi_2]=-q^{-1}\chi_2,
			\end{split}\qquad
			\begin{split}
				[\chi_3,\chi_1]
				&=-q^{3}\chi_3,\\
				[\chi_3,\chi_4]
				&=-[\chi_4,\chi_3]=q\chi_3,
			\end{split}
		\end{equation}
		where we only displayed the non-trivial brackets. Using Remark \ref{rem:bracket} this implies that $\mathfrak{g}$ is a bicovariant quantum tangent space since $\mathcal{O}_q(\mathrm{SL}_2)$ and $U_q\mathfrak{sl}(2)$ are dually paired Hopf algebras according to \cite[Section 4.4.1]{ks}. Thus $(\cO_q(\rSL_2)\otimes \mathfrak{g}, \mathscr{L}^{\mathfrak{g}})$ is a bicovariant FOVC by Theorem \ref{thm:bicovVF}. This bicovariant FOVC is dual to the so-called $4D_+$ bicovariant FODC on $\mathcal{O}_q(\mathrm{SL}_2)$, see \cite{Woronowicz1989}.
	\end{example}
	We finish this section with a result that allows us to induce a (bicovariant) quantum tangent space from a Hopf algebra to any Hopf algebra quotient of it.
	\begin{proposition}\label{prop:quotientFOVC}
		Let $\pi\colon A\to H\cong A/I$ be a Hopf algebra surjection with corresponding Hopf ideal $I\subset A$. Given a quantum tangent space $\mathfrak{g}$ on $A$, we obtain a quantum tangent space
		\begin{equation}
			\mathfrak{g}_H:=\{\chi\in\mathfrak{g}~|~\chi(I)=0\}
		\end{equation}
		on $H$, where $\chi\in\mathfrak{g}_H$ is viewed as a functional on $H$ via $\underline{\chi}\colon H\to\Bbbk$, $\underline{\chi}(\pi(a)):=\chi(a)$ for all $a\in A$. If $\mathfrak{g}$ is a bicovariant quantum tangent space on $A$ then $\mathfrak{g}_H$ is a bicovariant quantum tangent space on $H$.
\end{proposition}
\begin{proof}
	Since $\mathfrak{g}$ is a quantum tangent space on $A$, there is a right $A$-action $\mathfrak{g}\otimes A\to A$, $\chi\otimes a\mapsto\chi\leftharpoonup a$, see \eqref{eq:leftharpoon}. By \eqref{qu-plane}, this action reads $(\chi\leftharpoonup a)(a')=\chi(aa')-\chi(a)\varepsilon(a')$ on elements $a,a'\in A$. Thus, for $\chi\in\mathfrak{g}_H$ it follows that $(\chi\leftharpoonup a)(a')=0$ if either $a\in I$ or $a'\in I$ (recalling that $I$ is a Hopf ideal). In particular, the action restricts to a map $\mathfrak{g}_H\otimes A\to\mathfrak{g}_H$ and this restriction descends to the quotient $\leftharpoonup_H\colon\mathfrak{g}_H\otimes H\to\mathfrak{g}_H$. From the properties of $\leftharpoonup$ it follows that $\leftharpoonup_H$ is a right $H$-action on $\mathfrak{g}_H$ and thus $H\otimes\mathfrak{g}_H$ becomes an $H$-bimodule via
	$$
	h\cdot(g\otimes\underline{\chi})\cdot k:=hgk_1\otimes(\underline{\chi}\leftharpoonup_Hk_2)
	$$
	for all $h,g,k\in H$ and $\chi\in\mathfrak{g}_H$. In complete analogy to Theorem \ref{dual-thm} one proves that $H\otimes\mathfrak{g}_H$, together with 
	$$
	\mathscr{L}^{\mathfrak{g}_H}\colon H\otimes\mathfrak{g}_H\to\mathrm{End}_\Bbbk(H),\qquad\mathscr{L}^{\mathfrak{g}_H}_{h\otimes\underline{\chi}}(g):=hg_1\underline{\chi}(g_2),
	$$
	where $h,g\in H$ and $\chi\in\mathfrak{g}_H$, forms a left covariant FOVC on $H$. From the same theorem it then follows that $\mathfrak{g}_H\cong{}^{\mathrm{co}H}(H\otimes\mathfrak{g}_H)$ is a quantum tangent space on $H$.
	
	Assume now that $\mathfrak{g}$ is a bicovariant quantum tangent space on $A$, i.e., the adjoint coaction takes the form $\mathrm{ad}_\mathfrak{g}\colon\mathfrak{g}\to\mathfrak{g}\otimes A$, $\mathrm{ad}_\mathfrak{g}(\chi)(a)=S(a_1)a_3\chi(a_2)$ for all $\chi\in\mathfrak{g}$ and $a\in A$. If $\chi\in\mathfrak{g}_H$, it follows that $\mathrm{ad}_\mathfrak{g}(\chi)(I)\subseteq I$ since $I$ is a Hopf ideal and $\chi(I)=0$. Thus, $(\mathrm{id}\otimes\pi)\circ\mathrm{ad}_\mathfrak{g}\colon\mathfrak{g}_H\to\mathfrak{g}_H\otimes H$ is well-defined and reads 
	$$
	(\mathrm{id}\otimes\pi)(\mathrm{ad}_\mathfrak{g}(\chi))(\pi(a))
	=(\mathrm{id}\otimes\pi)(\mathrm{ad}_\mathfrak{g}(\chi))(a)
	=\pi(S(a_1)a_3)\chi(a_2)
	=S(\pi(a)_1)\pi(a)_3\underline{\chi}(\pi(a)_2)
	$$ 
	on elements $\pi(a)\in H$, where $a\in A$ and $\chi\in\mathfrak{g}_H$. This shows that $(\mathrm{id}\otimes\pi)\circ\mathrm{ad}_\mathfrak{g}\colon\mathfrak{g}_H\to\mathfrak{g}_H\otimes H$ coincides with the adjoint $H$-coaction on $\mathfrak{g}_H$ and we conclude that the latter is a bicovariant quantum tangent space.
\end{proof}

\subsection{Vertical FOVC on comodule algebras}\label{fovc-ca-sec}

We extend the consideration of first order vector field calculi on $H$ of the form $\mathfrak{X}_H=H\otimes\mathfrak{g}$ with a quantum tangent space $\mathfrak{g}$ on $H$ to first order vector field calculi on right $H$-comodule algebras $A$ of the form $\mathfrak{X}=A\otimes\mathfrak{g}$, where $\mathfrak{g}$ is still a quantum tangent space on $H$. It turns out that $B:=A^{\mathrm{co}H}\subseteq A$ has to be a Hopf--Galois extension (see Section \ref{sec:durd}) in order for the axioms of a FOVC to hold.

Let $H$ be a Hopf algebra, $H^\circ$ its restricted
dual, as before, and let $A$ be a right $H$-comodule algebra with right $H$-coaction $\Delta_A\colon A\to A\otimes H$. For a finite dimensional vector space 
$\fg \subseteq (H^\circ)^+$, we define:
\begin{equation}
	\Chi:=A \otimes \fg, \qquad \mathscr{L}\colon\Chi \to  \mathrm{End}_\Bbbk(A),
	\quad \mathscr{L}_{a \otimes X}(a'):=aa'_0X(a'_1), \quad a,a' \in A, \, X \in \fg. 
\end{equation}
We show that the above comes very close to being a FOVC on $A$. In case of a Hopf--Galois extension all axioms of a FOVC are satisfied 
and if $\mathfrak{g}$ is bicovariant, then $(\mathfrak{X},\mathscr{L})$ is right $H$-covariant.

\begin{proposition}\label{prop:hopf}
	In the notation as above, let $\fg$ be a quantum tangent space. Then $\Chi$ admits the $A$-bimodule structure
	\begin{equation}\label{eq:bimod}
		a\cdot(a'\otimes\chi_i)\cdot a''
		:=aa'a''_0f_i^j(a''_1)\otimes\chi_j
	\end{equation}
	and $\mathscr{L}_{X}$ satisfies the Leibniz rule 
	\begin{equation}
		\mathscr{L}_{X}(aa')
		=\mathscr{L}_X(a)a'
		+\mathscr{L}_{X\cdot a}(a')
	\end{equation}
	for all $a,a',a''\in A$ and $X \in \Chi$.
	
	If, in addition, $B:=A^{\mathrm{co}H}\subseteq A$ is a Hopf--Galois extension, then $(\mathfrak{X},\mathscr{L})$ is a FOVC.
	
	If, yet in addition, $\mathfrak{g}$ is a bicovariant quantum tangent space, then $(\mathfrak{X},\mathscr{L})$ is a right $H$-covariant FOVC with respect to the right $H$-coaction
	\begin{equation}
		\Delta_\mathfrak{X}\colon\mathfrak{X}\to\mathfrak{X}\otimes H,\qquad\qquad a\otimes\chi_i\mapsto a_0\otimes\chi_j\otimes a_1N_i^j,
	\end{equation}
	where $N_i^j\in H$ are determined via \eqref{eq:gcoact} with $\Delta_\mathfrak{g}=\mathrm{ad}_\mathfrak{g}$.
\end{proposition}

\begin{proof}
	
	If $\Delta(\mathfrak{g})\subseteq \mathfrak{g}\otimes\varepsilon+H^\circ\otimes\mathfrak{g}$,
	by (\ref{qu-plane})
	this means that there are linear functionals $f_i^j\colon H\to\Bbbk$ with $f_i^j\in H^\circ$
	such that
	$\Delta(\chi_i)=\chi_i\otimes\varepsilon+f_i^j\otimes\chi_j$, $\Delta(f_i^j)=f_i^k\otimes f_k^j$ and $\varepsilon(f_i^j)=f_i^j(1)=\delta_i^j$. 
	Using these properties it is straightforward to check that \eqref{eq:bimod} defines a an $A$-bimodule structure on $\mathfrak{X}=A\otimes\mathfrak{g}$.
	Moreover, for all $a,a',a''\in A$ and $\chi_i\in\mathfrak{g}$ we have
	\begin{align*}
		\mathscr{L}_{a\otimes\chi_i}(a'a'')
		&=aa'_0a''_0\chi_i(a'_1a''_1)\\
		&=aa'_0a''_0\big(\chi_i(a'_1)\varepsilon(a''_1)+f_i^j(a'_1)\chi_j(a''_1)\big)\\
		&=aa'_0\chi_i(a'_1)a''+aa'_0f_i^j(a'_1)a''_0\chi_j(a''_1)\\
		&=\mathscr{L}_{a\otimes\chi_i}(a')a''
		+\mathscr{L}_{(a\otimes\chi_i)\cdot a'}(a'')
	\end{align*}
	and hence $\mathscr{L}_X$ satisfies the Leibniz rule for all $X \in \Chi$.
	
	Assume now that $B:=A^{\mathrm{co}H}\subseteq A$ is a Hopf--Galois extension, in addition.
	The only property left to check is the injectivity of $\mathscr{L}$.
	Let $\{\chi_i\}$ be a basis of $\fg$. We need to show that if $\mathscr{L}_{a^i \otimes \chi_i}(c)=0$
	for all $c \in A$, then $a^i\otimes \chi_i=0$ that is $a^i=0$ for all $i$.
	Let us define the following map:
	$$
	\phi:A \otimes A \lra A, \qquad \phi(e \otimes c):=\mathscr{L}_{(a^i \otimes \chi_i)\cdot e}(c)
	$$
	As one can readily check, this map descends to $A\otimes_B A$, $B:=A^{\mathrm{co}H}$. 
	In fact
	$$
	\phi(eb\otimes c)
	=\mathscr{L}_{(a^i \otimes \chi_i)\cdot eb}(c)
	=\mathscr{L}_{(a^i\otimes\chi_i)\cdot e}(bc)-\mathscr{L}_{(a^i\otimes\chi_i)\cdot e}(b)c
	=\phi(e\otimes bc)-0
	$$
	since $b_0\otimes b_1=b\otimes 1$ for $b \in B=A^{\mathrm{co}H}$.
	Let $\phi_B:A \otimes_B A \lra A$ be the map obtained by $\phi$, that is 
	$$
	\phi_B(e \otimes_B c):=\mathscr{L}_{(a^i \otimes \chi_i)\cdot e}(c).
	$$
	Since $A$ is a Hopf-Galois extension, we have the translation map
	$\tau:H \to A\otimes_B A$, $\tau(h)=h^{\langle1\rangle} \otimes_B h^{\langle2\rangle}$, see Section \ref{sec:durd}. Then
	$$
	\phi_B(\tau(h))=\mathscr{L}_{(a^i \otimes \chi_i)\cdot h^{\langle1\rangle}}(h^{\langle2\rangle})=
	\mathscr{L}_{a^i \otimes \chi_i} (h^{\langle1\rangle}h^{\langle2\rangle})
	-\mathscr{L}_{a^i\otimes\chi_i}(h^{\langle 1\rangle})h^{\langle 2\rangle}
	$$
	and since $h^{\langle1\rangle}h^{\langle2\rangle}=\varepsilon(h)1$ by (\ref{transl1}), 
	thus $\mathscr{L}_{a^i \otimes \chi_i}( h^{\langle1\rangle}h^{\langle2\rangle})=0$, we have
	$$
	-\phi_B(\tau(h))=a^i {h^{\langle1\rangle}}_0\chi_i( {h^{\langle1\rangle}}_1) h^{\langle2\rangle}=a^i {(h_2)^{\langle1\rangle}}\chi_i(S(h_1)) (h_2)^{\langle2\rangle}=
	a^i\varepsilon(h_2)\chi_i(S(h_1))=a^i\chi_i(S(h))
	$$
	where we have used first property (\ref{transl4}) and then property (\ref{transl1}) of the translation map. 
	By assumption $\mathscr{L}_{a^i\otimes\chi_i}(c)=0$ for all $c\in A$ and thus
	$$
	0=\mathscr{L}_{a^i\otimes\chi_i}(h^{\langle 1\rangle}h^{\langle 2\rangle})-\mathscr{L}_{a^i\otimes\chi_i}(h^{\langle 1\rangle})h^{\langle 2\rangle}
	=\mathscr{L}_{(a^i\otimes\chi_i)\cdot h^{\langle 1\rangle}}(h^{\langle 2\rangle})
	=\phi_B(\tau(h))=-a^i\chi_i(S(h))
	$$
	for all $h\in H$.
	Define, similarly to before, 
	$h:=S^{-1}(x^j)$, such that $\chi_i(x^j)=\delta_i^j$. Then we have
	$$
	0=\phi_B(\tau(h))=-a^i\chi_i(S(S^{-1}(x^j))=-a^j,
	$$
	thus obtaining that $(\mathfrak{X},\mathscr{L})$ is a FOVC.
	
	Under the additional assumption that $\mathfrak{g}$ is a bicovariant quantum tangent space one proves that $(\mathfrak{X},\mathscr{L})$ is right $H$-covariant in complete analogy to Theorem \ref{thm:bicovVF}, where one replaces the coproduct by the right $H$-coaction $\Delta_A\colon A\to A\otimes H$.
\end{proof}

\begin{definition}
	Let $B:=A^{\mathrm{co}H}\subseteq A$ be a Hopf--Galois extension and consider a bicovariant quantum tangent space $\mathfrak{g}\subseteq H^\circ$. We call elements of the right $H$-covariant FOVC $(A\otimes\mathfrak{g},\mathscr{L})$ of Proposition \ref{prop:hopf} \emph{vertical vector fields}.
\end{definition}

The term \emph{vertical} in the previous definition will be motivated a posteriori in the following sections. Before, we discuss a class of vertical vector fields for quantizations of algebraic groups and parabolic subgroups.

\begin{example}\label{ex-loc-p}
	Let $G$ be a complex semisimple algebraic group and $P$ be a parabolic subgroup with Hopf algebras $A=\cO_q(G)$ and $H=\cO_q(P)=\cO_q(G)/I_P$ quantizing the coordinate rings $\mathcal{O}(G)$ and $\mathcal{O}(P)$, where $I_P$ is a suitable Hopf algebra ideal (see \cite[Section 3]{afl} for more details). The algebra $A$ carries a natural right $H$-coaction given by $\Delta_A:=(\mathrm{id} \otimes \pi)\circ\Delta\colon A\to A\otimes H$, where $\Delta$ is the comultiplication of $A$ and $\pi\colon A\to H$ the quotient map. Assume that there is an Ore element $a\in A$ such that the Ore localization $A[a^{-1}]$ is a well-defined right $H$-comodule algebra (see \cite[Section I.1.7]{Kassel} and \cite[Lemma 4.12 (i.)]{aflw} for more details). We further assume that $B:=A[a^{-1}]^{\mathrm{co}H}\subseteq A$ is a Hopf--Galois extension. Then, given a bicovariant quantum tangent space $\mathfrak{g}$ on $A$ with induced bicovariant quantum tangent space $\mathfrak{p}:=\mathfrak{g}_H$ on $H$ (see Proposition \ref{prop:quotientFOVC}), we obtain a right $H$-covariant FOVC $(A \otimes \mathfrak{p}, \mathscr{L})$ on $A$ by Proposition~\ref{prop:hopf}.  
	
	To give an explicit instance of such a situation, consider $A=\cO_q(\rSL_2)$, as in Example \ref{ex:SL2:FOVC}. The ideal $(\gamma)\subset A$ generated by $\gamma$ is a Hopf ideal and thus the quotient map 
	$$
	\pi\colon A\to H:=A/(\gamma),\qquad\begin{pmatrix}
		\alpha & \beta\\
		\gamma & \delta
	\end{pmatrix}\mapsto\begin{pmatrix}
		t & p\\
		0 & t^{-1}
	\end{pmatrix}
	$$ 
	is a Hopf algebra surjection. Moreover, $\alpha\in A$ is an Ore element such that the Ore localization $A[\alpha^{-1}]$ becomes a right $H$-comodule algebra with respect to the right $H$-coaction
	\begin{align*}
		\begin{pmatrix}
			\alpha & \beta\\
			\gamma & \delta
		\end{pmatrix}&\mapsto\begin{pmatrix}
			\alpha & \beta\\
			\gamma & \delta
		\end{pmatrix}\otimes\begin{pmatrix}
			t & p\\
			0 & t^{-1}
		\end{pmatrix}\\
		\alpha^{-1}&\mapsto\alpha^{-1}\otimes t^{-1}
	\end{align*}
	with coinvariant subalgebra $B:=A[\alpha^{-1}]^{\mathrm{co}H}=\mathbb{C}[\gamma\alpha^{-1}]$ generated by the element $\gamma\alpha^{-1}$. It was shown in \cite{afl} that $B\subseteq A$ is Hopf--Galois, even a trivial extension with respect to the cleaving map $j\colon H\to A[\alpha^{-1}]$ determined on generators by
	$$
	j\colon\begin{pmatrix}
		t & p\\
		0 & t^{-1}
	\end{pmatrix}\mapsto\begin{pmatrix}
		\alpha & \gamma\\
		0 & \alpha^{-1}
	\end{pmatrix}.
	$$
	Considering the bicovariant quantum tangent space $\mathfrak{g}$ from Example \ref{ex-pair1} we obtain
	\begin{equation}\label{tg-space-p}
		\mathfrak{p}:=\mathfrak{g}_H=\mathrm{span}\{\underline{\chi_1},\underline{\chi_3},\underline{\chi_4}\},
	\end{equation}
	where $\chi_i$ were defined in (\ref{chis}) and $\underline{\chi_i}$ denotes the induced element on the quotient, as in Proposition \ref{prop:quotientFOVC}. Note that  $\chi_2(I_P)\neq 0$, hence it is not well defined on the quotient. From the previous argument and Proposition \ref{prop:quotientFOVC} we infer that $(\mathcal{O}_q(\mathrm{SL}_2)[\alpha^{-1}]\otimes\mathfrak{p},\mathscr{L})$ is a right $H$-covariant FOVC on $\mathcal{O}_q(\mathrm{SL}_2)[\alpha^{-1}]$.
	
	This example can be generalized to $\mathcal{O}_q(\mathrm{SL}_n)$, utilizing the observations in \cite[Section 2]{afl}.
\end{example}

\subsection{Duality of vertical FOVC and FODC}\label{sec:duality}

In this section we relate our construction of vertical FOVCi to the vertical FODCi of \DJ ur\dj evi\'c \cite{DurII}, also motivating the term \textit{vertical}.

Fix a Hopf algebra $H$ and a bicovariant FODC $(\Gamma,\mathrm{d}_H)$ on $H$. By the Fundamental Theorem of Hopf modules, we have the following isomorphism
\begin{equation}\label{FundThm}
	\Gamma \cong H \otimes \Lambda
\end{equation}
of bicovariant $H$-bimodules, where $\Lambda:=\{\omega\in\Gamma~|~\lambda_\Gamma(\omega)=1\otimes\omega\}$ is the vector subspace of coinvariant $1$-forms. Explicitly, the left $H$-action and left $H$-coaction are induced by the multiplication and comultiplication of $H$, while the right $H$-action and right $H$-coaction read
$$
(h\otimes\vartheta)\cdot h'=hh'_1\otimes S(h'_2)\vartheta h'_3,\qquad\qquad
\Delta_{H\otimes\Lambda}(h\otimes\vartheta)
=h_1\otimes g^i_2\mathrm{d}(k^i_2)\otimes h_2S(g^i_1k^i_1)g^i_3k^i_3
$$
for all $h,h'\in H$ and $\vartheta=g^i\mathrm{d}k^i\in\Lambda$. Transporting the differential $\mathrm{d}_H$ via the above isomorphism, we obtain
$$
\mathrm{d}'_H\colon H\to H\otimes\Lambda,\qquad\qquad
\mathrm{d}'_H(h)=h_1\otimes S(h_2)\mathrm{d}(h_3)
$$
which structures $H\otimes\Lambda$ as a bicovariant FODC isomorphic to $(\Gamma,\mathrm{d}_H)$. This construction admits a generalization to Hopf--Galois extensions, see \cite[Lemma 3.1]{DurII} or \cite[Proposition 3.17]{AntEmaTho}.
\begin{lemma}
	Let $A$ be a right $H$-comodule algebra such that $B:=A^{\mathrm{co}H}\subseteq A$ is a Hopf--Galois extension and let $\Lambda$ be as in \eqref{FundThm}. Then $\Gamma_\mathrm{ver}:=A\otimes\Lambda$ is a right $H$-covariant FODC on $A$ with respect to the obvious left $A$-action, the right $A$-action and right $H$-coaction
	\begin{equation}\label{verRight}
		(a\otimes\vartheta)\cdot a'
		:=aa'_0\otimes S(a'_1)\vartheta a'_2,\qquad\qquad
		\Delta_\mathrm{ver}(a\otimes\vartheta)
		:=a_0\otimes g^i_2\mathrm{d}(k^i_2)\otimes a_1S(g^i_1k^i_1)g^i_3k^i_3
	\end{equation}
	for all $a,a'\in A$ and $\vartheta=g^i\mathrm{d}(k^i)\in\Lambda$. The differential is given by
	\begin{equation}
		\mathrm{d}(a):=a_0\otimes S(a_1)\mathrm{d}(a_2)
	\end{equation}
	for all $a\in A$. We will refer to $\Gamma_{\textrm{ver}}$ as \DJ ur\dj evi\'c vertical forms.
\end{lemma}
Note that on a basis $\{\omega_i\}$ of $\Lambda$ the right $A$-action \eqref{verRight} reads
$$
(a \otimes \omega^i)\cdot a':=aa_0' f_j^i(a_1')\otimes \omega^j
$$
where $a,a'\in A$. The right $H$-covariant FODC $(\mathrm{ver},\mathrm{d})$ is dual to the FOVC that we have constructed in Proposition \ref{prop:hopf}.
In fact, under the assumption that $\Lambda$ is finite-dimensional, we define
$$
\fg:=\Lambda^*
$$
and let $\chi_i$ be the dual basis to $\omega^i$, that is $\langle \chi_j, \omega^i\rangle=\de_{ij}$. This extends to an $A$-bilinear dual pairing $\langle\cdot,\cdot\rangle\colon (A\otimes\mathfrak{g})\otimes(A\otimes\Lambda)\to A$
\begin{equation}
	\langle a\otimes\chi_i,a'\otimes\omega^j\rangle
	:=aa'\delta_i^j.
\end{equation}
From Proposition \ref{prop:duality'} we induce a Lie derivative 
\begin{align*}
	\mathscr{L}_{a\otimes\chi_i}(a')
	&=\langle a\otimes\chi_i, da'\rangle
	=\langle a\otimes\chi_i, a'_0\otimes S(a'_1)\mathrm{d}_H(a'_2)\rangle
	=\langle a\otimes\chi_i, a'_0\otimes S(a'_1)a'_2\omega^j\chi_j(a'_3)\rangle\\
	&=\langle a\otimes\chi_i, a'_0\otimes\omega^j\rangle\chi_j(a'_1)
	=aa'_0\chi_i(a'_1)
\end{align*}
on $A\otimes\mathfrak{g}$. This coincides with the Lie derivative considered for vertical vector fields in Proposition \ref{prop:hopf}. Thus we obtain the following result.
\begin{proposition}
	Let $B:=A^{\mathrm{co}H}\subseteq A$ be a Hopf--Galois extension. Then, given a bicovariant quantum tangent space $\mathfrak{g}$ on $H$, the vertical vector fields $\mathfrak{X}_\mathrm{ver}:=A\otimes\mathfrak{g}$ are dual (in the sense of Corollary \ref{cor:covFGPequ}) to the vertical forms $\Gamma_\mathrm{ver}:=A\otimes\Lambda$,  where $\Lambda\leftrightarrow\mathfrak{g}$ are dual vector spaces.
\end{proposition}

\subsection{Base vector fields}\label{base-vf-sec}

Let $A$ be a right $H$-comodule algebra with subalgebra $B=A^{\mathrm{co}H}\subseteq A$ of coinvariant elements.
\begin{definition}
	Given any FOVC $(\mathfrak{X}_A,\mathscr{L}^A)$ on $A$ we call the quotient
	\begin{equation}\label{X_B}
		\mathfrak{X}_B:=\{X\in\mathfrak{X}_A~|~\mathscr{L}^A_X(B)\subseteq B\}\bigg/\{X\in\mathfrak{X}_A~|~\mathscr{L}^A_X(B)=0\}
	\end{equation}
	the vector space of \textit{base vector fields}.
\end{definition}
The above base vector fields can be understood as vector fields on $A$ that restrict to the base algebra $B$ and where we identify such vector fields if their difference constantly vanishes on $B$.
\begin{proposition}\label{prop:base}
	$\mathfrak{X}_B$ is a FOVC on $B$ with corresponding Lie derivative $\mathscr{L}^B$ induced from $\mathscr{L}^A$. Explicitly, the latter reads
	\begin{equation}\label{BLieDer}
		\mathscr{L}^B_{[X]}(b):=\mathscr{L}^A_X(b)
	\end{equation}
	for all $[X]\in\mathfrak{X}_B$ and $b\in B$, where $X\in\mathfrak{X}_A$ is an arbitrary representative of $[X]$.
\end{proposition}
\begin{proof}
	The map \eqref{BLieDer} is well-defined, i.e., independent from the choice of representative, since the kernel of \eqref{X_B} vanishes on $B$. We show that $(\mathfrak{X}_B,\mathscr{L}^B)$ is a FOVC on $B$. First of all, $\mathfrak{X}_B$ is a $B$-bimodule: $b\cdot [X]\cdot b':=[b\cdot X\cdot b']\in\mathfrak{X}_B$ if $X\in\mathfrak{X}_B$ and $b,b'\in B$, since 
	$$
	\mathscr{L}^A_{b\cdot X\cdot b'}(b'')
	=b\underbrace{\mathscr{L}^A_X(b'b'')}_{\in B}-b\underbrace{\mathscr{L}^A_X(b')}_{\in B}b''\in B
	$$
	and the above clearly vanishes if $\mathscr{L}^A_X(B)=0$. In particular, the $B$-bimodule action is well-defined. The Leibniz rule of $\mathscr{L}^B$ is then induced from the one of $\mathscr{L}^A$. It remains to prove injectivity of $\mathscr{L}^B\colon\mathfrak{X}_B\to\mathrm{End}_\Bbbk(B)$. Assume that $\mathscr{L}^B_{[X]}(b)=0$ for an $[X]\in\mathfrak{X}_B$ and all $b\in B$. Then $\mathscr{L}^A_X(b)=0$ for a representative $X\in\mathfrak{X}_A$ and all $b\in B$, which implies that $[X]=0$.
\end{proof}

The above definition and proposition are valid for every subalgebra $B\subseteq A$, so we do not necessarily have to assume $B=A^{\mathrm{co}H}$. However, also in the following, we restrict our attention to the case $B=A^{\mathrm{co}H}$.

Recall that, given a right $H$-covariant FOVC $(\mathfrak{X},\mathscr{L})$ on a right $H$-comodule algebra $A$ we denote the vector space of right $H$-coinvariant vector fields by
\begin{equation}\label{eq:coinvX}
	\mathfrak{X}^{\mathrm{co}H}=\{X\in\mathfrak{X}~|~\Delta_{\mathfrak{X}}(X)=X\otimes 1\}\subseteq\mathfrak{X}.
\end{equation}
\begin{lemma}\label{lem:XcoH}
	Let $A$ be a right $H$-comodule algebra and denote $B:=A^{\mathrm{co}H}\subseteq A$.
	For any right $H$-covariant FOVC $(\mathfrak{X},\mathscr{L})$ we have
	\begin{equation}\label{eq:Xco2}
		\mathfrak{X}^{\mathrm{co}H}\subseteq\{X\in\mathfrak{X}~|~\mathscr{L}_X(B)\subseteq B\},
	\end{equation}
	i.e., coinvariant vector fields respect the subalgebra $B$ of coinvariants.
\end{lemma}
\begin{proof}
	Let $X\in\mathfrak{X}^{\mathrm{co}H}$ and $b\in B$. Then $\mathscr{L}_Xb\in B$, since
	\begin{align*}
		\Delta_A(\mathscr{L}_Xb)
		=\mathscr{L}_{X_0}(b_0)\otimes X_1b_1
		=\mathscr{L}_Xb\otimes 1
	\end{align*}
	by the commutativity of \eqref{diagLcov}.
\end{proof}

Note that \eqref{eq:Xco2} is not an equality in general. To give an explicit counterexample we discuss the vector field calculus on the Podle\'s sphere $B$ induced from the $3$-dimensional vector field calculus on $\mathcal{O}_q(\mathrm{SL}_2(\mathbb{C}))$, as introduced in Example \ref{ex:SL2:FOVC}.

\begin{example}\label{ex:SLq2'}
	Let $A=\mathcal{O}_q(\mathrm{SL}_2(\mathbb{C}))$ and consider
	the circle Hopf algebra $H:=\mathcal{O}(U(1))=\mathbb{C}[t,t^{-1}]$, which is defined with the usual group-like comultiplication, counit and antipode, namely $\Delta(t^\pm)=t^\pm\otimes t^\pm$, $\varepsilon(t^\pm)=1$ and $S(t^\pm)=t^\mp$. Then $A$ becomes a right $H$-comodule algebra with respect to the coaction
	$$
	\Delta_A\begin{pmatrix}
		\alpha & \beta\\
		\gamma & \delta
	\end{pmatrix}=\begin{pmatrix}
		\alpha & \beta\\
		\gamma & \delta
	\end{pmatrix}\otimes\begin{pmatrix}
		t & 0\\
		0 & t^{-1}
	\end{pmatrix}.
	$$
	The subalgebra of coinvariants $B=A^{\mathrm{co}H}=\mathbb{C}_q[\mathbb{S}^2]=\mathrm{span}_\mathbb{C}\{B_\pm,B_0\}$ is the \textit{Podle\'s sphere} with generators $x:=-q^{-1}\beta\gamma$, $z:=\gamma\delta$ and $z^*:=-q\alpha\beta$ satisfying the relations
	$$
	zx=q^2xz,\quad
	z^*x=q^{-2}xz^*,\quad
	zz^*=q^4z^*z+q^2(1-q^2)x,\quad
	z^*z=x(1-x),
	$$
	see for example \cite[Lemma 2.3.4]{bm}.
	One can show that $B\subseteq A$ is a faithfully flat Hopf--Galois extension, see e.g. \cite[Example 6.26]{BrJaMa}.
	Now consider the $3$-dimensional FOVC $\mathfrak{X}_A=\mathrm{span}_A\{\chi^0,\chi^\pm\}$ of Example \ref{ex:SL2:FOVC}. There, we remarked that $\mathfrak{X}_A$ is left covariant. One verifies that $\mathfrak{X}_A$ is further right $H$-covariant if endowed with the right $H$-coaction
	\begin{equation}\label{eq:HcoactSL2}
		\begin{split}
			\Delta_{\mathfrak{X}(A)}(a\otimes\chi_0)&=a_0\otimes\chi_0\otimes a_1,\\
			\Delta_{\mathfrak{X}(A)}(a\otimes\chi_\pm)&=a_0\otimes\chi_\pm\otimes a_1t^{\mp 2}
		\end{split}
	\end{equation}
	where $a\in A$. In fact, \eqref{eq:HcoactSL2} constitutes a right $H$-coaction and is compatible with the $A$-module actions of $\mathfrak{X}_A$ and the Lie derivative.
	One verifies that $\mathfrak{X}_B$ is left $B$-generated by $[\alpha^2\chi_+]$ and $[\beta^2\chi_-]$. 
	In fact, since $\mathfrak{X}_A$ is a free $A$-module this can be checked on $\chi_+$ and $\chi_-$, which only respect $B$ if multiplied by an element of degree $-2$ and $2$, respectively.
	Now note that the vector field $\alpha\otimes\chi_0$ preserves $B$, in fact $\alpha\otimes\chi_0$ vanishes on $B$. However, $\Delta_{\mathfrak{X}(A)}(\alpha\otimes\chi_0)=\alpha\otimes\chi_0\otimes t$ and thus $\alpha\otimes\chi_0$ is not an element of $\mathfrak{X}^{\mathrm{co}H}$. This shows that \eqref{eq:Xco2} is not an equality for the vector calculus on the Podle\'s sphere.
\end{example}

\begin{corollary}
	If $(\mathfrak{X}_A,\mathscr{L}^A)$ is a right $H$-covariant FOVC on $A$ then
	$$
	[\mathfrak{X}_A^{\mathrm{co}H}]\subseteq\mathfrak{X}_B,
	$$
	where $[\mathfrak{X}_A^{\mathrm{co}H}]$ denotes the image of $\pi_{\mathfrak{X}_B}\colon\{X\in\mathfrak{X}_A~|~\mathscr{L}^A_X(B)\subseteq B\}\to\mathfrak{X}_B$.
\end{corollary}
\begin{proof}
	This follows immediately from Lemma \ref{lem:XcoH}.
\end{proof}

\subsection{The Atiyah sequence for vector field calculi}\label{sec:Atiyah}

In this section we fix a (faithfully flat) Hopf--Galois extension $B=A^{\mathrm{co}H}\subseteq A$ and denote the translation map by $\tau\colon H\to A\otimes_BA$, $\tau(h)=h^{\langle 1\rangle}\otimes_Bh^{\langle 2\rangle}$. Quantum principal vector field calculi on $B\subseteq A$ will be defined by the exactness of a noncommutative version of the Atiyah sequence. Vertical vector fields, as introduced in Section~\ref{fovc-ca-sec}, inject into vector fields on the total space algebra, which themselves project onto horizontal vector fields. In preparation of the definition we discuss the following result.
\begin{lemma}\label{lem:sequence}
	Let $(\mathfrak{X}_A,\mathscr{L}^A)$ be a FOVC on $A$.
	\begin{enumerate}
		\item[i.)] If $\mathfrak{g}$ is a quantum tangent space on $H$ such that $\mathscr{L}_{A\otimes\mathfrak{g}}\subseteq\mathscr{L}^A_{\mathfrak{X}_A}\subseteq\mathrm{End}_\Bbbk(A)$, then there exists a (unique) morphism $\phi\colon A\otimes\mathfrak{g}\to\mathfrak{X}_A$ of FOVCi, i.e. an injective $A$-bimodule morphism such that the diagram
		\begin{equation}\label{phi-fovc}
			\begin{tikzcd}
				A\otimes\mathfrak{g}\arrow{dr}[swap]{\mathscr{L}} \arrow{r}{\phi}
				& \mathfrak{X}_A \arrow{d}{\mathscr{L}^A}\\
				& \mathrm{End}_\Bbbk(A)
			\end{tikzcd}
		\end{equation}
		commutes.
		\item[ii.)] If $(\mathfrak{X}_A,\mathscr{L}^A)$ is right $H$-covariant and $\mathfrak{X}_B$ denotes the corresponding base vector fields, then $\psi\colon\mathfrak{X}_A\to\mathfrak{X}_BA$, given by
		\begin{equation}\label{chidef}
			\psi(X):=\underbrace{[X_0\cdot(X_1)^{\langle 1\rangle}]}_{\in\mathfrak{X}_B}(X_1)^{\langle 2\rangle}
		\end{equation}
		is well-defined, where $[\cdot]\colon\{X\in\mathfrak{X}_A~|~\mathscr{L}_X(B)\subseteq B\}\to\mathfrak{X}_B$ denotes the quotient map and $\mathfrak{X}_BA$ is a short notation for $\mathfrak{X}_B \otimes_B A$.
	\end{enumerate}
\end{lemma}
\begin{proof}
	\begin{enumerate}
		\item[i.)] From Proposition \ref{prop:hopf} it follows that $A\otimes\mathfrak{g}$ is a FOVC on $A$. From the injectivity of the Lie derivatives we infer that from the assumption $\mathscr{L}_{A\otimes\mathfrak{g}}\subseteq\mathscr{L}^A_{\mathfrak{X}_A}$ it follows that for every $a\otimes x\in A\otimes\mathfrak{g}$ there is a unique $X\in\mathfrak{X}_A$, such that $\mathscr{L}^A_X=\mathscr{L}_{a\otimes x}$. Thus $\phi(a\otimes x):=X$ gives the (necessarily unique) morphism $\phi\colon A\otimes\mathfrak{g}\to\mathfrak{X}_A$ of FOVCi.
		
		\item[ii.)] For well-definedness we have to argue that $[X_0\cdot(X_1)^{\langle 1\rangle}]\otimes_B(X_1)^{\langle 2\rangle}\in\mathfrak{X}_B\otimes_BA$ for all $X\in\mathfrak{X}_A$. Note that the quotient map $\pi_{\mathfrak{X}_B}\colon\{Y\in\mathfrak{X}_A~|~\mathscr{L}^A_Y(B)\subseteq B\}\to\mathfrak{X}_B$ is right $B$-linear (even $B$-bilinear) and thus for $X\in\{Y\in\mathfrak{X}_A~|~\mathscr{L}^A_Y(B)\subseteq B\}$ we have that
		$$
		X\otimes h\mapsto X\otimes h^{\langle 1\rangle}\otimes_Bh^{\langle 2\rangle}
		\mapsto[X\cdot h^{\langle 1\rangle}]\otimes_Bh^{\langle 2\rangle}
		$$
		is well-defined over the balanced tensor product $h^{\langle 1\rangle}\otimes_Bh^{\langle 2\rangle}$ for all $h\in H$ such that $X\cdot h^{\langle 1\rangle}\otimes_Bh^{\langle 2\rangle}\in\{Y\in\mathfrak{X}_A~|~\mathscr{L}^A_Y(B)\subseteq B\}\otimes_BA$.
		We show that $X_0\cdot(X_1)^{\langle 1\rangle}\otimes_B(X_1)^{\langle 2\rangle}\in\mathfrak{X}_A^{\mathrm{co}H}\otimes_BA$ for any $X\in\mathfrak{X}_A$. In fact, we have
		\begin{align*}
			\Delta_{\mathfrak{X}}(X_0\cdot(X_1)^{\langle 1\rangle})\otimes_B(X_1)^{\langle 2\rangle}
			&=X_0\cdot((X_2)^{\langle 1\rangle})_0\otimes X_1((X_2)^{\langle 1\rangle})_1\otimes_B(X_2)^{\langle 2\rangle}\\
			&\overset{\ref{transl4}}{=}X_0\cdot(X_3)^{\langle 1\rangle}\otimes X_1S(X_2)\otimes_B(X_3)^{\langle 2\rangle}\\
			&=X_0\cdot(X_1)^{\langle 1\rangle}\otimes 1\otimes_B(X_1)^{\langle 2\rangle}.
		\end{align*}
		Then, using Lemma \ref{lem:XcoH} the claim follows, which concludes the proof.
	\end{enumerate}
\end{proof}

Combining the maps $\phi$ and $\psi$ of the previous lemma into an exact sequence leads to the notion of quantum principal vector field calculus.
\begin{definition}\label{qpvc-def}
	A \textit{quantum principal vector field calculus} (QPVC)
	for $B=A^{\mathrm{co}H}\subseteq A$ is the following data:
	\begin{enumerate}
		\item[i.)] a 
		bicovariant quantum tangent space $\mathfrak{g}$ on $H$ and
		
		\item[ii.)] a right $H$-covariant FOVC $(\mathfrak{X}_A,\mathscr{L}^A)$ on $A$
	\end{enumerate}
	such that the sequence
	\begin{equation}\label{at-seq}
		0\to A\otimes\mathfrak{g}\xrightarrow{\phi}\mathfrak{X}_A\xrightarrow{\psi} \mathfrak{X}_BA\to 0
	\end{equation}
	is well-defined and exact.
	In this case we call $A\otimes\mathfrak{g}$ the \textit{vertical vector fields}, $\mathfrak{X}_BA$ the \textit{horizontal vector fields} and we colloquially refer to \eqref{at-seq} as the \textit{Atiyah sequence}.
\end{definition}

In the context of classical differential geometry, when considering a Lie group $G$ and a principal $G$-bundle $P\to M$, the classical Atiyah sequence
\begin{equation}
	0\to\Gamma\mathrm{ad}P\to\mathfrak{X}_G(P)\to\mathfrak{X}(M)\to 0,
\end{equation}
where $\Gamma\mathrm{ad}P$ denotes the adjoint bundle, $\mathfrak{X}_G(P)$ are the $G$-invariant vector fields on $P$ and $\mathfrak{X}(M)$ the base vector fields,
is recovered from \eqref{at-seq} by applying the coinvariant functor $M\mapsto M^{\mathrm{co}H}$. A generalization of this to module algebras of triangular Hopf algebras was given in \cite[Equation (3.1)]{AschieriLandiPagani}. The latter also appears as the invariant part of a special case of \eqref{at-seq}.

Let us discuss other examples of QPVCi.

\begin{example}
	We consider the 3D right $H=\mathcal{O}(U(1))$-covariant FOVC $(\mathfrak{X}_A,\mathscr{L}^A)$ on $A=\mathcal{O}_q(\mathrm{SL}_2(\mathbb{C}))$, as discussed in Examples \ref{ex:SL2:FOVC} and \ref{ex:SLq2'}. Recall that $(\mathfrak{X}_A,\mathscr{L}^A)$ is also a left $A$-covariant FOVC with quantum tangent space $\fg_A:=\mathrm{span}\{\chi_0,\chi_{\pm}\}$. The induced base vector field calculus on the Podle\'s sphere $B=A^{\mathrm{co}H}=\mathbb{C}_q[\mathbb{S}^2]$ was discussed in Example \ref{ex:SLq2'}. On $H$ we choose the bicovariant quantum tangent space $\mathfrak{g}_H:=\mathrm{span}_\mathbb{C}\{\theta\}$ corresponding to the bicovariant FOVC $(\mathfrak{X}_H,\mathscr{L}^H)$ with $\mathfrak{X}_H=H\otimes\mathfrak{g}_H$ determined by
	$$
	(1\otimes\theta)\cdot t=q^{-2}t\otimes\theta
	$$ 
	as detailed in Section \ref{sec:bicov}.
	Moreover, 
	$$
	\mathscr{L}^H_\theta(t)
	=\langle\theta,\mathrm{d}t\rangle
	=\langle\theta,tt^{-1}\mathrm{d}t\rangle
	=\langle\theta,t^{-1}\mathrm{d}t\rangle q^{-2}t
	=q^{-2}t
	$$
	and 
	$$
	\mathscr{L}^H_\theta(t^{-1})
	=\langle\theta,\mathrm{d}(t^{-1})\rangle
	=-\langle\theta,t^{-1}\mathrm{d}(t)t^{-1}\rangle
	=-t^{-1},
	$$
	where, to ease the notation, we write $\mathscr{L}^H_{\theta}$ in place of
	$\mathscr{L}^H_{1\otimes \theta}$.
	
	In order to check that the above form a QPVC we have to study the maps $\phi$ and $\psi$ introduced in Definition~\ref{qpvc-def}. We claim that 
	$$
	\phi:A\otimes\mathfrak{g}_H\lra\mathfrak{X}_A, \qquad \phi(a \otimes \theta)=a \otimes \chi_0,
	$$
	which is then clearly injective. In fact,
	$
	\mathscr{L}^A\circ\phi=\mathscr{L}^\mathrm{ver},
	$
	where $\mathscr{L}^\mathrm{ver}$ is pertaining to the vertical calculus on $A$ as
	defined in Section \ref{fovc-ca-sec} since, comparing with \eqref{eq:chi0},
	\begin{align*}
		\mathscr{L}^\mathrm{ver}_{1\otimes\theta}\begin{pmatrix}
			\alpha & \beta\\
			\gamma & \delta
		\end{pmatrix}
		=\begin{pmatrix}
			\alpha\theta(t) & \beta\theta(t^{-1})\\
			\gamma\theta(t) & \delta\theta(t^{-1})
		\end{pmatrix}
		=\begin{pmatrix}
			q^{-2}\alpha & -\beta\\
			q^{-2}\gamma & -\delta
		\end{pmatrix}
		=\mathscr{L}^A_{\chi_0}\begin{pmatrix}
			\alpha & \beta\\
			\gamma & \delta
		\end{pmatrix}
	\end{align*}
	holds. By the properties of the Lie derivative this is sufficient to show commutativity of (\ref{phi-fovc}). Thus, the map $\psi\colon\mathfrak{X}_A\to\mathfrak{X}_BA$ as in Lemma \ref{lem:sequence} $ii.)$ is well-defined.
	We continue to show the exactness of the sequence (\ref{at-seq}).
	For this, recall from Example \ref{ex:SLq2'} that $\mathfrak{X}_B$ is the left $B$-module generated by $[\alpha^2\chi_+]$ and $[\beta^2\chi_-]$. Then, using \eqref{transl5} and that $\alpha^2\chi_+$ and $\beta^2\chi_-$ are right $H$-coinvariant vector fields, it follows that $\psi$ is surjective, since
	$$
	\psi((\alpha^2\chi_+)\cdot a+(\beta^2\chi_-)\cdot a')
	=[(\alpha^2\chi_+)\cdot a_0(a_1)^{\langle 1\rangle}](a_1)^{\langle 2\rangle}
	+[(\beta^2\chi_-)\cdot a'_0(a'_1)^{\langle 1\rangle}](a'_1)^{\langle 2\rangle}
	=[(\alpha^2\chi_+)]a+[(\beta^2\chi_-)]a'
	$$
	for arbitrary $a,a'\in A$. Further recall that $\chi_0$ is coinvariant (see equation \eqref{eq:HcoactSL2}) and $\mathscr{L}^A_{\chi_0}(B)=0$, which implies $\chi_0\in\ker\psi$. Since $\mathfrak{X}_A$ is a free $A$-module it follows that the kernel of $\psi$ is precisely $A\cdot\chi_0$. Thus, the Atiyah sequence
	$$
	0\to A\otimes\mathfrak{g}_H\xrightarrow{\phi}\mathfrak{X}_A\xrightarrow{\psi}\mathfrak{X}_BA\to 0
	$$
	is exact and we obtain a QPVC as in Definition \ref{qpvc-def}.
\end{example}

\subsection{Crossed product FOVC}\label{sec:crossed}

As a class of examples of QPVCi in accordance with Definition \ref{qpvc-def} we construct FOVCi on crossed product algebras or, equivalently, on cleft extensions. This approach is dual to the one in \cite{SciWeb}.

Let $H$ be a Hopf algebra (with invertible antipode) and $\rhd\colon H\otimes B\to B$ a \textit{measure}, i.e., a linear map such that
$$
h\rhd(bb')=(h_1\rhd b)(h_2\rhd b'),\qquad h\rhd 1_B=\varepsilon(h)1_B
$$
for all $h\in H$ and $b,b'\in B$. We further assume that $\rhd$ is a weak action, i.e., that there exists a convolution invertible map $\sigma\colon H\otimes H\to B$ such that
$$
\sigma(1\otimes h)=\varepsilon(h)1=\sigma(h\otimes 1),\qquad
(h_1\rhd\sigma(h'_1\otimes h''_1))\sigma(h_2\otimes h'_2h''_2)
=\sigma(h_1\otimes h'_2)\sigma(h_2h'_2\otimes h'')
$$
and
$$
h\rhd(h'\rhd b)=\sigma(h_1\otimes h'_1)(h_2h'_2\rhd b)\sigma^{-1}(h_3\otimes h'_3)
$$
for all $h,h',h''\in H$ and $b\in B$. Then, we define an associative unital algebra $B\#_\sigma H$, the \textit{crossed product algebra}, to equal $B\otimes H$ as a vector space and being endowed with the multiplication
$$
(b\#_\sigma h)(b'\#_\sigma h')
=b(h_1\rhd b')\sigma(h_2\otimes h'_1)\#_\sigma h_3h'_2
$$
and unit $1\otimes 1$ (see e.g. \cite[Lemma 7.1.2]{mont}). Moreover, $B\#_\sigma H$ is a right $H$-comodule algebra with coaction $\Delta_{\#_\sigma}:=\mathrm{id}\otimes\Delta\colon B\#_\sigma H\to (B\#_\sigma H)\otimes H$. Then $B\cong B\otimes 1\subseteq B\#_\sigma H$ is the subalgebra of coinvariants under this coaction and $B\subseteq B\#_\sigma H$ is a cleft extension with cleaving map $j\colon H\to B\#_\sigma H$, $h\mapsto 1\#_\sigma h$. Conversely, for every cleft extension $B\subseteq A$ there exists a measure and a $2$-cocylce $\sigma$ such that $B\#_\sigma H\cong A$ as right $H$-comodule algebras (see e.g. \cite[Theorem 7.2.2]{mont}).

Fix a crossed product algebra $B\#_\sigma H$ and assume that there is a bicovariant FOVC $\mathfrak{X}_H$ on $H$ and a FOVC $\mathfrak{X}_B$ on $B$. We further assume that there is a linear map $\rhd\colon H\otimes\mathfrak{X}_B\to\mathfrak{X}_B$ such that
\begin{equation}\label{eq:crossed1}
	h\rhd(b\cdot X\cdot b')=(h_1\rhd b)\cdot(h_2\rhd X)\cdot(h_3\rhd b'),\qquad
	1\rhd X=X,    
\end{equation}
\begin{equation}\label{eq:crossed2}
	h\rhd(h'\rhd X)=\sigma(h_1\otimes h'_1)\cdot(h_2h'_2\rhd X)\cdot\sigma(h_3\otimes h'_3),    
\end{equation}
as well as
\begin{equation}\label{eq:crossed3}
	h\rhd(\mathscr{L}_X(b))=\mathscr{L}_{h_1\rhd X}(h_2\rhd b),\qquad
	\mathscr{L}_X\circ\sigma=0
\end{equation}
for all $h,h'\in H$, $b,b'\in B$ and $X\in\mathfrak{X}_B$.
Then $(\mathfrak{X}_{\#_\sigma},\mathscr{L}^{\#_\sigma})$ is a right $H$-covariant FOVC on the crossed product algebra $B\#_\sigma H$, where
\begin{equation}\label{eq:crossed4}
	\mathfrak{X}_{\#_\sigma}:=\mathfrak{X}_B\otimes H\oplus B\otimes\mathfrak{X}_H
\end{equation}
is an $B\#_\sigma H$-bimodule via
\begin{align*}
	(b\#_\sigma h)\cdot(X^B\otimes h'+b'\otimes X^H)
	&:=b\cdot(h_1\rhd X^B)\cdot\sigma(h_2\otimes h'_1)\otimes h_3h'_2
	+b(h_1\rhd b')\sigma(h_2\otimes X^H_{-1})\otimes h_3\cdot X^H_0,\\
	(X^B\otimes h+b\otimes X^H)\cdot(b'\#_\sigma h')
	&:=X^B\cdot(h_1\rhd b')\sigma(h_2\otimes h'_1)\otimes h_3h'_2
	+b(X^H_{-2}\rhd b')\sigma(X^H_{-1}\otimes h'_1)\otimes X^H_0\cdot h'_2
\end{align*}
and
\begin{equation}\label{eq:crossed5}
	\mathscr{L}^{\#_\sigma}_{X^B\otimes h+b\otimes X^H}(b'\#_\sigma h')
	:=\mathscr{L}_{X^B}(h_1\rhd b')\sigma(h_2\otimes h'_1)\#_\sigma h_3h'_2
	+b(X^H_{-2}\rhd b')\sigma(X^H_{-1}\otimes h'_1)\#_\sigma\mathscr{L}_{X^H_0}(h'_2).
\end{equation}
The right $H$-coaction on $\mathfrak{X}_{\#_\sigma}$ is given by $X^B\otimes h+b\otimes X^H\mapsto X^B\otimes h_1\otimes h_2+b\otimes X^H_0\otimes X^H_1$.

We show that the sequence
\begin{equation}\label{eq:crossed6}
	0\to\underbrace{B\otimes\mathfrak{X}_H}_{\cong(B\#_\sigma H)\otimes\mathfrak{g}_H}\to\mathfrak{X}_{\#_\sigma}\to\underbrace{\mathfrak{X}_B\otimes H}_{=(\mathfrak{X}_B\otimes 1)\cdot(B\#_\sigma H)}\to 0
\end{equation}
is exact and simply reflects the direct sum decomposition of $\mathfrak{X}_{\#_\sigma}$. For the first two arrows there is nothing to prove, since $\mathfrak{X}_H\cong H\otimes\mathfrak{g}_H$ by the fundamental theorem of Hopf modules. For the surjective arrow, note that for every $X^B\otimes h$ we have that 
$$
\psi(X^B\otimes h)
=[(X^B\otimes h_1)\cdot (h_2)^{\langle 1\rangle}](h_2)^{\langle 2\rangle}
=[X^B\otimes 1]\cdot(1\#_\sigma h).
$$
Since $\mathscr{L}^B$ is injective we can identify $[X^B\otimes 1]$ with $X^B\otimes 1$ (while all $b\otimes X^H$ are clearly in the kernel of the quotient, since $\mathscr{L}^{\#_\sigma}_{b\otimes X^H}(b'\otimes 1)=0$ for all $b'\in B$). Thus $\psi$ is surjective and $B\otimes\mathfrak{X}_H$ is precisely the kernel of $\psi$. We summarize the previous discussion.
\begin{proposition}
	Let $B\#_\sigma H$ be a crossed product algebra, $\mathfrak{X}_H$ a bicovariant FOVC on $H$ and $\mathfrak{X}_B$ a FOVC on $B$ such that \eqref{eq:crossed1}, \eqref{eq:crossed2} and \eqref{eq:crossed3} hold. Then $(\mathfrak{X}_{\#_\sigma},\mathscr{L}^{\#_\sigma})$ defined in \eqref{eq:crossed4} and \eqref{eq:crossed5} is a right $H$-covariant FOVC on $A=B\#_\sigma H$. It further defines a QPVC with corresponding short exact sequence \eqref{eq:crossed6}.
\end{proposition}

\section{The sheaf approach to first order vector field calculi}\label{sec:sheaf}

Generalizing the affine treatment we saw in the previous sections,
in this section we introduce a sheaf-theoretic approach to FOVCi (Section \ref{sec:sheafFOVC}), vertical vector fields (Section \ref{sec:sheafVertical}) and then the analogues of base vector fields and of
the Atiyah sequence in this setting (Section \ref{sec:sheafAtiyah}). These are accompanied by examples built on the ringed space $\mathbb{P}^1(\mathbb{C})$ and the quantum ringed spaces obtained by localizations of $\mathcal{O}_q(\mathrm{SL}_2(\mathbb{C}))$ and $\mathcal{O}_q(\mathrm{GL}_2(\mathbb{C}))$, respectively. 

We work over a {\it quantum ringed space} $(M,\cO_M)$, that is a topological space $M$ with a sheaf $\cO_M$ of noncommutative algebras on it. Restriction maps are algebra morphisms denoted by $r_{V,U}\colon\mathcal{O}_M(U)\to\mathcal{O}_M(V)$ for open sets $V\subseteq U\subseteq M$.
This follows the philosophy expressed in the classical setting in \cite[3, Section 16 (page 28)]{egaIV}. 

Recall that, given a quantum ringed space $(M,\mathcal{O}_M)$ and a sheaf $\mathfrak{X}$ of $\Bbbk$-vector spaces, one constructs the sheaf $\mathrm{Hom}(\mathfrak{X},\mathcal{O}_M)$ with 
$$
\mathrm{Hom}(\mathfrak{X},\mathcal{O}_M)(U)=\mathrm{Hom}(\mathfrak{X}|_U,\mathcal{O}_M|_U).
$$ 
Note that $\phi_U\in\mathrm{Hom}(\mathfrak{X}|_U,\mathcal{O}_M|_U)$ denotes a family of $\Bbbk$-linear maps $\phi_V\colon\mathfrak{X}(V)\to\mathcal{O}_M(V)$ for all $V\subseteq U$ such that $r_{V,W}\circ\phi_W=\phi_V\circ r'_{V,W}$ for all $V\subseteq W\subseteq U$.
The restriction maps of $\mathrm{Hom}(\mathfrak{X},\mathcal{O}_M)$ for $V\subseteq U$, $R_{V,U}\colon\mathrm{Hom}(\mathfrak{X},\mathcal{O}_M)(U)\to\mathrm{Hom}(\mathfrak{X},\mathcal{O}_M)(V)$, are given by
$$
R_{V,U}(\phi_U):=\phi_V\in\mathrm{Hom}(\mathfrak{X}|_V,\mathcal{O}_M|_V).
$$
If $\mathfrak{X}$ is even a sheaf of $\mathcal{O}_M$-bimodules, then $\mathrm{Hom}_{\mathcal{O}_M}(\mathfrak{X},\mathcal{O}_M)$ is the sheaf with
$$
\mathrm{Hom}_{\mathcal{O}_M}(\mathfrak{X},\mathcal{O}_M)(U)=\{\phi_U\in\mathrm{Hom}(\mathfrak{X},\mathcal{O}_M)(U)~|~\phi_V\colon\mathfrak{X}(V)\to\mathcal{O}_M(V)\text{ is right $\mathcal{O}_M(V)$-linear }\forall V\subseteq U\}
$$
and with restriction morphisms induced from $\mathrm{Hom}(\mathfrak{X},\mathcal{O}_M)(U)$.

Given any sheaf $\mathcal{F}$ we denote its stalk at $p\in M$ by $\mathcal{F}_p$ and given any sheaf morphism $\varphi\colon\mathcal{F}\to\mathcal{G}$ we denote the induced morphism on stalks by $\varphi_p\colon\mathcal{F}_p\to\mathcal{G}_p$.

In \cite{aflw} the first order differential calculus approach is adapted to the sheaf-theoretic setup. Let us briefly recall from \cite[Definition 4.5]{aflw} that a \textit{first order differential calculus} (FODC) on a quantum ringed space $(M,\mathcal{O}_M)$ is an $\mathcal{O}_M$-bimodule sheaf $\Upsilon$, together with a morphism of sheaves $\mathrm{d}\colon\mathcal{O}_M\to\Upsilon$ such that the Leibniz rule
$$
\mathrm{d}_p(fg)=\mathrm{d}_p(f)g+f\mathrm{d}_p(g)
$$
and surjectivity condition $\Gamma_p=\mathcal{O}_{M,p}\mathrm{d}_p\mathcal{O}_{M,p}$ are satisfied on stalks $f,g\in\mathcal{O}_{M,p}$ for all $p\in M$. If $\mathcal{O}_M$ is even a sheaf of right $H$-comodule algebras for a Hopf algebra $H$, we call a FODC $(\Upsilon,\mathrm{d})$ on $\mathcal{O}_M$ \textit{right $H$-covariant} if $\Upsilon$ is a sheaf of right $H$-covariant $\mathcal{O}_M$-bimodules and $\mathrm{d}$ is a morphism of sheaves of right $H$-comodules.

It is the goal of the following sections to develop the corresponding notions of (covariant) first order vector field calculi.

\subsection{The sheaf of FOVC on quantum ringed spaces}\label{sec:sheafFOVC}

Fix a quantum ringed space $(M,\mathcal{O}_M)$. In the following we give the sheaf-theoretic analogue of a first order vector field calculus, as in Definition \ref{def:FOVC}. We adopt the same notation as in the affine setting.

\begin{definition}\label{sheaf-vf}
	We say that $(\Chi,\mathscr{L})$ is a 
	\textit{first order vector field calculus} (FOVC)
	on  a quantum ringed space $(M,\cO_M)$,
	if
	\begin{enumerate}
		\item[i.)] $\Chi$ is an $\cO_M$-bimodule sheaf on $M$.
		\item[ii.)] $\mathscr{L}:\Chi \lra \mathrm{Hom}(\cO_M,\cO_M)$
		is a morphism of sheaves 
		$$
		\mathscr{L}^U:\Chi(U) \lra \mathrm{Hom}(\cO_M|_U,\cO_M|_U), \qquad
		\mathscr{L}^U_X:=\mathscr{L}^U(X)
		$$
		such that the Leibniz rule
		\begin{equation}
			\mathscr{L}^p_X(fg)=\mathscr{L}^p_X(f)g+\mathscr{L}^p_{X\cdot f}(g),
		\end{equation}
		holds on stalks $f,g\in \cO_{M,p}$, $X \in \Chi_p$ for all $p\in M$.
		\item[iii.)] $\mathscr{L}$ is injective and left $\mathcal{O}_M$-linear, i.e., $\mathscr{L}^U_{f\cdot X}(g)=f\mathscr{L}^U_X(g)$ for all $f,g\in\mathcal{O}_M(U)$ and $X\in\mathfrak{X}(U)$ for all opens $U\subseteq M$.
	\end{enumerate}
\end{definition}
We now focus on a special, but important case (see \cite{afl, aflw}).
We say that an open covering $\{U_i\}$ of $M$ is \textit{affine} if $\cO_M(U_i)$
is a deformation of the coordinate ring of an affine variety $U_i \subset \C^n$.
Assume we have a {\sl finite} affine
open covering $\{U_i\}$ of $M$ and consider the topology on $M$
with base $\left\{\cap_{i \in I} U_i\right\}$, where 
$I=(i_1, \dots, i_r)$ are multi-indices \cite{aflw}.
From now on we assume to be in this setting.

The duality between FOVCi and FODCi, as expressed in the affine setting
in Proposition \ref{prop:duality}, can be generalized to the sheaf-theoretic setting
in analogy with \cite[3, Section 16, Proposition 16.5.3]{egaIV}.

\begin{proposition}
Let $(M,\cO_M)$ be a quantum ringed space 
with topology induced by an affine finite open cover as above.
\begin{enumerate}
	\item[i.)] Given a FODC $(\Upsilon,\mathrm{d})$ on $\mathcal{O}_M$ there is a FOVC $(\mathfrak{X},\mathscr{L})$
	on $\cO_M$ defined by $\mathfrak{X}:=\mathrm{Hom}_{\cO_M}(\Upsilon,\cO_M)$
	with Lie derivative $\mathscr{L}\colon\mathfrak{X} \lra \mathrm{Hom}(\cO_M,\cO_M)$,
	given by $\mathscr{L}_X= X \circ\mathrm{d}$.
	
	\item[ii.)] Given a FOVC $(\mathfrak{X},\mathscr{L})$ on $\mathcal{O}_M$ there is a FODC
	$(\Upsilon,\mathrm{d})$ on $\mathcal{O}_M$ defined by
	$\Upsilon:={}_{\cO_M}\mathrm{Hom}(\mathfrak{X},\cO_M)$ with differential $\mathrm{d}\colon \cO_M\to\Upsilon$
	determined by $X \circ \mathrm{d}=\mathscr{L}_X$ for all $X\in\mathfrak{X}$.
\end{enumerate}
\end{proposition}

\begin{proof}
For $i.)$, fix a point $p\in M$ and consider the maximal intersection $U_I\ni p$ with corresponding multi-index $I=(i_1,\ldots,i_r)$. As argued in \cite[Section 4.1]{aflw}, it follows that the stalk $\mathcal{F}_p$ of a sheaf $\mathcal{F}$ at $p$ coincides with $\mathcal{F}(U_I)$ since the considered topology is finite and the intersection $U_I$ is maximal.
Let $(\Upsilon,\mathrm{d})$ be a FODC on $\mathcal{O}_M$. 
Define $\mathfrak{X}$ as the sheafification of the presheaf 
$\mathfrak{X}(U):=\mathrm{Hom}_{\mathcal{O}_M(U)}(\Upsilon(U),\mathcal{O}_M(U))$ with $\mathscr{L}^{U}\colon\mathfrak{X}(U)\to\mathrm{End}_\Bbbk(\mathcal{O}_M(U))$ defined by $\mathscr{L}^U_X(f):=\langle X,\mathrm{d}_U(f)\rangle$.
Then by Proposition \ref{prop:duality}, this is a FOVC on the algebra $\mathcal{O}_M(U)$ as
one can readily check the Leibniz rule on the stalks $U_I$, inherited by $\mathrm{d}_{U_I}$.

The proof of $ii.)$ follows similarly.
\end{proof}

We now give an example of FOVC in the sheaf-theoretic setting for the case of the quantum ringed space $(\bP^1, \cO_{\bP^1})$.

\begin{example}
Let us consider the quantum ringed space $\bP^1= (\bP^1(\C),\cO_{\bP^1})$
($\bP^1(\C)$ denoting the topological space,
for the notation see also \cite[Section 2]{afl}).
Consider the usual affine open cover
$\{U_1,U_2\}$ of $\bP^1(\C)$  
with algebras of functions
$$
\cO_{\bP^1}(U_1):= 
\C[u], \qquad \cO_{\bP^1}(U_2):= 
\C[v]~, \qquad \cO_{\bP^1}(U_1 \cap U_2):=\C[u,u^{-1}]
$$
and restriction maps:
$$
\cO_{\bP^1}(U_1) \lra \cO_{\bP^1}(U_1 \cap U_2), \quad u \mapsto u, \qquad
\cO_{\bP^1}(U_2) \lra \cO_{\bP^1}(U_1 \cap U_2), \quad v\mapsto  u^{-1}. \qquad
$$
We now define the $\cO_{\bP^1}$-bimodule sheaf of vector fields $\Chi_{\bP^1}$ on $\bP^1$
as follows:
\beq
\begin{array}{rl}
	\Chi_{\bP^1}(U_1)&=\mathrm{span}_\C\{p(u)\partial_u, \, p(u) \in \cO_{\bP^1}(U_1)\}, \\ \\
	\quad \Chi_{\bP^1}(U_2)&=\mathrm{span}_\C\{p(v)\partial_v, \, p(v) \in \cO_{\bP^1}(U_2)\}, \\ \\
	\quad \Chi_{\bP^1}(U_1 \cap U_2)&=\mathrm{span}_\C\{p(u,u^{-1})\partial_u, \, p(u,u^{-1}) \in \cO_{\bP^1}(U_1 \cap U_2)\}
\end{array}
\eeq
with restriction maps:
$$
\begin{array}{ccc}
	\Chi_{\bP^1}(U_1) & \lra &\Chi_{\bP^1}(U_1 \cap U_2)\\
	p(u)\partial_u & \mapsto &\quad p(u)\partial_u\\
\end{array},
\qquad
\begin{array}{ccc}
	\Chi_{\bP^1}(U_2)  & \lra  & \Chi_{\bP^1} (U_1 \cap U_2), \\
	p(v)\partial_v  & \mapsto  &- u^2p(u^{-1})\partial_u.
\end{array}
$$
As for its bimodule structure, we have the left multiplication by $\cO_{\bP^1}(U_i)$ ($i=1,2$)
elements defined as usual:
$$
u \cdot \partial_u = u\partial_u, \qquad v \cdot \partial_v = v\partial_v.
$$
For the right multiplication, 
we define for a complex number $q\neq 0$ and not a root of unity
$$
\partial_u \cdot u = q^{-2} u \partial_u, \qquad \partial_v \cdot v = q^{2} v \partial_v.
$$
Note that this FOVC is dual to the FODC in \cite[Proposition 4.19]{aflw}.

As one can readily check all properties are verified. We notice that
the sheaf of quantum vector fields on $\bP^1(\C)$ is a quantum version of $\cO(2)$
(Serre's twisting sheaf \cite[Chapter 2]{ha}).
\end{example}

\subsection{FOVC on comodule algebras}\label{sec:sheafVertical}

We now define FOVCi on sheaves of comodule algebras. This is done by generalizing the affine model of Definition \ref{def:covFOVC} .

\begin{definition}
Let $H$ be a Hopf algebra and consider a sheaf of right $H$-comodule algebras $\mathcal{F}$ on $M$. 
We call a FOVC $(\mathfrak{X},\mathscr{L})$ on the quantum ringed space of $(M,\mathcal{F})$ \textit{right $H$-covariant} if $\mathfrak{X}$ is a sheaf of right $H$-covariant $\mathcal{F}$-bimodules and the induced sheaf morphism
\begin{equation}
	\begin{split}
		\widetilde{\mathscr{L}}\colon\mathfrak{X}\otimes\mathcal{F}\to\mathcal{F},\qquad\qquad\widetilde{\mathscr{L}}(U)\colon(\mathfrak{X}\otimes\mathcal{F})(U)&\to\mathcal{F}(U)\\
		X\otimes f&\mapsto\mathscr{L}^U_X(f)
	\end{split}
\end{equation}
is a morphism of sheaves of right $H$-comodules.
\end{definition}

The first class of examples of right $H$-covariant FOVCi we consider is constructed as vertical vector fields with respect to a fixed bicovariant quantum tangent space on $H$ and a sheaf $\mathcal{F}$ of right $H$-comodule algebras which are principal comodule algebras (faithfully flat Hopf--Galois extensions) on an open cover of the topology. The latter is also called \textit{quantum principal bundle} in \cite{afl} and we briefly recap the corresponding definition.
\begin{definition}\label{hg-shf} 
Let $(M,\cO_M)$ be a quantum ringed space and $H$ a Hopf algebra. 
We say that $\cF$ 
is a \textit{quantum principal bundle} over $(M,\cO_M)$
if:
\begin{itemize}
	\item $\cF$ is a sheaf of $H$ comodule algebras;
	\item There exists an open covering 
	$\{U_i\}$ of $M$ such that:
	\begin{enumerate}
		\item $\cF(U_i)^{\mathrm{co} H}=\cO_M(U_i)$,
		\item 
		$\cF (U_i )$ is a Hopf-Galois 
		extension of 
		$\cF(U_i)^{\mathrm{co} H}$.
	\end{enumerate}
\end{itemize}
\end{definition}

Let $\mathcal{F}$ be a sheaf of right $H$-comodule algebras with respect to the topology induced by a finite open covering $\{U_i\}$. As usual, we denote by $U_I=\cap_{i\in I}U_i$ the intersection for a multi-index $I=(i_1,\ldots,i_r)$. For a given bicovariant quantum tangent space $\mathfrak{g}\subseteq (H^\circ)^+$ (recall Definition \ref{def:bicovQTS}) we define a sheaf of right $H$-covariant $\mathcal{F}$-bimodules $\mathfrak{X}$ and a morphism of sheaves $\mathscr{L}\colon\mathfrak{X}\to\mathrm{Hom}(\mathcal{F},\mathcal{F})$ by
\begin{equation}\label{eq:sheafverticalVF}
\Chi(U_I):=\cF(U_I) \otimes \fg, \qquad \mathscr{L}^{U_I}:\Chi(U_I) \to  \mathrm{End}_\Bbbk(\cF(U_I)),
\quad \mathscr{L}^{U_I}_{f \otimes X}(g):=fg_0X(g_1), 
\end{equation}
where $f,g \in \cF(U_I)$ and $X \in \fg$.
\begin{proposition}\label{prop:sheafqpb}
Given a Hopf algebra $H$, a bicovariant quantum tangent space $\mathfrak{g}\subseteq (H^\circ)^+$ and
a quantum principal bundle $\mathcal{F}$ over a quantum ringed space $(M,\cO_M)$
with respect to the topology induced from a finite open cover as above, the tuple $(\mathfrak{X},\mathscr{L})$ determined by \eqref{eq:sheafverticalVF} is a right $H$-covariant FOVC on $\mathcal{F}$. Explicitly, the right $H$-coaction and $\mathcal{F}$-bimodule structure are defined locally by
\begin{equation}
	\begin{split}
		\Delta_{\mathfrak{X}(U_I)}\colon\mathfrak{X}(U_I)&\to\mathfrak{X}(U_I)\otimes H,\\
		f\otimes X&\mapsto f_0\otimes X_0\otimes f_1X_1,
	\end{split}\qquad\qquad f\cdot(g\otimes X)\cdot k:=fgk_0\otimes X\leftharpoonup k_1.
\end{equation}
for all $f,g,k\in\mathcal{F}(U_I)$ and $X\in\mathfrak{g}$.
\end{proposition}
\begin{proof}
The proof goes for the base $U_I$ for the open sets as in Proposition \ref{prop:hopf}.
By sheafification
we get the natural transformations defined for all open sets, thus proving the proposition.
\end{proof}

We call the pair $(\Chi,\mathscr{L})$ as in Proposition \ref{prop:sheafqpb} a \textit{vertical FOVC} on the quantum principal bundle $\cF$.
We now give an example extending Example \ref{ex-loc-p}.

\begin{example}\label{ex:vert-sheaf}
Consider the algebra $A=\cO_q(\rSL_2)$ with respect to $H=\cO_q(P)$, as in Examples \ref{ex:SL2:FOVC}, \ref{ex-loc-p}. 
The ideal $(\gamma)\subset A$ generated by $\gamma$ is a Hopf ideal and thus the quotient map 
is a Hopf algebra surjection. 
Moreover, both $\alpha, \gamma \in A$ are Ore elements such that the Ore localizations $A[\alpha^{-1}]$, $A[\gamma^{-1}]$
become right $H$-comodule algebras with respect to the right $H$-coaction
\begin{align*}
	\begin{pmatrix}
		\alpha & \beta\\
		\gamma & \delta
	\end{pmatrix}&\mapsto\begin{pmatrix}
		\alpha & \beta\\
		\gamma & \delta
	\end{pmatrix}\otimes\begin{pmatrix}
		t & p\\
		0 & t^{-1}
	\end{pmatrix}
	\\
	\alpha^{-1}&\mapsto\alpha^{-1}\otimes t^{-1}\\
	\gamma^{-1}&\mapsto\gamma^{-1}\otimes t^{-1}
\end{align*}
This allows us to define a sheaf of right $H$-comodule algebras:
$$
\cF(V_1):=A[\alpha^{-1}], \qquad \cF(V_2):=A[\gamma^{-1}], \qquad \cF(V_1 \cap V_2):=A[\alpha^{-1}, \gamma^{-1}],
$$
where 
$$
V_1:=\left\{ \begin{pmatrix}
	a & b\\
	c & d
\end{pmatrix} \in M_2(\Bbbk) \, \bigg| \, a \neq 0\right\}, \qquad V_2:=\left\{ \begin{pmatrix}
	a & b\\
	c & d
\end{pmatrix}\in M_2(\Bbbk)\,\bigg| \, c \neq 0\right\}, 
$$
and $M_2(\Bbbk)$ denotes the $2 \times 2$ matrices with coefficients in $\Bbbk$.
Furthermore, $\mathcal{F}$ is a quantum principal bundle, as in Definition \ref{hg-shf}, since $\mathcal{F}(V_1),\mathcal{F}(V_2),\mathcal{F}(V_1\cap V_2)$ are faithfully flat Hopf--Galois extensions (even trivial extensions) according to \cite[Section 4.5.1]{aflw}. 
Define
$\mathfrak{p}:=\mathfrak{g}_H=\mathrm{span}\{\underline{\chi_1},\underline{\chi_3},\underline{\chi_4}\}$,
where $\chi_i$ were defined in (\ref{chis}) and $\underline{\chi_i}$ denotes the induced element on the quotient, as in Proposition \ref{prop:quotientFOVC}. 
Then, by Proposition \ref{prop:sheafqpb},  we have that
$$ 
\Chi(V_I):=\cF(V_I) \otimes \fg, \qquad \mathscr{L}^{V_I}:\Chi(V_I) \to  \mathrm{End}_\Bbbk(\cF(V_I)),
\quad \mathscr{L}^{V_I}_{f \otimes X}(g):=fg_0X(g_1), \qquad V_I \in \{V_1, V_2, V_1 \cap V_2\}
$$
defines a FOVC on the quantum ringed space $(\rSL_2(\Bbbk), \cF)$ (see \ref{eq:sheafverticalVF}).
\end{example}

\subsection{Base FOVC and the Atiyah sequence}\label{sec:sheafAtiyah}

{

We continue by introducing principal first order vector field calculi on sheaves, extending the affine picture that was previously developed in Section \ref{base-vf-sec} and Section \ref{sec:Atiyah}. 
Let $M$ be a topological space.

\begin{definition}
	Let $\mathcal{F}$ be a sheaf of right $H$-comodule algebras on $M$ and $(\mathfrak{X},\mathscr{L})$ a right $H$-covariant FOVC on $\mathcal{F}$. We define the \textit{base vector fields} as the quotient sheaf 
	\begin{equation}\label{eq:basevf}
		\mathfrak{X}_M
		:=\frac{\{X\in\mathfrak{X}~|~\mathscr{L}_X\in\mathrm{Hom}(\mathcal{F}^{\mathrm{co}H},\mathcal{F}^{\mathrm{co}H})\}}{\{X\in\mathfrak{X}~|~\mathscr{L}_X(\mathcal{F}^{\mathrm{co}H})=0\}}.
	\end{equation}
\end{definition}
\begin{lemma}
	The sheaf $\mathfrak{X}_M$ defined in \eqref{eq:basevf} is a FOVC on $\mathcal{F}^{\mathrm{co}H}$ with Lie derivative $\mathscr{L}^M\colon\mathfrak{X}_M\to\mathrm{Hom}(\mathcal{F}^{\mathrm{co}H},\mathcal{F}^{\mathrm{co}H})$ induced from $\mathscr{L}\colon\mathfrak{X}\to\mathrm{Hom}(\mathcal{F},\mathcal{F})$.
\end{lemma}
\begin{proof}
	This sheaf is well-defined as quotient sheaf and we obtain the result by applying Proposition \ref{prop:base} on its stalk values.
\end{proof}

We are ready to define principal FOVC on quantum principal bundles, i.e., covariant FOVC which make the Atiyah sequence exact.

\begin{definition}
	If $H$ is a Hopf algebra and $\mathcal{F}$ is a quantum principal bundle over a quantum ringed space $(\mathcal{O}_M,M)$, we call a right $H$-covariant FOVC $(\mathfrak{X},\mathscr{L})$ on $\mathcal{F}$, together with a bicovariant quantum tangent space $\mathfrak{g}$ on $H$, a \textit{principal first order vector field calculus} (PFOVC for short), if for all $p\in M$ the sequence of stalks
	\begin{equation}
		0\to\mathcal{F}_p\otimes\mathfrak{g}\xrightarrow{\phi_p}\mathfrak{X}_p\xrightarrow{\psi_p}(\mathfrak{X}_M\mathcal{F})_p\to 0
	\end{equation}
	is well-defined and exact.
\end{definition}

In more detail, $\phi_p$ and $\psi_p$ are the maps on stalks as defined in Section \ref{sec:Atiyah}.

We give a quantum version of the sheaf of vector fields on the principal bundle
$\mathrm{GL}_2(\C) \lra \bP^1(\C)$. This example is particularly enlightening for the sheaf-theoretic approach we consider in this paper since it obeys all axioms of a principal first order vector field calculus.

\begin{example} {\bf Principal vector field calculus on $\mathrm{GL}_q(2)$.}

	Given $q\in\mathbb{C}$ not zero and not a root of unity, let us consider $A=\mathcal{O}_q(\mathrm{GL}_2(\mathbb{C}))$, namely the free $\mathbb{C}$-algebra generated by $\alpha,\beta,\gamma,\delta,r$ modulo the following relations
	\begin{align*}
		\beta\alpha
		=q\alpha\beta,\qquad
		\gamma\alpha
		=q\alpha\gamma,\qquad
		\delta\beta
		=q\beta\delta,\qquad
		\delta\gamma
		=q\gamma\delta
	\end{align*}
	\begin{align*}
		\gamma\beta
		=\beta\gamma,\qquad
		\delta\alpha-\alpha\delta
		=(q-q^{-1})\beta\gamma,
	\end{align*}
	where $r$ is defined to be a central element such that 
	$$
	r\mathrm{det}_q=1,\qquad\qquad
	\text{for}\quad\mathrm{det}_q:=\alpha\delta-q^{-1}\beta\gamma\,.
	$$
	The algebra $A$ is a Hopf algebra with the usual matrix comultiplication and counit, with $r$ being group-like, and with antipode
	$$
	S\begin{pmatrix}
		\alpha & \beta\\
		\gamma & \delta
	\end{pmatrix}=r\begin{pmatrix}
		\delta & -q\beta\\
		-q^{-1}\gamma & \alpha
	\end{pmatrix},\qquad
	S(r)=\mathrm{det}_q.
	$$
	The parabolic $H=\mathcal{O}_q(P_\mathrm{GL})=A/\langle\gamma\rangle$ is obtained as the Hopf algebra quotient on generators defined above
	$$
	\pi\colon\begin{pmatrix}
		\alpha & \beta\\
		\gamma & \delta
	\end{pmatrix}\to\begin{pmatrix}
		t & p\\
		0 & s
	\end{pmatrix}.
	$$		
	Given the usual affine open cover
	$\{U_1,U_2\}$ of $\bP^1(\C)$ we consider the rough topology  $(\bP^1(\C),\emptyset, U_1,U_2,U_{12}:=U_1\cap U_2)$. We construct then the sheaf of $H$-comodule algebras $\cF$, defined on the base for the topology $(U_1,U_2,U_{12})$ by
	$$
	\cF(U_1):=A_1=A[\alpha^{-1}]\,,\,\,\,\cF(U_2):=A_2=A[\gamma^{-1}]\,,\,\,\,\cF(U_{12}):=A_{12}:=A[\alpha^{-1},\gamma^{-1}].
	$$
	The sheaf $\mathcal{F}$ is a quantum principal bundle, i.e., faithfully flat Hopf--Galois on the open cover of its topology (see e.g. \cite[Section 4.5.3]{aflw}).
	On the algebra $A$ we consider the quantum tangent space $\mathfrak{g}_A$ given by $\mathrm{span}_\mathbb{C}\{\chi_1,\chi_+,\chi_-,\chi_2\}$ with the right $A$-module structure determined on generators by the following relations: 
	\begin{equation*}
		\begin{split}
			\chi_1\leftharpoonup \alpha
			&=q^{-2} \alpha \chi_1 +(q^{-1}-q) \beta\chi_-+(q^{-2}-1)^2\alpha 
			\chi_2\\
			\chi_1\leftharpoonup \beta
			&=\beta\chi_1+q^{-2}(q^{-1}-q)\alpha\chi+\\
			\chi_1\leftharpoonup \gamma
			&=q^{2} \gamma \chi_1 +(q^{-1}-q) \delta\chi_-+(q^{-2}-1)^2\gamma 
			\chi_2\\
			\chi_1\leftharpoonup \delta&=\delta \chi_1+q^{-2}(q^{-1}-q)\gamma\chi_+\\
			\chi_-\leftharpoonup \alpha
			&=q^{-1} \alpha \chi_-\\
			\chi_-\leftharpoonup \beta
			&=q^{-1}\beta\chi_{-}+q^{-3}(q^{-1}-q)\alpha\chi_2\\
			\chi_-\leftharpoonup \gamma
			&=q^{-1} \gamma \chi_- \\
			\chi_-\leftharpoonup \delta&=q^{-1}\delta \chi_-+q^{-3}(q^{-1}-q)\gamma\chi_2
		\end{split}\qquad
		\begin{split}
			\chi_+\leftharpoonup\alpha
			&=q^{-1}\alpha\chi_++(q^{-2}-1)\beta\chi_2\\
			\chi_+\leftharpoonup \beta
			&=q^{-1}\beta \chi_+   \\
			\chi_+\leftharpoonup \gamma
			&=q^{-1}\gamma\chi_++(q^{-1}-q)\delta\chi_-\\
			\chi_+\leftharpoonup \delta
			&=q^{-1}\delta\chi_+\\
			\chi_2\leftharpoonup \alpha
			&=\alpha \chi_2\\
			\chi_2\leftharpoonup \beta
			&=q^2\beta\chi_2\\
			\chi_2\leftharpoonup \gamma
			&=\gamma \chi_2\\
			\chi_2\leftharpoonup \delta&=q^{-2}\delta \chi_2
		\end{split}
	\end{equation*}

	The corresponding Lie derivatives on the generators of $A$ are given by
	\begin{equation*}
		\begin{split}
			\mathscr{L}^A_{\chi_1}\alpha&=-q\alpha\\
			\mathscr{L}^A_{\chi_1}\beta&=(q^{-1}-q)\beta\\
			\mathscr{L}^A_{\chi_1}\gamma&=-q\gamma\\
			\mathscr{L}^A_{\chi_1}\delta&=(q^{-1}-q)\delta\\
		\end{split}\qquad
		\begin{split}
			\mathscr{L}^A_{\chi_+}\alpha&=-\beta\\
			\mathscr{L}^A_{\chi_-}\beta&=-\alpha\\
			\mathscr{L}^A_{\chi_+}\gamma&=-\delta\\
			\mathscr{L}^A_{\chi_-}\delta&=-\gamma\\
		\end{split}\qquad
		\begin{split}
			\mathscr{L}^A_{\chi_2}\alpha&=0\\
			\mathscr{L}^A_{\chi_2}\beta&=-q\beta\\
			\mathscr{L}^A_{\chi_2}\gamma&=0\\
			\mathscr{L}^A_{\chi_2}\delta&=-q\delta\\
		\end{split}
	\end{equation*}
	Note that the right module actions and the  Lie derivatives are extended to all of $A$ by requiring them to satisfy \eqref{eq:fij} and the Leibniz rule, respectively. Observe that $\mathfrak{g}_A$ as defined is dual, in the sense of \ref{prop:duality'}, to the bicovariant FODC constructed in Example 3.9 of \cite{aflw} thus is a bicovariant quantum tangent space. This can also be verified directly. Indeed, the functionals $f_i^{j}$ can be extracted from the right module action via \eqref{eq:leftharpoon}, and the corresponding quantum brackets can then be computed explicitly. Since this calculation closely parallels the one carried out in Example \ref{ex-pair1}, we omit the details and leave the verification to the reader.
	
	In order to extend this FOVC to the localized algebra $A_1$ we use the Leibniz rule and the right module action described above and we get that the only non-vanishing new relations are
	$$
	\mathscr{L}^{A_1}_{\chi_1}\alpha^{-1}=q^3\alpha^{-1}\,\,\,\,\,\,\,\,\,\,\,\, \mathscr{L}^{A_1}_{\chi_+}\alpha^{-1}=q\alpha^{-2}\beta,
	$$
	while the right module action is consistently extended by 
	\begin{equation*}
		\begin{split}
			\chi_1\leftharpoonup \alpha^{-1}&=q^2\alpha^{-1}\chi_1-q^2(q^{-1}-q)\alpha^{-2}\beta\chi_--(q^{-1}-q)^2\alpha^{-1}\chi_2,\\
			\chi_+\leftharpoonup\alpha^{-1}&=q\alpha^{-1}\chi_+-q^{-1}(q^{-1}-q)\alpha^{-2}\beta\chi_2,
		\end{split}\qquad
		\begin{split}
			\chi_-\leftharpoonup \alpha^{-1}&=q\alpha^{-1}\chi_-,\\
			\chi_2\leftharpoonup\alpha^{-1}&=\alpha^{-1}\chi_2\,.
		\end{split}
	\end{equation*}
	Similar relations hold for the algebra $A_{2}$. Note that the Lie derivative $\mathscr{L}^{A_1}\colon\mathfrak{X}_{A_{1}}:=A_1\otimes \mathfrak{g}_A\to\mathrm{End}_\Bbbk(A_1)$ is injective since $A$ is an integral domain, as well as its Ore localization $A_1$\footnote{See for example  \cite{domainLanois} for the quantized coordinate rings of matrix groups and \cite{domainRobson} for the  Ore localization.}.   This construction  induces a sheaf of FOVC by setting $\Chi_{\cF}=\cF\otimes \mathfrak{g}_A$.
	
	The sheaf of coinvariants $\cF_B$ is given by (see \cite{afl} for more details)
	$$
	\cF_B(U_1):=B_1=\C_q[\alpha^{-1}\gamma]\,,\,\,\,\cF_B(U_2):=B_2=\C_q[\gamma^{-1}\alpha]\,,\,\,\,\cF_B(U_{12}):=B_{12}:=\C_q[\alpha^{-1}\gamma,\gamma^{-1}\alpha]
	$$
	We focus now on the base calculus on $U_1$ that is $\{X\in A_1\otimes\mathfrak{g}_A~|~\mathscr{L}^A_X(B_1)\subseteq B_1\}\bigg/\{X\in\ A_1\otimes\mathfrak{g}_A~|~\mathscr{L}^A_X(B_1)=0\}$. A simple computation shows that it is indeed generated by $ r^{-1} \alpha^2 [\chi_+]$; by using similar arguments one gets then that the base calculus on $U_2$ is generated by $ r^{-1} \gamma^2 [\chi_+]$. We then construct the sheaf of base calculus $\Chi_B$ by defining it on the base for the topology chosen by the following
	$$
	\Chi_B(U_1):=B_1\otimes r^{-1} \alpha^2 [\chi_+]\,,\,\,\,\Chi_B(U_{12}):=B_{12}\otimes r^{-1} \alpha^2 [\chi_+]\,\\,\,\,\Chi_B(U_2):=B_2\otimes r^{-1} \gamma^2 [\chi_+]
	$$
	with the restriction maps simply given by
	$$
	\begin{array}{ccc}
		\Chi_{B}(U_1) & \lra &\Chi_{B}(U_1 \cap U_2)\\
		p(\alpha^{-1}\gamma)\otimes r^{-1} \alpha^2 [\chi_+] & \mapsto &\quad p(\alpha^{-1}\gamma)\otimes r^{-1} \alpha^2 [\chi_+]\\
	\end{array}
	$$
	and
	$$
	\begin{array}{ccc}
		\,\,\,\,\,\,\,\,\,\,\Chi_{B}(U_2)  & \lra  & \Chi_{B} (U_1 \cap U_2), \\
		\,\,\,\,\,\,\,\,\,p(\gamma^{-1}\alpha)\otimes r^{-1} \gamma^2 [\chi_+] & \mapsto  & p(\gamma^{-1}\alpha)\gamma^2\alpha^{-2}\otimes r^{-1} \alpha^2 [\chi_+]
	\end{array}
	$$
	where $p(\gamma^{-1}\alpha)$ and $p(\alpha^{-1}\gamma)$ are elements of $B_1$ or $B_2$ respectively.

	We focus our attention now on the the Hopf algebra $H$; by using Proposition \ref{prop:quotientFOVC} we define a bicovariant quantum tangent space $\mathfrak{g}_H=\mathrm{span}_\mathbb{C}\{\underline{\chi}_1,\underline{\chi}_-,\underline{\chi}_2\}$. 
	The Lie derivative are then easily constructed from the values on the functional on generators, and the non-vanishing ones are listed below:
	$$
	\begin{array}{rclrclrcl}
		\mathscr{L}_{\underline{\chi}_1}t&=&-qt & \mathscr{L}_{\underline{\chi}_1}s&=&\lambda s& \mathscr{L}_{\underline{\chi}_1}p&=&\lambda p\\
		\mathscr{L}_{\underline{\chi}_-}p&=&\lambda t & \mathscr{L}_{\underline{\chi}_2}s&=&-q s & \mathscr{L}_{\underline{\chi}_2}p&=&-q p
	\end{array}
	$$
	and the right $H$-module structure is
	\begin{equation*}
		\begin{split}
			\underline{\chi}_1\leftharpoonup t
			&=q^{-2}t\underline{\chi}_1+\lambda p\underline{\chi}_-,\\
			\underline{\chi}_1\leftharpoonup p
			&=p\underline{\chi}_1,\\
			\underline{\chi}_1\leftharpoonup s
			&=qs\underline{\chi}_1+q^{-2}\lambda^2s\underline{\chi}_2,
		\end{split}\qquad
		\begin{split}
			\underline{\chi}_-\leftharpoonup t
			&=q^{-1}t\underline{\chi}_-,\\
			\underline{\chi}_-\leftharpoonup p
			&=q^{-1}p\underline{\chi}_-+q^{-3}\lambda t\underline{\chi}_2,\\
			\underline{\chi}_-\leftharpoonup s
			&=q^{-1}s\underline{\chi}_-,
		\end{split}\qquad
		\begin{split}
			\underline{\chi}_2\leftharpoonup t
			&=q^2t\underline{\chi}_2,\\
			\underline{\chi}_2\leftharpoonup p
			&=p\underline{\chi}_2,\\
			\underline{\chi}_2\leftharpoonup s
			&=q^{-2}s\underline{\chi}_2,
		\end{split}
	\end{equation*}
	where $\lambda:=q^{-1}-q$.  
	
	As discussed above, in this case it is indeed possible to extend the Lie derivative on $H$ to the whole algebra $A$ by the  construction the vertical vector field calculus as follows
	$$
	\mathscr{L}:A\otimes \mathfrak{g}_H \mapsto  \mathrm{End}_\Bbbk(A),
	\quad \mathscr{L}_{a' \otimes \underline{\chi}_i}(a):=a'a_0\underline{\chi}_i(a_1)
	$$
	and a similar formula holds for the localized $H$-comodule algebras $A_1$, $A_2$ and $A_{12}$.
	
	We are ready now to analyze the Atiyah sequence. We start defining the map $\phi(1\otimes \underline{\chi}_i)=1\otimes \chi_i$ for $i=1,-,2$; it is well defined and it makes the diagram 
	\begin{equation}
		\begin{tikzcd}
			A\otimes\mathfrak{g}_H\arrow{dr}[swap]{\mathscr{L}} \arrow{r}{\phi}
			& A\otimes \mathfrak{g}_A \arrow{d}{\mathscr{L}^A}\\
			& \mathrm{End}_\Bbbk(A)
		\end{tikzcd}
	\end{equation}
	commutative. Moreover, in the spirit of the sheaf approach it naturally extends to the localizations. Setting for simplicity $U_I=(U_1,U_2,U_{12})$ we get 
	\begin{equation}
		\begin{tikzcd}
			\cF(U_I)\otimes\mathfrak{g}_H\arrow{dr}[swap]{\mathscr{L}} \arrow{r}{\phi}
			& \cF(U_I)\otimes \mathfrak{g}_A \arrow{d}{\mathscr{L}^{\cF(U_I)}}=\Chi_\cF \\
			& \mathrm{End}_\Bbbk(\cF(U_I))
		\end{tikzcd}
	\end{equation}
	We now analyze the map 
	$$\begin{array}{rcl}
		\psi\colon\mathfrak{X}_{\cF(U_I)}&\to& \mathfrak{X}_B(U_I) \cF(U_I)\\[3mm]
		X&\mapsto& [X_0\cdot(X_1)^{\langle 1\rangle}](X_1)^{\langle 2\rangle}.
	\end{array}$$
	as defined in (\ref{chidef}). Let us start with the open $U_1$. In this case one gets that
	$$
	\psi(1\otimes \chi_1)=\psi(1\otimes \chi_2)=\psi(1\otimes \chi_-)=[0]
	$$
	as expected, while
	$$
	\psi(1\otimes \chi_+)=q^{-2}[r^{-1}\alpha^2 \chi_+]  r\alpha^{-2}
	$$

	Since $r\alpha^{-2}$ is invertible, the map is obviously surjective. A similar result holds in $U_{12}$ and in $U_2$ with $\alpha$ replaced by $\gamma$. This means that the Atiyah sequence is exact on all maximal intersections of opens of the topology and thus the Atiyah sequence is exact on all stalks.
\end{example}

}
\end{document}